\documentclass[UTF8,10pt]{article}
\usepackage[left=1.91cm,right=1.91cm,top=2.54cm,bottom=2.54cm]{geometry}
\usepackage{amsfonts}
\usepackage{amsmath}
\usepackage{amssymb}
\usepackage{pgfplots}
\usepackage{tikz}
\usetikzlibrary{arrows}
\usepackage{titlesec}
\usepackage{bm}
\usepackage{appendix}
\usepackage{tikz-cd}
\usepackage{mathtools}
\usepackage{amsthm}
\usepackage{enumitem}
\setenumerate[1]{itemsep=0pt,partopsep=0pt,parsep=\parskip,topsep=5pt}
\setitemize[1]{itemsep=0pt,partopsep=0pt,parsep=\parskip,topsep=5pt}
\setdescription{itemsep=0pt,partopsep=0pt,parsep=\parskip,topsep=5pt}
\usepackage[colorlinks,linkcolor=black,anchorcolor=black,citecolor=black]{hyperref}
\usepackage{setspace}
\usepackage[numbers]{natbib}
\usepackage{cases}
\usepackage{fancyhdr}
\usepackage{txfonts}

\newtheorem{Theorem}{Theorem}[section]
\newtheorem{Proposition}{Proposition}[section]
\newtheorem{Lemma}{Lemma}[section]
\newtheorem{Corollary}{Corollary}[section]
\newtheorem{Remark}{Remark}[section]
\newtheorem{Definition}{Definition}[section]

\numberwithin{equation}{section}

\numberwithin{equation}{section}
\usepackage{xcolor}

\begin{document}

\bibliographystyle{plain}
\title{\textbf{
Global Schr\"odinger map flows to K\"ahler manifolds with small data in critical Sobolev spaces: Energy critical case.  }}

\author{Ze Li\thanks{School of Mathematics, School of Mathematics and Statistics, Ningbo University, Ningbo, 315211, Zhejiang, P.R. China. Email: \texttt{rikudosennin@163.com}} }
\date{ }
\maketitle



\begin{abstract}
In this paper and the companion work \cite{LIZE2}, we prove that the Schr\"odinger map flows from $\Bbb R^d$ with $d\ge 2$ to compact K\"ahler manifolds with small initial data in critical Sobolev spaces are global.  The main difficulty compared with the constant sectional curvature case is that the gauged equation now is not self-contained due to the curvature part. Our main idea is to use a novel bootstrap-iteration scheme to reduce the gauged equation to an approximate constant curvature system in finite times of iteration.
This paper with the companion work \cite{LIZE2} solves the open problem raised by Tataru.
\end{abstract}

\setcounter{tocdepth}{1}
 \tableofcontents

\section{Introduction}

Let $(\mathcal{N},J,h)$ be a K\"ahler manifold, the Schr\"odinger map flow (SMF) on Euclidean spaces is a map $u: (t,x)\in \Bbb R\times\Bbb R^d\mapsto \mathcal{N}$ which satisfies
\begin{align}\label{1}
\begin{cases}
u_t =J(\sum^{d}_{i=1}\nabla_i\partial_iu)\\
u\upharpoonright_{t=0} = u_0,
\end{cases}
\end{align}
where $\nabla$ denotes the induced covariant derivative in the pullback bundle $u^{*}T\mathcal{N}$.

Assume that $\mathcal{N}$ is isometrically embedded into $\Bbb R^N$, then  (\ref{1}) can be formulated as
\begin{align}\label{11}
\begin{cases}
u_t =JP^{\mathcal{N}}_u(\Delta_{\Bbb R^d}u)\\
u\upharpoonright_{t=0} = u_0,
\end{cases}
\end{align}
where $P^{\mathcal{N}}_u$ denotes the orthogonal projection from $\Bbb R^N$ onto $T_u\mathcal{N}$.

(\ref{1}) plays a fundamental role in solid-state physics and is usually referred as the Landau-Lifshitz flow in physics literature. The various forms of SMF are commonly used in micromagnetics to  model the effects of a magnetic filed on ferromagnetic materials (e.g. \cite{LL}).  In the $d=1$ and $d=2$ case  with $\mathcal{N}=\Bbb S^2$,  SMF is referred as the ferromagnetic chain equation and the continuous isotropic Heisenberg spin model respectively (e.g. \cite{ZGT}).

The Schr\"odinger map flow can be viewed as a Hamiltonian flow on infinite dimensional symplectic manifolds, see Ding \cite{Ding}. One of the conservation law of SMF is its energy defined by
\begin{align*}
E(u)=\frac{1}{2}\int_{\Bbb R^d}|\partial_x u|^2dx.
\end{align*}
And SMF has the scaling invariance property: $u(t,x)\longmapsto u(\lambda^2 t,\lambda x)$. Thus $d=2$ is the energy critical case. In the case $\mathcal{N}=\Bbb S^2$, SMF has masses as another conservation law:
\begin{align*}
M(u)=\frac{1}{2}\int_{\Bbb R^d}|u-P|^2dx,\mbox{  } {\rm if}\mbox{  }\|u_0-P\|_{L^2_x}<\infty, \mbox{  }{\rm  for} \mbox{  }{\rm some}\mbox{ }P\in \Bbb S^2.
\end{align*}
However, the mass does not conserve  for general target $\mathcal{N}$.

We recall the following non-exhaustive list of works on Cauchy problems. The local well-posedness theory of Schr\"odinger map flows was developed by Sulem-Sulem-Bardos \cite{SSB}, Ding-Wang \cite{DW}, McGahagan \cite{Mc}. The global well-posedness theory was started by Chang-Shatah-Uhlenbeck \cite{CSU}  and Nahmod-Stefanov-Uhlenbeck \cite{NSU}.  And the $d=1$ case with general targets was studied by Rodnianski-Rubinstein-Staffilani \cite{RRS}. The global existence for small data in critical Besov spaces was proved by  Ionescu-Kenig \cite{IK} and Bejenaru \cite{B} independently. The small data global well-posedness theory in critical Sobolev spaces was done by Bejenaru-Ionescu-Kenig \cite{BIK} for high dimensions $d\ge 4$.  The two dimension case, which is energy critical, was studied by Bejenaru-Ionescu-Kenig-Tataru \cite{BIKT} where the global well-posedness theory for small data in critical Sobolev spaces was established for $\mathcal{N}=\mathbb{S}^2$ with $d\ge 2$. And Dodson-Smith \cite{DS} studied the conditional global regularity problem for $d=2$.

The stationary solutions of SMF are harmonic maps. So far, the dynamical behavior of SMF near harmonic maps is partly known in the equivariant case with $d=2$, $\mathcal{N}=\Bbb S^2$. The works of Gustafson, Kang, Tsai, Nakanish  \cite{GKT2,GNT} proved asymptotic stability v.s. wind oscillating near harmonic maps in high equivariant classes. Bejenaru-Tataru \cite{BT} proved global stability v.s. instability of harmonic maps for 1-equivariant 2D SMF.  The type II blowup solutions were constructed by Merle-Raphael-Rodnianski \cite{MPR} and Perelman \cite{P} for 1-equivariant 2D SMF. And the below threshold conjecture was verified for equivariant SMF from $\Bbb R^2$ into $\Bbb S^2$ or $\Bbb H^2$ by Bejenaru-Ionescu-Kenig-Tataru \cite{BIKT2, BIKT3}.

All the mentioned global well-posedness results of SMF with $d\ge 2$ are for targets of $\Bbb S^2$ or $\Bbb H^2$. Tataru raised the proof of  small data global well-posedness in critical Sobolev spaces for general compact K\"ahler targets as an open problem
in the survey report \cite{Koch}. This work, which deals with  the energy critical case $d=2$, together with our companion work  \cite{LIZE2} solves this problem.

\subsection{Main results}
Before stating our main results, we introduce some notations on working spaces.
For geometric PDEs, it is convenient to work in both intrinsic Sobolev spaces and extrinsic Sobolev spaces. For smooth maps from $\Bbb R^2\to \mathcal{N}$  the intrinsic norms are defined by
\begin{align*}
\|u\|^p_{\mathcal{W}^{k,p}}:=\sum^{k}_{j=1}\|\nabla^{j} u\|^p_{L^p_x(\Bbb R^d)},
\end{align*}
where $\nabla$ denotes the induced covariant derivative in $u^*T\mathcal{N}$.

Given a point $Q\in \mathcal{N}$, we define the extrinsic Sobolev space $H^{k}_Q$ by
\begin{align*}
{ {H}^{k }_Q}:=\{u:\Bbb R^d\to \Bbb R^N\mid u(x)\in \mathcal{N} a.e. {\mbox{  }}{\rm in}\mbox{  } \Bbb R^d, \| u-Q\|_{H^{k}(\Bbb R^d)}<\infty\},
\end{align*}
which is equipped with the metric $d_{Q}(f,g)=\|f-g\|_{H^k}$. Define $\mathcal{H}_Q$ to be
\begin{align*}
{ {\mathcal{H}}_Q}:=\bigcap^{\infty}_{k=1}H^{k}_Q.
\end{align*}
Our main results are the following two.
\begin{Theorem}  \label{ThR}
Let $d=2$,  $\mathcal{N}$ be a 2n-dimensional compact K\"ahler manifold which is isometrically embedded into $\Bbb R^N$, and let $Q\in\mathcal{N}$ be a given point.
 There exists a sufficiently small constant $\epsilon_* >0$ such that if $u_0\in \mathcal{H}_{Q}$  satisfies
\begin{align}
\|\partial_x u_0\|_{L^2_x}\le \epsilon_*,
\end{align}
then (\ref{1}) with initial data $u_0$ evolves into a global unique solution $u\in  C(\Bbb R;\mathcal{H}_{Q})$. Moreover, as $|t|\to \infty$  the solution $u$ converges to the constant map $Q$ in the sense that
\begin{align}\label{yuhgvb90}
\lim\limits_{|t|\to \infty}\| u(t) -Q\|_{L^{\infty}_x}=0.
\end{align}
 Furthermore, in the energy space, we also have
\begin{align}\label{y0}
\lim\limits_{t\to \infty}\|  u(t) -\sum^{n}_{j=1}\Re ( e^{it\Delta}h^{j}_{+}) -\sum^{n}_{j=1}\Im (e^{it\Delta}g^j_{+})\|_{{\dot H}^{1}_x}=0,
\end{align}
for some functions $h^j_+,g^j_+:\Bbb R^2\to \Bbb C^{N}$ belonging  to $\dot {H}^1$ with $j=1,...,n$.
\end{Theorem}

\begin{Remark}
The asymptotic behaviors  (\ref{yuhgvb90}) and (\ref{y0}) are new for SMF.
The analogous result of (\ref{yuhgvb90}) for wave maps was obtained in  Part VII of Tao \cite{T4}.
Similar result  like (\ref{y0})  was recently obtained by the author \cite{Lihy} in the setting of SMF on hyperbolic planes. One can see (\ref{y0}) is natural by checking the trivial target $\mathcal{N}=\Bbb R^{2n}$, see  Remark 1.1 of \cite{Lihy} for instance.
\end{Remark}

We also prove the uniform bounds and well-posedness results analogous to that of \cite{BIKT}.
\begin{Theorem}\label{Two}
Let $d=2$, $\sigma_1\ge 0$. Let $\mathcal{N}$ be a compact K\"ahler manifold which is isometrically embedded into $\Bbb R^N$, and let $Q\in\mathcal{N}$ be a given point.
 There exists a sufficiently small constant $\epsilon_{\sigma_1} >0$ depending only on $\sigma_1$ such that the global solution $u=S_{Q}(t)u_0\in  C(\Bbb R;\mathcal{H}_{Q})$  constructed in Theorem  \ref{ThR}   satisfies the uniform bounds
\begin{align}\label{XcDfg}
\sup_{t\in \Bbb R} \| u(t)-Q \|_{H^{\sigma+1}_x}\le C_{\sigma}(\|u_0-Q\|_{H^{\sigma+1}_x}), \mbox{ }\forall \sigma\in [0,\sigma_1].
\end{align}
In addition, for any $\sigma\in[0,\sigma_1]$, the operator $S_{Q}$ admits a continuous extension
\begin{align*}
S_{Q}:\mathfrak{B}^{\sigma}_{\epsilon_{\sigma_1}}\to    C(\Bbb R; {H}^{\sigma+1}),
\end{align*}
where we denote
\begin{align*}
\mathfrak{B}^{\sigma}_{\epsilon}:=\{f\in H^{\sigma+1}_Q: \|f-Q\|_{\dot{H}^1}\le \epsilon\}.
\end{align*}
\end{Theorem}

\begin{Remark}
 Theorem \ref{ThR} holds as well for $d\ge 3$. This will be proved in our companion work \cite{LIZE2}.
 The main proof of higher dimensions uses ideas of this work  and some additional ingredients in heat flows. We will explain this issue at the end of introduction.
\end{Remark}

\subsection{Caloric gauge and heat flows}

For dispersive geometric PDEs, especially for critical problems, it is important to choose suitable gauges and function spaces adapted to the structure of nonlinearities (e.g. null structure).
Most of these tools were developed in the study of wave map equations, see for instance \cite{KM,KS,NSU1,Tao1,Tao2,Krieger,Tataru1,Tataru2}.
In this work, we will use Tao's caloric gauge and function spaces developed by \cite{BIKT,IK1}.
As observed by  \cite{BIKT,Tao8}, the caloric gauge is essential for eliminating bad frequency interactions in dimension two compared with Coulomb gauge.
For the convenience of statement, we briefly recall the definition of caloric gauge.

First, let's recall the moving frame dependent quantities and some identities related to them, see \cite{SZ}  and \cite{RRS} for more expositions.
Let Greek indices run in $\{1,...,n\}$. Let Roman indices run in $\{1,...,2n\}$ or $\{1,...,d\}$ according to the context. Denote $\overline{\beta}=\beta+n$ for $\beta\in \{1,...,n\}$.

Let $\mathcal{N}$ be a 2n-dimensional compact K\"ahler  manifold. Since $\Bbb R^2\times[-T,T]$  is contractible, there must exist global orthonormal frames for $u^*(T\mathcal{N})$. Using the complex structure one can assume the orthonormal frames are of the form
\begin{align}
\mathbf{E}:=\{e_1(t,x),Je_1(t,x),....,e_n(t,x),Je_n(t,x)\}.
\end{align}
Let $\psi_i=(\psi^1_i,\psi^{\bar{1}}_i,...,\psi^n_i,\psi^{\overline{n}}_i)$ for $i=0,1,2$ be the components of $\partial_{t,x}u$ in the frame $\mathbf{E}$:
\begin{align}
\psi_i^\alpha = \left\langle {\partial _iu,{e_\alpha}} \right\rangle ,\psi^{\overline{\alpha}}_i = \left\langle {{\partial_i}u,J{e_\alpha}} \right\rangle.
\end{align}
We always use $0$ to represent $t$ in lower index. The isomorphism of $\Bbb R^{2n}$ to ${\Bbb C}^n$ induces a ${\Bbb C}^n$-valued function defined by $\phi^\beta_i=\psi_i^\beta+\sqrt{-1}\psi^{\overline{\beta}}_i$ with $\beta=1,...,n$. Conversely, given function $\phi:[-T,T]\times \Bbb R^2\to {\Bbb C}^n$, we associate it with a section $\phi{\bf E}$  of the bundle $u^*(T\mathcal{N})$ via
\begin{align}
\phi\Longleftrightarrow\phi {\bf E}:= \Re (\phi^\beta)e_\beta+\Im(\phi^{{\beta}})Je_\beta,
\end{align}
where $(\phi^1,...,\phi^n)$ denotes the components of $\phi$.
Then $u$ induces a covariant derivative on the trivial complex vector bundle over base manifold $[-T,T]\times\Bbb R^d$ with fiber $\Bbb C^n$ defined by
\begin{align*}
D_i\varphi^\beta=\partial_i \varphi^\beta+\sum^{n}_{\alpha=1}\left([A_i]^\beta_\alpha+\sqrt{-1}[A_i]^{\overline{\beta}}_{\alpha}\right)\varphi^\alpha,
\end{align*}
where the induced connection coefficient matrices are defined by
\begin{align*}
[A_i]^p_q= \left\langle \nabla _ie_p,{e_q} \right\rangle.
\end{align*}
Schematically we write $D_i=\partial_i+A_i$.
Recall the torsion free identity and the commutator identity
\begin{align}
D_i\phi_j&=D_j\phi_i\label{pknb}\\
[D_i,D_j]\varphi&=\left(\partial_i  A_j-\partial_j A_i+[ A_i, A_j]\right)\varphi\Longleftrightarrow \mathbf{R}(\partial_iu, \partial_j u)(\varphi {\bf E}).\label{commut}
\end{align}
Schematically, we write $[D_i,D_j]=\mathcal{R}(\phi_i,\phi_j)$.
With the notations given above, (\ref{1}) can be written as
\begin{align}\label{jnk}
\phi_t=\sqrt{-1}\sum^{2}_{i=1}D_i\phi_i.
\end{align}

\cite{Paul} proved the heat flow  with initial data $u(t,x)$ below threshold energy would converge to $Q$ as $s\to\infty$ in the topology of $C([-T,T];C^{\infty}_x)$.
 Tao's caloric gauge is defined as follows:
\begin{Definition}\label{cg}
 Let  $u(t,x):[-T,T]\times \Bbb R^2\to \mathcal{N}$ be a solution of (\ref{1}) in $C([-T,T];\mathcal{H}_Q)$. For a given orthonormal frame $E^{\infty}:=\{ {e}^{\infty}_1,J {e}^{\infty}_1,..., {e}^{\infty}_n,J {e}^{\infty}_n\}$ for $ T_{Q}\mathcal{N}$, a caloric gauge is a tuple consisting of a map  $v:\Bbb R^+\times [-T,T]\times \Bbb R^2\to \mathcal{N}$ and  orthonormal frames $\mathbf{E}(v(s,t,x)):=\{ {e}_1,J {e}_1,...,{e}_n,J {e}_n\}$ such that
\begin{align}\label{Heat}
\left\{ \begin{array}{l}
{\partial _s}v= \sum^{2}_{i=1}\nabla_i\partial_iv \\
v(0,t,x)= u(t,x) \\
\end{array} \right.
\end{align}
and
\begin{align}\label{aq2}
\left\{ \begin{array}{l}
{\nabla _s}{{e} _k} = 0,\mbox{ } k=1,...,n.\\
\mathop {\lim }\limits_{s \to \infty } {{e}_k} = {  {e}^{\infty}_k}. \\
\end{array} \right.
\end{align}
\end{Definition}

Denote
\begin{align*}
\mathcal{H}_Q(T):=C([-T,T];\mathcal{H}_Q).
\end{align*}

\begin{Proposition}\label{Caloric}
Let  $u\in \mathcal{H}_Q(T)$ solve  SMF with $u_0\in  \mathcal{H}_Q$. For any fixed frame $E^{\infty}:=\{{e}^{\infty}_k,J{e}^{\infty}_k\}^n_{k=1}$ for  $ T_{Q}\mathcal{N}$, there exists a unique corresponding caloric gauge defined in Definition \ref{cg}.
Moreover, we have for $i=1,2$ and $p,q=1,...,2n$
\begin{align*}
&\mathop {\lim }\limits_{s \to \infty } [{A_i}]^q_p(s,t,x) =0\\
&\mathop {\lim }\limits_{s \to \infty } {[A_t]^{q}_p}(s,t,x) = 0.
\end{align*}
Particularly, we have for $i=1,2$, $s>0$,
\begin{align*}
&[{A_i}]^p_q(s,t,x)= -\int^\infty_s\langle \mathbf{R} \left( {\partial _s}v(\widetilde{s})\right),{\partial _i}v(\widetilde{s}) e_{p},e_{q}\rangle  d\widetilde{s}\\
&[{A_t}]^{p}_{q}(s,t,x)=-\int^{\infty}_s \langle \mathbf{R} \left(\partial_s v(\widetilde{s}), \partial_tv(\widetilde{s})\right) e_{p}, e_{q}\rangle d\widetilde{s}.
\end{align*}
\end{Proposition}
\begin{proof}
The proof is standard (see e.g. \cite{Paul}). The only new issue here is the complex structure $J$. But this will not cause any trouble since $J$ commutes with $\nabla_s$.
\end{proof}

Given  $u\in \mathcal{H}_Q(T)$ which solves  (\ref{1}), let  $v:\Bbb R^+\times [-T,T]\times \Bbb R^2\to \mathcal{N}$ be the solution to (\ref{Heat}).  Let  $\{ {e}_\alpha,J {e}_\alpha\}^n_{\alpha=1}$ be the corresponding caloric gauge.
Define the heat tension field  $\phi_s$ to be
\begin{align*}
\phi^{\alpha}_s=\langle \partial_s v,e_{\alpha}\rangle+ \sqrt{-1} \langle \partial_s v,Je_{\alpha}\rangle, \mbox{ }\alpha=1,...,n.
\end{align*}
And define the differential fields to be
\begin{align*}
\phi^{\alpha}_i=\langle \partial_i v,e_{\alpha}\rangle+ \sqrt{-1} \langle \partial_i v,Je_{\alpha}\rangle, \mbox{ }\alpha=1,...,n,
\end{align*}
where $i=1,2$ refers to the variable  $x_i$, $i=1,2$, and $i=0$ refers to the variable $t$.

\begin{Lemma}\label{asdf}
The heat tension filed $\phi_s$ satisfies
\begin{align}\label{heat1}
\phi_s=\sum^{2}_{j=1}D_j\phi_j.
\end{align}
The differential fields $\{\phi_i\}^2_{i=1}$ along the heat flow satisfy
\begin{align}\label{heat2}
\partial_s\phi_i= \sum^{2}_{j=1}D_jD_j\phi_i + \sum^{2}_{j=1}\mathcal{R}(\phi_i,\phi_j)\phi_j.
\end{align}
And when $s=0$, along the Schr\"odinger  flow direction,  $\{\phi_i\}^2_{i=1}$ satisfy
\begin{align}
-\sqrt{-1}D_t\phi_i&=  \sum^{2}_{j=1}D_jD_j\phi_i+\sum^{2}_{j=1} \mathcal{R}(\phi_i,\phi_j)\phi_j.\label{xcgfrt}
\end{align}
\end{Lemma}

\subsection{Function Spaces built by \cite{BIKT}}

We recall the spaces developed by Bejenaru-Ionescu-Kenig-Tataru \cite{BIKT}.
Given a unite vector ${\bf e}\in \Bbb S^{1}$ we denote its orthogonal complement of $\Bbb R^2$ by ${\bf e}^\bot$. The lateral space $L^{p,q}_{\bf e}$ is defined by the norm
\begin{align*}
\|f\|_{L^{p,q}_{\bf e}}=\left(\int_{\Bbb R}\left(\int_{{\bf e}^{\bot}\times\Bbb R}\left|f(t,x_1{\bf e}+x')\right|^qdx'dt\right)^{\frac{p}{q}}dx_1\right)^{\frac{1}{p}},
\end{align*}
with standard modifications when either $p=\infty$ or $q=\infty$.
And for any given $\lambda\in \Bbb R$, $W\subset \Bbb R$, we define the spaces $L^{p,q}_{{\bf e},\lambda}$, $L^{p,q}_{{\bf e},W}$ with norms
\begin{align*}
\|f\|_{L^{p,q}_{{\bf e},\lambda}}&=\|G_{\lambda{\bf e}}(f)\|_{L^{p,q}_{\bf w}}\\
\|f\|_{L^{p,q}_{{\bf e},W}}&=
\mathop {\inf }\limits_{f = \sum\limits_{\lambda  \in W} {{f_\lambda }} } \sum\limits_{\lambda  \in W} {{{\left\| {{f_\lambda }} \right\|}_{L_{{\mathbf{e}},\lambda }^{p,q}}}},
\end{align*}
where $G_{{\bf a}}$ denotes the Galilean transform:
\begin{align*}
G_{{\bf a}}(f)(t,x)=e^{-\frac{1}{2}ix\cdot {\bf a}}e^{-\frac{i}{4}t|{\bf a}|^2}f(x+t{\bf a},t).
\end{align*}
We remark that the lateral space was developed by Linares-Ponce \cite{LP}, Kenig-Ponce-Vega \cite{KPV} and Ionescu-Kenig \cite{IK}.

The main dyadic function spaces $N_k(T),F_{k}(T),G_k(T)$ are recalled as follows:
Given $T\in \Bbb R^+$, $k\in \Bbb Z$, let $I_k:=\{\eta\in \Bbb R^2: 2^{k-1}\le  |\eta|\le 2^{k+1}\}$ and
\begin{align}
L^2_k(T):=\{g\in L^2([-T,T]\times \Bbb R^2): \mathcal{F}{g}(t,\eta)  \mbox{  }{\rm is} \mbox{  }{\rm supoorted} \mbox{  }{\rm in} \mbox{  }\Bbb R\times I_{k}\}.
\end{align}
Given $\mathcal{L}\in \Bbb Z_+$,  $T\in(0,2^{2\mathcal{L}}]$, $k\in\Bbb Z$,  define
\begin{align}
W_k=\{\lambda\in[-2^{2k},2^{2k}]:2^{k+2\mathcal{L}}\lambda\in \Bbb Z\}.
\end{align}
The $N_k(T),F_{k}(T),G_k(T)$ spaces are Banach spaces of functions in $L^2_k(T)$ for which the associated norms are finite:
 \begin{align*}
\|g\|_{F^0_k(T)}&:=\|g\|_{L^{\infty}_tL^2_x}+2^{-\frac{k}{2}}\|g\|_{L^{4}_xL^{\infty}_t}+\|g\|_{L^4}+2^{-\frac{k}{2}}\sup_{{\bf e}\in \Bbb S^{1}}\|g\|_{L^{2,\infty}_{{\bf e},W_{k+40}}}\\
\|g\|_{F_k(T)}&:=
\mathop {\inf }\limits_{j\in {\mathbb{Z}_ + },{n_1},...,{n_j} \in \Bbb N} \mathop {\inf }\limits_{g = {g_{{n_1}}} + .... + {g_{{n_j}}}} \sum\limits_{l = 1}^j {{2^{{n_l}}}} {\left\| {{g_{{n_l}}}} \right\|_{F_{k + {n_l}}^0}}
\\
\|g\|_{G_k(T)}&:=\|g\|_{F^0_k}+2^{-\frac{k}{6}}\sup_{{\bf e}\in \Bbb S^1}\|g\|_{L^{3,6}_{\bf e}}+
2^{\frac{k}{6}}\sup_{|k-j|\le 20}\sup_{{\bf e}\in \Bbb S^{1}}\|P_{j,{\bf e}}g\|_{L^{6,3}_{\bf e}}\\
&+2^{\frac{k}{2}}\sup_{|k-j|\le 20}\sup_{{\bf e}\in \Bbb S^{1}}\sup_{|\lambda|<2^{k-40}}  \|P_{j,{\bf e}}g\|_{L^{\infty,2}_{{\bf e},\lambda}}\\
\|g\|_{N_k(T)}&:=\inf_{g=g_1+g_2+g_3+g_4}\|g_1\|_{L^{\frac{4}{3}}}+2^{\frac{k}{6}}\|g_2\|_{L^{\frac{3}{2},\frac{6}{5}}
_{{{\bf e}_1}}}+2^{\frac{k}{6}}\|g_3\|_{L^{\frac{3}{2},\frac{6}{5}}
_{{{\bf e}_2}}}\\
&+2^{-\frac{k}{2}}\sup_{{\bf e}\in \Bbb S^{1}}\|g_4\|_{L^{1,2}_{{\bf e}, W_{k-40}}},
\end{align*}
where $\{{\bf e}_1,{\bf e}_2\}\subset \Bbb S^1$ consist of the standard basis of $\Bbb R^2$.
The $G_{k},F_{k}$ spaces were built by
Bejenaru-Ionescu-Kenig-Tataru \cite{BIKT}.

\subsection{Overview of the proof}

\underline{{\it Main difficulty for general targets.}}
The new difficulty arising in the case of general targets is to control the curvature terms in frequency localized spaces. Since the curvature term relates with the map itself, it cannot be written in a self-closed form of differential fields and heat tension fields $\{\phi_{x},\phi_s\}$. Thus directly working with the moving frame dependent quantities may lose control of curvature terms, which is much serious when frequency interactions are considered. In the wave map setting, the general targets case was solved by Tataru \cite{Ta7} using Tao's micro-local gauge and Tataru's function spaces. It is important that the wave map equation is semilinear in the extrinsic form,
and the micro-local gauge adapts to the extrinsic formulation well. However, for SMF, on one side, since the extrinsic form equation is quasilinear one has to use the intrinsic formulation to obtain a semilinear equation. On the other side, the intrinsic form is not a self-contained system where the curvature term is not determined by differential fields. The two conflict sides make solving SMF for general targets challenging.

\underline{{\it  Outline of proof for $d=2$.}}
Let us sketch the outline of proof in the $d=2$ case. The whole proof is divided into ten steps. Given $\delta>0$, let $\{a_k\}_{k\in \Bbb Z}$ be a positive sequence, we call it a frequency envelope of order $\delta$ if
\begin{align*}
\sum_{k\in\Bbb Z} &a^2_k   <\infty ,     {\mbox{  }}{\rm and}\mbox{  } a_{j}  \le 2^{\delta|l-j|}a_{l},   \mbox{  }\forall {j,l\in \Bbb Z}.
\end{align*}
We call the frequency envelope $\{a_k\}$ an $\epsilon$-envelope if it additionally satisfies
\begin{align*}
\sum_{k\in\Bbb Z}a^2_k\le \epsilon^2.
\end{align*}

{\bf Step 1. Tracking $L^{\infty}_tL^2_x $ bounds along the heat flow direction.}
Recall the extrinsic formulations of heat flows:
Assume that the target manifold $\mathcal{N}$ is isometrically embedded into $\Bbb R^{N}$, then the heat flow equation can be formulated as
\begin{align}\label{f1OOT}
\partial_s v^{l}-\Delta v^{l}=\sum^2_{a=1}\sum^N_{i,j=1}S^{l}_{ij}\partial_{a}v^{i}\partial_{a}v^{j},\mbox{ }l=1,...,N,
\end{align}
where $S=\{S^{l}_{ij}\}$ denotes the second fundamental form of the embedding $\mathcal{N}\hookrightarrow \Bbb R^N$.
For $u\in \mathcal{H}_{Q}(T)$, define
\begin{align*}
\gamma_k(\sigma)=\sup_{k'\in\Bbb Z}2^{-\delta|k-k'|}2^{\sigma k'+k'}\|P_{k'}u\|_{L^{\infty}_tL^2_{x}}, \mbox{  }\sigma\ge 0, \mbox{  }\delta=\frac{1}{800}.
\end{align*}
Denote $\{\gamma_k\}$ the frequency envelope for the energy norm, i.e. $
\gamma_k=\gamma_k(0)$.
The first result of  Step 1 is stated for $\sigma\in[0,\frac{5}{4}]$:
\begin{Proposition}\label{888a}
Assume that $u\in \mathcal{H}_Q(T)$  satisfies
\begin{align}\label{fBOOT}
\|\partial_xu\|_{L^{\infty}_tL^2_{x}}=\epsilon_1\ll1.
\end{align}
And let $v(s,t,x)$ be the solution of heat flow (\ref{f1OOT}) with initial data $u(t,x)$. Then $v$ satisfies
\begin{align*}
\sup_{s\ge 0}(1+s2^{2k})^{31}2^{k}\|P_{k }v\|_{L^{\infty}_tL^2_{x}}\lesssim  2^{-\sigma k} \gamma_k(\sigma)
\end{align*}
for all $\sigma\in[0,\frac{99}{100}]$, $k\in\Bbb Z$.
Moreover, for any $\sigma\in[\frac{99}{100},\frac{5}{4}]$,  $k\in\Bbb Z$, we have
\begin{align*}
\sup_{s\ge 0}(1+s 2^{2k})^{30}2^{\sigma k+k}\|P_{k }v\|_{L^{\infty}_tL^2_{x}}\lesssim_{M} \gamma_k(\sigma)+\gamma_k(\sigma-\frac{3}{8})\gamma_k(\frac{3}{8}).
\end{align*}
\end{Proposition}
\begin{Remark} The power of $(1+s2^{2k})$ in  Proposition \ref{888a} can be chosen to be any $M\in \Bbb Z_+$  with additionally assuming that $\epsilon_1$ is sufficiently small depending on $M$. See Proposition \ref{fJth}  below.
\end{Remark}

The second result of this step  is  bounds of $2^{(\sigma+1)k}\|P_{k}v\|_{L^{\infty}_tL^2_x}$  for $\sigma\in[0,\frac{j+1}{4}]$:
\begin{Proposition}($j$-th iteration)  \label{fJth}
Let $j\in \Bbb N, M\in\Bbb Z_+ $. Assume that $u\in \mathcal{H}_Q(T)$  satisfies (\ref{fBOOT}) with $\epsilon_1$ sufficiently small depending on $j+M$.
Let $v(s,t,x)$ be the solution of heat flow (\ref{f1OOT}) with initial data $u(t,x)$. Then for $\sigma\in[0,1+\frac{j}{4}]$ and any $k\in\Bbb Z$, $v$ satisfies
 \begin{align}\label{fj4JHB}
\sup_{s\in[0,\infty)}(1+s2^{2k})^{M}2^{k+\sigma k}\|P_{k}v\|_{L^{\infty}_tL^2_{x}}\lesssim \gamma^{(j)}_k(\sigma).
\end{align}
where $ \{\gamma^{(j)}_k(\sigma)\}$ are defined in (\ref{fd1})-(\ref{fd2}).
\end{Proposition}

{\bf Step 2. Pretreat curvature terms.} The curvature part in master equation (\ref{xcgfrt}) can be schematically written as
\begin{align*}
 \Re[\mathcal{R}(\phi_i,\phi_j)\phi_j]^{\alpha}=\sum_{1\le j_0,j_1,j_2\le 2n}\langle {\bf R}(e_{j_0},e_{j_1})e_{j_2},e_{\alpha}\rangle  \psi^{i_0}_i\psi^{i_1}_j \psi^{i_2}_j\\
\Im[\mathcal{R}(\phi_i,\phi_j)\phi_j]^{\alpha}=\sum_{1\le j_0,j_1,j_2\le 2n}\langle {\bf R}(e_{j_0},e_{j_1})e_{j_2},e_{\overline{\alpha}}\rangle  \psi^{i_0}_i\psi^{i_1}_j \psi^{i_2}_j.
\end{align*}
With abuse of notations, we denote
\begin{align*}
\mathcal{G}=\langle {\bf R}(e_{j_0},e_{j_1})e_{j_2},e_{j_3}\rangle,
\end{align*}
for any given index $j_0,...,j_3\in\{1,...,2n\}$.
And let $\phi_i\diamond\phi_j$ denote the linear combinations of  multiplications of real and imaginary parts of $\phi_i$, $\phi_j$, i.e. $\sum_{ij} c_{ij}\phi^{\pm}_i\phi^{\pm}_{j}$, where we denote $\phi^{+}_j=\Re \phi_j,\phi^{-}_j={\Im} \phi_j$. Then the master equation (\ref{xcgfrt}) is now  schematically written as
\begin{align}
-\sqrt{-1}D_t\phi_i&=  \sum^{2}_{j=1}D_jD_j\phi_i+\sum^{2}_{j=1} \mathcal{G}\phi_i\diamond\phi_j\diamond\phi_j.\label{1xcgfrt}
\end{align}
Meanwhile, the connection coefficients in $D_j$ also depend on curvatures and can be schematically written as
\begin{align*}
[A_j]^{p}_{q}(s)=-\int^{\infty}_s\phi_s\diamond \phi_j\mathcal{G}ds'.
\end{align*}

We shall perform dynamical separation to $\mathcal{G}$. In fact, by caloric condition and  twice dynamic separation, $\mathcal{G}$ can be decomposed into
\begin{align*}
\mathcal{G}(s)&= \langle{\bf R}(e_{j_0},e_{j_1})e_{j_2},e_{j_3}\rangle(s)\\
&=\Gamma^{\infty}-\Gamma^{\infty,(1)}_l\int^{\infty}_s\psi^{l}_s(\widetilde{s})d\widetilde{s}
-\int^{\infty}_s \psi^{l}_s(\widetilde{s})\left(\int^{\infty}_{\widetilde{s}}\psi^{p}_s(s') ( \widetilde{\nabla}^2{\bf R})(e_{l},e_{p};e_{j_0},...,e_{j_3} ) ds'\right)d\widetilde{s}.\\
&:=\Gamma^{\infty}+\mathcal{U}_{00}+\mathcal{U}_{01}+\mathcal{U}_{I}+\mathcal{U}_{II},\nonumber
\end{align*}
where we denote
\begin{align*}
&{\rm\bf constant \mbox{ }limit \mbox{ }part} \mbox{ }\mbox{ }\Gamma^{\infty}:=\lim\limits_{s\to\infty}\mathcal{G}(s);\mbox{  } \mbox{ }\mbox{ } \Gamma^{\infty,(1)}_{l}:=\lim\limits_{s\to\infty}(\widetilde{\nabla}{\bf R})( {e_l};e_{j_0},...,e_{j_3})\\
&{\rm\bf one \mbox{ }order \mbox{ }terms} \mbox{ }\mbox{ } \mbox{ }\mbox{ }\mbox{  }\mathcal{U}_{00}:=-\Gamma^{\infty}_l\int^{\infty}_s\sum^{2}_{i=1}(\partial_i\psi_i)^{l}ds',
\end{align*}
and {\bf quadratic terms} by
 \begin{align*}
\mathcal{U}_{01}&:=-\int^{\infty}_s\sum^{2}_{i=1}(\partial_i\psi_i)\left(  (\widetilde{\nabla}{\bf R})( {e_l};e_{j_0},...,e_{j_3})-\Gamma^{\infty,(1)}_l\right)ds'\\
\mathcal{U}_{I}&:=-\Gamma^{\infty,(1)}_l\int^{\infty}_s\sum^{2}_{i=1}(A_i\psi_i)^{l}ds'\\
\mathcal{U}_{II}&:=-\int^{\infty}_s\sum^{2}_{i=1} (A_i\psi_i)^{l}(\widetilde{s})\left(\int^{\infty}_{\widetilde{s}}\psi^{p}_s(s')  ( \widetilde{\nabla}^2 {\bf R})(e_{l},e_{p};e_{j_0},...,e_{j_3} )ds'\right)d\widetilde{s}\nonumber\\
&=-\int^{\infty}_s\sum^{2}_{i=1}(A_i\psi_i)^{l}(\widetilde{s})\left(  (\widetilde{\nabla}{\bf R})({e_l};e_{j_0},...,e_{j_3}) -\Gamma^{\infty,(1)}_l\right)d\widetilde{s}.
\end{align*}
Here, with abuse of notations, $A_j$ denotes the $2n\times2n$ real-valued matrix with elements $\{[A_j]^{p}_{q}\}^{2n}_{p,q=1}$.
The constant  limit  part and the one order terms will be dominated by frequency  envelopes of $\{\phi_i\}$, the bounds of quadratic terms essentially rely on a delicate bootstrap on the term:
{\begin{align*}
\widetilde{\mathcal{G}}^{(1)}_{l}&:= (\widetilde{\nabla}{\bf R})({e_l};e_{j_0},...,e_{j_3}) -\Gamma^{\infty,(1)}_{l}.
\end{align*}}

We also need higher order derivatives of $\mathcal{G}$. Given $k\in\Bbb N$, let
 \begin{align*}
{\mathcal{G}}^{(k)}_{l_1,...,l_k}&:=(\widetilde{\nabla}^{(k)}{\bf R})(e_{l_1},...,e_{l_{k}};e_{j_0},...,e_{j_3}) \\
\widetilde{\mathcal{G}}^{(k)}_{l_1,...,l_k}&:= {\mathcal{G}}^{(k)}_{l_1,...,l_k} -\Gamma^{\infty,(k)}_{l_1,...,l_{k}},
\end{align*}
where we denote
{\begin{align*}
\Gamma^{\infty,(k)}_{l_1,...,l_{k}}&:= \lim\limits_{s\to\infty}{\mathcal{G}}^{(k)}_{l_1,...,l_k}(s).
\end{align*}}

Similarly, we perform  dynamical separation to frames. In fact, let $\mathcal{P}$ be the isometric embedding of $\mathcal{N}$ into $\Bbb R^{N}$, and let $\{e_{l}\}^{2n}_{l=1}$ be the caloric gauge built in Prop. 1.1. With abuse of notations, we denote
 \begin{align*}
[d\mathcal{P}]^{(k)}&:=({\bf D}^kd\mathcal{P})(\underbrace{e ,...,e }_{k};e) \\
[\widetilde{d\mathcal{P}}]^{(k)}&:=[d\mathcal{P}]^{(k)}-\lim\limits_{s\to\infty}[d\mathcal{P}]^{(k)}.
\end{align*}

{\bf Step 3. Tracking $L^4\cap L^{\infty}_tL^2_x$ bounds for curvature terms and frames along heat direction.}
\begin{Proposition}\label{f7.2q}
Let $d=2$. And let $u \in\mathcal{H}_Q(T)$ be a solution of SMF.
Denote $v(s,t,x)$ the solution to heat flow with initial  data $u(t,x)$, and denote $\{\phi_i\}^{2}_{i=0}$ the corresponding  differential fields under the caloric gauge.
Assume that $\{\beta_{k}(\sigma)\}$ is a frequency envelope of order $\delta$ such that for all $i=1,2$, $k\in\Bbb Z$,
\begin{align}\label{favNMl0mnm}
2^{\sigma k}\|\phi_i(\upharpoonright_{s=0})\|_{L^{\infty}_tL^{2}_x\cap L^{4}_{t,x}}\le \beta_k(\sigma).
\end{align}
\begin{itemize}
  \item There exists a sufficiently small constant $\epsilon>0$ such that if
\begin{align}\label{fbvNMl0mnm}
\sum_{k\in\Bbb Z} |\beta_k(0)|^2<\epsilon,
\end{align}
then we have for any $m\in\Bbb N$, any $\sigma\in[0,\frac{99}{100}]$, $s\in [2^{2j-1}, 2^{2j+1})$,  $j,k\in\Bbb Z$,
 \begin{align*}
 \|P_{k}\widetilde{\mathcal{G}}^{(m)}\|_{L^4\cap L^{\infty}_tL^2_x}&\lesssim_m 2^{-\sigma k-k} \beta_{k}(\sigma)(2+s2^{2k})^{-30} \\
 \|P_{k}\phi_s \|_{L^4\cap L^{\infty}_tL^2_x}&\lesssim 2^{-\sigma k+k}\left[1_{k+j\ge 0}(1+s2^{2k})^{-30}\beta_{k}(\sigma)+1_{k+j\le 0}\sum_{k\le l\le -j}\beta_{l}(\sigma)\beta_l\right] \\
 \|P_{k}([\widetilde{d\mathcal{P}}]^{(m)})\|_{L^{\infty}_tL^{2}_x\cap L^{4}}&\lesssim_{m} \beta_{k}(\sigma)(1+s2^{2k})^{-29}2^{-\sigma k-k} \\
 \|P_kA_i\|_{L^{\infty}_tL^2_x }
 &\lesssim \beta_{k,s}(\sigma)(1+s2^{2k})^{-27}2^{-\sigma k}.
\end{align*}
 \item Furthermore, given $j,M \in\Bbb Z_+$, if $\{\beta_{k}(\sigma)\}$ is a frequency envelope of order $\frac{1}{2^{j}}\delta$, then similar results hold for $\sigma\in[0,\frac{1}{4}j+1]$ and $\epsilon$ sufficiently small depending only on $j,M\in\Bbb Z_+$. Particularly, for  any $m\in\Bbb N$,  $k\in\Bbb Z$ and any  $\sigma\in[0,\frac{1}{4}j+1]$,  one has
\begin{align*}
(1+s2^{2k})^{M+2} 2^{\sigma k+k}\|P_{k}\widetilde{\mathcal{G}}^{(m)}\|_{L^4\cap L^{\infty}_tL^2_x}&\lesssim_{m,M} \beta^{(j)}_{k}(\sigma) \\
(1+s2^{2k})^{M+1}2^{\sigma k+k}\|P_{k}([\widetilde{d\mathcal{P}}]^{(m)})\|_{L^{\infty}_tL^{2}_x\cap L^{4}_{t,x}}&\lesssim_{m,M} \beta^{(j)}_{k}(\sigma) \\
(1+s2^{2k})^{M}2^{\sigma k}\|P_kA_i\|_{L^{\infty}_tL^2_x}&\lesssim_{M} \beta^{(j)}_{k,s}(\sigma).
\end{align*}
\end{itemize}
\end{Proposition}

\begin{Remark}
The $\{\beta^{(j)}_{k}\}$ and $\{\beta^{(j)}_{k,s}(\sigma)\}$ are defined as (\ref{fd1})-(\ref{fd2}).
\end{Remark}

{\bf Step 4.1. $F_{k}\bigcap S^{\frac{1}{2}}_k$ bounds for  connections along the heat direction.}
In this step, we prove
\begin{Lemma}\label{fParabolic}
Given $\sigma\in [0,\frac{99}{100}]$, let $\{h_{k}(\sigma)\}$ be frequency envelopes defined by
\begin{align}\label{hbjbmmm}
h_{k}(\sigma)=\sup_{k'\in\Bbb Z,j=1,2}2^{-\delta|k-k'|}2^{\sigma k'}(1+s2^{2k'})^{4}\|P_{k'}\phi_j\|_{F_{k'}(T)}.
\end{align}
Let $\{b_k\}$ be a $\varepsilon$-frequency envelope.
Assume that for any $k,j\in\Bbb Z$, $s\in [2^{2j-1},2^{2j+1})$, there holds
\begin{align}
2^{\frac{1}{2}k}\|P_{k }\widetilde{\mathcal{G}}^{(1)}\|_{L^4_xL^{\infty}_t(T)}&\le \varepsilon^{-\frac{1}{4}}h_k[(1+s2^{2k})^{-9}1_{j+k\ge 0}+1_{j+k\le 0}2^{\delta|k+j|}].\label{ftian}
\end{align}
Then, if $\varepsilon>0$ is sufficiently small, for $\sigma\in[0,\frac{99}{100}]$ one has
\begin{align*}
\|P_{k }A_i(s)\|_{F_k(T)\cap S^{\frac{1}{2}}_k}&\lesssim h_{k,s}(\sigma) 2^{-\sigma k} (1+s2^{2k})^{-4}.
\end{align*}
\end{Lemma}

{\bf Step 4.2. $F_{k}$ bounds along the heat direction without assuming (\ref{ftian}).}
\begin{Lemma}\label{fParabolic}
Let $\{b_k\}$ be a $\varepsilon$-frequency envelope.
Given $\sigma\in [0,\frac{99}{100}]$, suppose that  $\{b_k(\sigma)\}$ are also frequency envelopes,  and $\{h_{k}(\sigma)\}$ are the frequency envelopes defined by (\ref{hbjbmmm}).
Assume that for $i=1,2$,
\begin{align*}
\|P_{k }\phi_i\upharpoonright_{s=0}\|_{F_k(T)}&\le b_k(\sigma)2^{-\sigma k}, \mbox{  }\sigma\in [0,\frac{99}{100}] \\
\|P_{k }\phi_t\upharpoonright_{s=0}\|_{L^{4}_{t,x}}&\lesssim b_k(\sigma)2^{-(\sigma -1)k}, \mbox{  }\sigma\in [0,\frac{99}{100}] \\
\|P_{k }\phi_i(s)\|_{F_{k}(T)}&\lesssim \varepsilon^{-\frac{1}{2}}b_k (1+s2^{2k})^{-4}.
\end{align*}
Then, if $\varepsilon>0$ is sufficiently small, for $\sigma\in[0,\frac{99}{100}]$ one has
\begin{align*}
\|P_{k }A_i(s)\|_{F_k(T)\cap S^{\frac{1}{2}}_k}&\lesssim h_{k,s}(\sigma) 2^{-\sigma k} (1+s2^{2k})^{-4} \nonumber\\
\|P_{k }\phi_i(s)\|_{F_k(T)}&\lesssim b_k(\sigma) 2^{-\sigma k} (1+s2^{2k})^{-4}  \\
\|P_k A_i\upharpoonright_{s=0}\|_{L^{4}_{t,x}}&\lesssim b_k(\sigma)2^{-\sigma k}, \mbox{ }i=1,2. \\
\|P_k\phi_t(s)\|_{L^{4}_{t,x}}&\lesssim b_k(\sigma)2^{-(\sigma-1)k}(1+2^{2k}s)^{-2}\nonumber\\
\|P_kA_t\upharpoonright_{s=0}\|_{L^{2}_{t,x}}&\lesssim \varepsilon b_{k}(\sigma)2^{- \sigma k},\mbox{ }\sigma\in[\frac{1}{100},\frac{99}{100}]\\
\|P_kA_t\upharpoonright_{s=0}\|_{L^{2}_{t,x}}&\lesssim \varepsilon^2, \mbox{ } \sigma\in[0,\frac{99}{100}].\nonumber
\end{align*}
\end{Lemma}
{\bf Step 5. $G_{k}$ bounds along the SMF direction for $\sigma\in[0,\frac{99}{100}]$.}
\begin{Proposition}
Assume that  $\sigma\in [0,\frac{99}{100}]$.
  Given any $\mathcal{L}\in \Bbb Z_+$, assume that $T\in(0,2^{2\mathcal{L}}]$. Let $\epsilon_0$ be a sufficiently small constant. Assume that  $\{c_k\}$ is an $\epsilon_0$-frequency envelope of order $\delta$. And let $\{c_k(\sigma)\}$ be another frequency envelope of order $\delta$. Let $u\in \mathcal{H}_Q(T)$ be the solution to SMF with initial data $u_0$ which satisfies
\begin{align*}
\|P_{k}\nabla u_0\|_{L^2_x}&\le c_k\\
\|P_{k}\nabla u_0\|_{L^2_x}&\le c_k(\sigma)2^{-\sigma k}.
\end{align*}
Denote $\{\phi_i\}$ the corresponding differential fields of the heat flow initiated from $u$. Suppose also that at the heat initial time $s=0$,
\begin{align*}
\|P_{k}\phi_i\|_{G_k(T)}&\le \epsilon^{-\frac{1}{2}}_0c_k.
\end{align*}
Then, when $s=0$, we have for all $i=1,2$, $k\in\Bbb Z$,
\begin{align*}
\|P_{k}\phi_i\|_{G_k(T)}&\lesssim c_k \\
\|P_{k}\phi_i\|_{G_k(T)}&\lesssim c_k(\sigma)2^{-\sigma k}.
\end{align*}
\end{Proposition}

{\bf Step 6. Improve $F_{k}$ bounds of  $P_{k}\widetilde{\mathcal{G}}^{(1)}$ once.}
\begin{Lemma}
Let $u\in\mathcal{H}_Q(T)$ solve SMF with data $u_0$. Let $\{c_k\}$  be $\epsilon_0$-frequency envelopes of order $\frac{1}{2}\delta$. Given any $\sigma\in[0,\frac{99}{100}]$, let $\{c_{k}(\sigma)\}$ be another frequency envelopes of order $\delta$ such that
\begin{align*}
\|P_{k}\nabla u_0\|_{L^2_x}&\le c_k\\
\|P_{k}\nabla u_0\|_{L^2_x}&\le c_k(\sigma)2^{-\sigma k}.
\end{align*}
Then for $\epsilon_0$ sufficiently small there holds
\begin{align*}
2^{k}\|P_{k}\widetilde{ \mathcal{G}}^{(1)}\|_{F_{k}(T)}\le c_k(\sigma)2^{-\sigma k}[(1+2^{2k+2k_0})^{-20}1_{k+k_0\ge 0}+1_{k+k_0\le 0}2^{\delta|k+k_0|}],
\end{align*}
for any $\sigma\in[0,\frac{99}{100}]$, $k,k_0\in
\Bbb Z, s\in[2^{2k_0-1},2^{2k_0+1})$.
\end{Lemma}

{\bf Step 7. Improve $F_{k}\cap S^{\frac{1}{2}}_k$ bounds of  $\{P_{k}A_j\}^{2}_{j=0}$ and derive parabolic estimates for $\{\phi_j\}^2_{j=0}$  with  $\sigma\in[0,\frac{5}{4}]$ .}
\begin{Lemma}\label{ffParabolic}
Let $u\in \mathcal{H}_{Q}(T)$ be solution to SMF with initial data $u_0\in\mathcal{H}_{Q}$. Define frequency envelopes   $\{c^{(1)}(\sigma)\}$ as Def. 6.1.
Given any $\sigma\in [0,\frac{5}{4}]$,
let $\{b_k(\sigma)\}$ be frequency envelopes of order $\delta$. And assume that
\begin{align*}
b_{k}(\sigma)\lesssim c^{(1)}_{k}(\sigma),\mbox{  }\forall \sigma\in [0,\frac{99}{100}].
\end{align*}
Assume also that for $i=1,2$,
\begin{align*}
\|P_{k }\phi_i\upharpoonright_{s=0}\|_{F_k(T)}&\le b_k(\sigma)2^{-\sigma k}, \mbox{  }\sigma\in [0,\frac{5}{4}] \\
\|P_{k }\phi_t\upharpoonright_{s=0}\|_{L^{4}_{t,x}}&\lesssim b_k(\sigma)2^{-(\sigma -1)k}, \mbox{  }\sigma\in [0,\frac{5}{4}].
\end{align*}
Then, if $\varepsilon>0$ is sufficiently small, for $\sigma\in[0,\frac{5}{4}]$ one has for any $\sigma\in [0,\frac{5}{4}]$
\begin{align*}
\|P_{k }A_i(s)\|_{F_k(T)\cap S^{\frac{1}{2}}_k}&\lesssim h^{(1)}_{k,s}(\sigma) 2^{-\sigma k} (1+s2^{2k})^{-4} \nonumber\\
\|P_{k }\phi_i(s)\|_{F_k(T)}&\lesssim b^{(1)}_k(\sigma) 2^{-\sigma k} (1+s2^{2k})^{-4}  \\
\|P_k A_i\upharpoonright_{s=0}\|_{L^{4}_{t,x}}&\lesssim b^{(1)}_k(\sigma)2^{-\sigma k}, \mbox{ }i=1,2. \\
\|P_k\phi_t(s)\|_{L^{4}_{t,x}}&\lesssim b^{(1)}_k(\sigma)2^{-(\sigma-1)k}(1+2^{2k}s)^{-2}\nonumber\\
\|P_kA_t\upharpoonright_{s=0}\|_{L^{2}_{t,x}}&\lesssim \varepsilon b^{(1)}_{k}(\sigma)2^{- \sigma k},\mbox{ }\sigma\in[\frac{1}{100},\frac{5}{4}]\\
\|P_kA_t\upharpoonright_{s=0}\|_{L^{2}_{t,x}}&\lesssim \varepsilon^2, \mbox{ } \sigma\in[0,\frac{5}{4}].\nonumber
\end{align*}
\end{Lemma}

\begin{Remark}
In Lemma \ref{ffParabolic}, the $\{b^{(j)}_{k}\}$, $\{h^{(j)}_{k}\}$,$\{b^{(j)}_{k,s}\}$, $\{h^{(j)}_{k,s}\}$  are defined as Def. 6.2.
\end{Remark}

{\bf Step 8. $G_{k}$ bounds along the SMF direction for $\sigma\in[0,\frac{5}{4}]$.}
\begin{Lemma}\label{6fffParabolic}
Assume that  $\sigma\in [\frac{99}{100},\frac{5}{4}]$.
Let  $\epsilon_0>0$ be a sufficiently small constant.  Given any $\mathcal{L}\in \Bbb Z_+$, assume that $T\in(0,2^{2\mathcal{L}}]$.  Let $\{c_k(\sigma)\}$ be  frequency envelopes of order $\frac{1}{2}\delta$. And let $\{c_k\}$ be an $\epsilon_0$-frequency envelope of order $\frac{1}{2}\delta$. Let $u\in \mathcal{H}_Q(T)$ be the solution to SMF with initial data $u_0$ which satisfies
\begin{align*}
\|P_{k}\nabla u_0\|_{L^2_x}&\le c_k\\
\|P_{k}\nabla u_0\|_{L^2_x}&\le c_k(\sigma)2^{-\sigma k}.
\end{align*}
Denote $\{\phi_i\}$ the corresponding differential fields of the heat flow initiated from $u$.
Then, when $s=0$, given $\sigma\in[\frac{99}{100},\frac{5}{4}]$, one has
\begin{align*}
\|P_{k}\phi_i\upharpoonright_{s=0}\|_{G_{k}(T)}\lesssim \left(c_{k}(\sigma)+ c_{k}(\sigma -\frac{3}{8})c_{k}(\frac{3}{8})\right)2^{-\sigma k}.
\end{align*}
\end{Lemma}

{\bf Step 9. $F_{k}$ bounds of $P_{k}\widetilde{\mathcal{G}}^{(1)}$ and  $G_{k}$ bounds of $\phi_x$ along the  SMF  direction   for $\sigma\in[0,1+\frac{j}{4}]$.}
\begin{Lemma}\label{6fffParabolic}
Given $j\ge 2$, assume that  $\sigma\in [0,1+\frac{j}{4}]$.
Let $Q\in \mathcal{N}$ be a fixed point and $\epsilon_0$ be a sufficiently small constant depending on $j$.  Given any $\mathcal{L}\in \Bbb Z_+$, assume that $T\in(0,2^{2\mathcal{L}}]$.  Let $u\in \mathcal{H}_Q(T)$ be the solution to SMF with initial data $u_0$. Let $\{c^{(j)}_k(\sigma)\}$ be  frequency envelopes defined in Def. 6.1. And assume that $\{c^{(j)}_k(0)\}$ is an $\epsilon_0$-frequency envelope with $0<\epsilon_0\ll 1$.
\begin{itemize}
 \item Then, for $\sigma\in[0,1+\frac{j-1}{4}]$, we have
\begin{align*}
2^{k}\|P_{k}\widetilde{ \mathcal{G}}^{(1)}\|_{F_{k}(T)}\le c^{(j)}_k(\sigma)2^{-\sigma k}[(1+2^{2k+2k_0})^{-20}1_{k+k_0\ge 0}+1_{k+k_0\le 0}2^{\delta|k+k_0|}],
\end{align*}
for any $s\in[2^{2k_0-1},2^{2k_0+1})$, $k_0,k\in\Bbb Z$.
   \item Denote $\{\phi_i\}$ the corresponding differential fields of the heat flow initiated from $u$. Then, for $\sigma\in[0,\frac{j}{4}+1]$  there holds
\begin{align*}
\|P_{k}\phi_i\upharpoonright_{s=0}\|_{G_{k}(T)}\lesssim  c^{(j)}_k(\sigma)2^{-\sigma k}.
\end{align*}
 \end{itemize}
\end{Lemma}

{\bf Step 10. Global regularity, global well-posedness and asymptotic behaviors.}

\noindent As Step 9,  proceeding the bootstrap-iteration  scheme  for $K$ times gives  bounds of $2^{\sigma k}\|P_{k}\phi_j\|_{G_{k}}$ for $\sigma\in[0,\frac{K}{4}+1]$.
Then, transforming the bounds of $\{\phi_j\}\upharpoonright_{s=0}$ back to the solution $u$ gives
\begin{align*}
\|u\|_{L^{\infty}_t\dot {H}^{\sigma+1}_x(T)}\lesssim  \|u_0\|_{{\dot H}^{1}_x\bigcap{\dot {H}}^{\sigma+1}_x}
\end{align*}
Noticing that an $\dot {H}^{1} \bigcap \dot {H}^{2+}$ uniform bound will rule out blow-up for SMF in $\Bbb R^2$, one step iteration suffices to show $u$ is global. And the $\epsilon_0$ depends only on the dimension and the target manifold $\mathcal{N}$.
Moreover, we proceed the bootstrap-iteration  scheme  for $K$ times and obtain uniform bounds for higher Sobolev norms.

The asymptotic behaviors stated in  (\ref{y0}) will be proved following our recent work \cite{Lihy} on SMF on hyperbolic planes.

\subsection{Main ideas}

Let us explain the main ideas. To control the curvature terms, which is the non-selfclosed part, we use dynamical separation and a bootstrap-iteration  scheme to obtain an approximate
 constant sectional curvature nonlinearity with controllable remainder in finite steps of iteration. The essential  advantage of this   scheme is that it reduces estimates in frequency localized spaces such as $F_k,G_k$ to decay estimates in Lebesgue spaces along the heat direction.

{{\bf Iteration scheme for Step 1.}  }We describe the iteration scheme  for heat flows.  The starting point of heat flow iteration is bounds for $\partial_sv$.  \underline {\it 1. First time iteration.}  Suppose that we have obtained parabolic decay estimates of $\|P_{k}\partial_sv\|_{L^{\infty}_tL^2_x}$ such as
\begin{align*}
 \|P_{k}\partial_s v\|_{L^{\infty}_tL^2_x}\lesssim (1+2^{2k}s)^{-M_1}\gamma_{k,s}(\sigma)2^{-\sigma k}, \mbox{ }\sigma\in [0,\frac{99}{100}].
\end{align*}
By applying dynamical separation
\begin{align*}
 S^{l}_{ij}(v(s))=S^{l}_{ij}(Q)-\int^{\infty}_s (DS^{l}_{ij})(v(s'))\cdot\partial_sv ds'
\end{align*}
bounds for $\|P_{k}\partial_sv\|_{L^{\infty}_tL^2_x}$ yield improved frequency localized bounds for the second fundamental form term, i.e. $\|P_{k} S^{l}_{ij}(v(s))\|_{L^{\infty}_tL^2_x}$. The only potential trouble to do this is the ${\rm High\times Low}$ interaction of $(DS^{l}_{ij})(v(s'))\cdot\partial_sv$. But
we will see this interaction can be handled with additionally proving the decay estimates
\begin{align*}
\|\partial^{L+1}_x DS^{l}_{ij}(v(s))\|_{L^{\infty}_tL^2}\lesssim_{L} \epsilon_1 s^{-\frac{L}{2}}, \mbox{ }\forall\mbox{ } L\in \Bbb N.
\end{align*}
Then back to the extrinsic map $v$, using the heat flow equation will give an improved bound for $\|P_{k}\partial_sv\|_{L^{\infty}_tL^2_x}$ for $\sigma\in[1,\frac{5}{4}]$.   \underline{{\it 2. $m$-th time iteration.} } Using dynamical separation of the schematic form
\begin{align*}
 D^{m-1}S(v(s))=D^{m-1}S(Q)-\int^{\infty}_s D^{m}S(v(s'))\partial_sv ds'
\end{align*}
and the decay estimates
\begin{align*}
\|\partial^{L+1}_x D^{m-1}S^{l}_{ij}(v(s))\|_{L^{\infty}_tL^2}\lesssim_{L,m} \epsilon_1 s^{-\frac{L}{2}}, \mbox{ }\forall\mbox{ } L\in \Bbb N,
\end{align*}
one gets frequency localized  bounds of the extrinsic map $v$ for  $\sigma\in[1,1+\frac{m}{4}]$:
\begin{align*}
2^{\sigma k}\|P_{k}v(s)\|_{L^{\infty}_tL^2_x}\lesssim  \epsilon_1 (1+2^{2k}s)^{-M_1+m}\gamma^{(m)}_{k}(\sigma).
\end{align*}

{\bf The motivation of decomposition in  Step 2.}
To bound the curvature terms,  by a kind of dynamical separation we have the decomposition of curvature terms denoted by $\mathcal{G}$:
\begin{align*}
\mathcal{G}={\rm constant}{\mbox{ }} +{\rm one}{\mbox{ }}{\rm  order}{\mbox{ }}{\rm  terms}{\mbox{ }}+{\mbox{ }}{\rm quadratic}\mbox{ }{\rm terms},
\end{align*}
(see Step 2). We observe that to control $\mathcal{G}$ in the $F_{k}$ space, it suffices to prove parabolic decay estimates of $\mathcal{G}^{(j)}$. The same idea will be applied to bound the frames in frequency localized spaces. We remark that dynamical separation was previously used by \cite{Krieger,T4} to expose the implicit null structures. Here, we apply  dynamical separation to do iterations.
Besides using dynamical separation, in order to give a bound for connection coefficients which is the heart for  bootstrap, we further decompose the curvature term into differential fields $\phi_i$ dominated terms and relatively smaller quadratic terms. By an appropriate bootstrap argument, bounding  connection coefficients $\{A_j\}^2_{j=1}$ in the { $F_{k}\bigcap S^{1/2}_k$} space reduces to derive parabolic decay estimates of covariant derivatives of curvatures $\mathcal{G}^{(j)}$ in the simpler $L^{\infty}_tL^2_x\cap L^4$ spaces.

{\bf The motivation of adding (\ref{ftian}) in Step 4.1.}
The key difficulty in bounding curvature involved terms is the ${\rm High\times Low\to High}$ interaction of curvatures and differential fields or heat tension fields, i.e. the frequency of curvatures occupies the dominate position compared with differential fields or heat tension fields. First of all, we observe that it suffices to control the $F_k$ norms of curvatures $\mathcal{G}$. Then we further clarify that  among the four blocks of $F_k$ space only the three blocks  $L^{\infty}_tL^2$, $L^{4}$, $L^{4}_xL^{\infty}_t$  need to be estimated for curvatures.  Second, we find that using dynamic separation in the heat direction
  \begin{align*}
\mathcal{G}:= {\bf R}(e_{i_0},e_{i_1},e_{i_2},e_{i_3})&=\Gamma^{\infty}-\int^{\infty}_s\psi^{l}_s(\widetilde{\nabla}{\bf R})(e_l;,e_{i_0},e_{i_1},e_{i_2},e_{i_3})ds'\\
&:=\Gamma^{\infty}-\Gamma^{\infty,(1)}_{l}\int^{\infty}_s\psi^{l}_sds'-\int^{\infty}_s\psi^{l}_s\widetilde{\mathcal{G}}^{(1)}_lds'
 \end{align*}
and the heat flow iteration scheme, the  $L^{\infty}_tL^2$, $L^{4}$ norms of  curvatures can be controlled by corresponding norms of differential fields.

The troublesome block is  the $L^4_xL^{\infty}_t$ norm. This norm of curvatures  can not be obtained by dynamic separation in the heat direction and the heat flow iteration scheme as before. The problem is the ${\rm High\times Low\to High}$ interaction of  $\widetilde{\mathcal{G}}^{(1)}_l\psi^{l}_s$ in  $L^4_xL^{\infty}_t$ fails if one only previously has bounds for $\widetilde{\mathcal{G}}^{(1)}$ in  $L^{\infty}_tL^2_x\cap L^4$. (Estimates of $\widetilde{\mathcal {G}}^{(1)}$ in $L^{\infty}_tL^2_x\cap L^4$ are obtained in Step 3.) So we add a bootstrap assumption  (\ref{ftian})
to bound $L^4_xL^{\infty}_t$ of $\widetilde{\mathcal {G}}^{(1)}$.  The key is that one can indeed  improve the  assumption  (\ref{ftian}) and then close bootstrap.

{\bf How to drop  (\ref{ftian}) in Step 4.2.}
Let us explain how to improve  (\ref{ftian}) and thus close bootstrap of Step 4.  ({\bf{I}}) With assumption  (\ref{ftian}) we can prove bounds of $A_t$ and $\phi_t$ in $L^4$ by envelopes of differential fields $\phi_x$, loosely speaking say
  \begin{align*}
&\|P_{k}(A_t)\|_{L^{4}}\lesssim 2^{-\sigma k+k}h_{k}(\sigma)[(1+2^{2k+2k_0})^{-1}1_{k+k_0\ge 0}+1_{k+k_0\le 0}2^{\delta|k+k_0|}]\\
 &\|P_{k}(\phi_t)\|_{L^4}\lesssim 2^{k} b_{k}(1+2^{2k}s)^{-2}
 \end{align*}
where $k,k_0\in\Bbb Z$, $s\in[2^{2k_0-1},2^{2k_0+1})$, $\{h_{k}(\sigma)\}$ is the frequency envelope associated with differential fields $\{\phi_i\}^2_{i=1}$ defined in (\ref{hbjbmmm}), and  $\{b_k\}$ is  some frequency envelope with $\|b_{k}\|_{\ell^2}$ sufficiently small. This will give bounds for $\|P_{k}{\partial_t}\widetilde{\mathcal{G}}^{(1)}\|_{L^{4}}$.
({\bf{II}})  We improve the assumption  (\ref{ftian}) by interpolation
\begin{align}\label{Uy78}
\|P_{k}\widetilde{\mathcal{G}}^{(1)}\|_{L^{4}_xL^{\infty}_t}\le \|P_{k}\widetilde{\mathcal{G}}^{(1)}\|^{\frac{3}{4}}_{L^{4}} \|P_{k}\partial_t\widetilde{\mathcal{G}}^{(1)}\|^{\frac{1}{4}}_{L^{4}},
\end{align}
and the $L^4$ estimates obtained from Step 3 (Prop. 1.4 for heat flow iteration)
\begin{align*}
2^{k}\|P_{k}\widetilde{\mathcal{G}}^{(1)}\|_{L^{4}}\le h_{k} (1+2^{2k}s)^{-M}.
\end{align*}
In fact, we can prove
\begin{align*}
2^{\frac{1}{2}k}\|P_{k}\widetilde{\mathcal{G}}^{(1)}\|_{L^{4}_xL^{\infty}_t}\le h_{k}  2^{-\sigma k}[(1+2^{2k+2k_0})^{-\frac{3}{4}M}1_{k+k_0\ge 0}+1_{k+k_0\le 0}2^{\delta|k+k_0|}].
\end{align*}
Then choosing sufficiently large  $M$, one obtains better bounds of $\widetilde{\mathcal{G}}^{(1)}$ than the assumption  (\ref{ftian}).
  This presents the way to bound curvatures in the block space $L^4_xL^{\infty}_t $ of $F_k$.

In  ({\bf{I}}), the key step is to obtain bounds of connection coefficients, see Lemma \ref{connection}. To prove Lemma  \ref{connection}, as mentioned before,  we  decompose the curvature term into differential fields $\phi_i$ dominated terms and the remained quadratic terms, see  Step 2 of Lemma \ref{connection}. The additional smallness gain yielded by the remained quadratic terms gives us the chance to use bootstrap argument to control  connection coefficients.

{\bf  Iteration in Step 6 to Step  9.}
With these new ideas  and \cite{BIKT}'s framework, the range of $\sigma\in[0,\frac{99}{100}]$ can be reached before performing iteration for the SMF evolution. In order to reach larger $\sigma$, one combine the heat flow iteration with an SMF iteration. For the SMF iteration scheme, the key is to improve the  estimate
$\|P_{k}\widetilde{\mathcal{G}}^{(1)}\|_{L^{4}_xL^{\infty}_t}$ step by step to reach larger $\sigma$.

\subsection{Idea for higher dimensions}

Let us give a prevue of the higher dimensions.  As $d=2$, for higher dimensions in order to track the curvature terms $\mathcal{G}$ in the $F_k$ space along the heat flow direction,  it suffices to control the one order covariant derivative of curvature term $\widetilde{\mathcal{G}}^{(1)}$ in the simpler Lebesgue spaces in the heat direction.
Thus the parabolic decay estimates of moving frame dependent quantities for $d\ge 3$ should be established. The difficulty is to bound all geometric quantities in the fractional Sobolev spaces when $d$ is odd. We solve this problem by using the geodesic parallel transpose and difference characterization of Besov spaces. The idea is that the difference characterization reduces bounding fractional Sobolev norms to bound  difference of
all these geometric quantities and their covariant derivatives in Lebesgue spaces. And the geodesic parallel transpose gives us the difference of geometric quantities at different points of the base manifold.

We divide the whole Theorem for $d\ge 2$ into two papers to make the main idea clear and avoid the paper being too long.

{\bf Notations}
Let $\Bbb Z_+=\{1,2,...\}$, $\Bbb N=\{0,1,2,...\}$.
We  apply the notation $X\lesssim Y$ whenever there exists some  constant $C>0$ so that $X\le C Y$. Similarly, we will use $X\sim Y$ if $X\lesssim Y \lesssim X$.
We sometimes drop the integral variable in the integration if no confuse occurs. And we closely follow the notations of \cite{BIKT} for reader's convenience.

Let $\mathcal{F}$ denote the Fourier transformation in $\Bbb R^2$.

Let $\chi:\Bbb R\to [0,1]$ be a given smooth even function which is supported in $\{z\in\Bbb R:|z|\le \frac{8}{5}\}$ and equals to 1 for $\{z\in \Bbb R:|z|\le \frac{5}{4}\}$. Define $\chi_{k}(z)=\chi(\frac{z}{2^{k}})-\chi(\frac{z}{2^{k-1}})$, $k\in\Bbb Z$.
The Littlewood-Paley projection operators with Fourier multiplier  $\eta\mapsto \chi_{k}(|\eta|)$ are denoted by
$P_k$, $k\in \Bbb Z$. For $I\subset \Bbb \Bbb R$, let $\chi_{I}=\sum_{i\in I} \chi_{i}(|\xi|)$.  The low frequency cutoff operator with Fourier multiplier $\eta\mapsto \chi_{(-\infty,k]}(|\xi|)$ is denoted by $P_{\le k}$. And  the high frequency cutoff is defined by  $P_{\ge k}=I-P_{\le k}$. Given ${\bf e}\in\Bbb S^1$, $k\in\Bbb Z$, denote $P_{k,{\bf e}}$ the operator with Fourier multiplier $\xi\mapsto \chi_{k}(\xi\cdot {\bf e})$.

The Riemannian curvature tensor on $\mathcal{N}$ is denoted by ${\bf R}$. The covariant derivative on $\mathcal{N}$ is denoted by $\widetilde{\nabla}$. And we denote $\nabla$ the induced covariant derivative
on $u^{*}T\mathcal{N}$. The metric tensor of $\mathcal{N}$ is denoted by $\langle \cdot, \cdot\rangle$. Let $E$ be a Riemannian manifold with connection $\nabla$, and ${\Bbb T}$ be a $(0,r)$ type tensor. For $k,r\in\Bbb Z_+$, we define the $(0,r+k)$ type tensor $\nabla^{k}{\Bbb T}$ by
\begin{align*}
\nabla^{k}{\Bbb T}(X_1,...,X_k; Y_1,...,Y_r):=(\nabla_{X_k}(\nabla^{k-1}{\Bbb T}))(X_1,...,X_{k-1}; Y_1,...,Y_r)
\end{align*}
for arbitrary tangent vector fields $X_1,...,X_k,Y_1,...,Y_r$  on $E$.

\section{Preliminaries }

\subsection{ Linear estimates}

The following is the main linear estimates established  by \cite{BIKT}.
\begin{Proposition}[\cite{BIKT}]\label{bvc457}
Given $\mathcal{L}\in \Bbb Z_+$, assume that $T\in (0,2^{2\mathcal{L}}]$. Then for every $u_0\in L^2_x$ with frequency localized in $I_k$ and every $F\in N_k(T)$,
we have the inhomogeneous estimate:
If $u$ solves
\begin{align}\label{q}
  \left\{
     \begin{array}{ll}
       i\partial_t u+\Delta u=F, & \hbox{ } \\
       u(0,x)=u_0(x), & \hbox{ }
     \end{array}
   \right.
\end{align}
then
\begin{align}
\|u\|_{G_k(T)}\lesssim \|u_0\|_{L^2_x}+\|F\|_{N_k(T)}.
\end{align}
\end{Proposition}

Recall the refined space for $F_k(T)$: Let $S^{\omega}_k(T)$ denote  the normed space of functions in $L^2_k(T)$ for which
\begin{align}
\|g\|_{S^{\omega}_k(T)}=2^{k\omega}\left(\|g\|_{L^{\infty}_tL^{2_\omega}_x}+\|g\|_{L^{4}_tL^{p^*_ \omega}_x}
+2^{-\frac{k}{2}}\|g\|_{L^{p^*_{\omega}}_xL^{\infty}_t}\right)
\end{align}
is finite, where the exponents $2_{\omega} $ and $p^*_{\omega}$ are defined via
\begin{align}
\frac{1}{2_{\omega}}-\frac{1}{2}=\frac{1}{p^*_{\omega}}-\frac{1}{4}=\frac{\omega}{2}.
\end{align}

The following lemma will be used widely.
\begin{Lemma}[\cite{BIKT}]
For $f\in L^2_k(T)$, there hold
\begin{align}
\|P_kf\|_{L^4}&\le \|f\|_{F_{k}(T)}\\
\|P_kf\|_{F_{k}(T)}&\lesssim \|f\|_{L^{2}_{x}L^{\infty}_t}+\|f\|_{L^4}\\
\|P_kf\|_{L^{2}_{x}L^{\infty}_x }&\le \|f\|_{S^{\frac{1}{2}}_k},
\end{align}
and
\begin{align}
\|e^{s\Delta} g \|_{F_k(T)}&\lesssim (1+s2^{2k})^{-20}\|g\|_{F_k(T)}\label{3.18},
\end{align}
provided that the RHS is finite.
\end{Lemma}

\subsection{ Frequency Envelopes}

We recall the definition of envelopes introduced by Tao.
\begin{Definition}
Let $\{a_k\}_{k\in \Bbb Z}$ be a positive sequence, we call it a frequency envelope if
\begin{align}\label{gvbnm}
\sum_{k\in\Bbb Z} &a^2_k   <\infty ,     {\mbox{  }}{\rm and}\mbox{  } a_{j}  \le 2^{\delta|l-j|}a_{l},   \mbox{  }\forall {j,l\in \Bbb Z}.
\end{align}
\end{Definition}
We call the frequency envelope $\{a_k\}$ an $\epsilon$-envelope if it additionally satisfies
\begin{align*}
\sum_{k\in\Bbb Z}a^2_k\le \epsilon^2.
\end{align*}
For any nonnegative sequence $\{a_j\}\in \ell^2$, we define its frequency envelope by
\begin{align*}
\tilde{ a}_j:=\sup_{j'\in\Bbb Z}a_{j'}2^{-\delta|j-j'|}.
\end{align*}
And $\{\tilde{ a}_j\}$ satisfies
\begin{align*}
|a_j|\le \tilde{ a}_j, \forall j\in\Bbb Z; \sum_{j\in\Bbb Z}\tilde{a}^2_j\lesssim \sum_{j\in\Bbb Z}a^2_j.
\end{align*}
Generally the $\delta $ in Definition 2.2 is not important if it has been fixed throughout the paper. But due to our iteration argument we shall introduce different $\delta$ in different steps of iterations. So we call $\{a_k\}$ satisfying  (\ref{gvbnm})\underline{ frequency envelope of order $\delta$}.

In this paper, each time we mention frequency envelopes, we will clearly state its order.

We recall the following two facts on envelopes: (a) If $d_k\le b_k$ for all $k\in \Bbb Z$ and $\{b_k\}$ is
a frequency envelope of order $\delta>0$ then $\widetilde{d}_k\le b_k$ for all $k\in \Bbb Z$ as well, where $\{\widetilde{d}_k\}$ denotes the envelope of $\{d_k\}$ of the same order $\delta>0$:
\begin{align*}
\tilde{d}_k:=\sup_{j\in\Bbb Z}d_{j}2^{-\delta|k-j|}.
\end{align*}
(b) If $\{d_k\}$ is already an envelope of order $\delta>0$ then ${d}_k=\widetilde{d}_k$ for all $k\in \Bbb Z$.

We recall the classic result obtained by  \cite{Paul, Tao7}.
\begin{Lemma}[\cite{Paul, Tao7}]\label{kuai}
Assume that $u\in \mathcal{H}_Q(T)$  satisfies
\begin{align}\label{BOOT}
\|\partial_xu\|_{L^{\infty}_tL^2_{x}}=\epsilon_1\ll1.
\end{align}
And let $v(s,t,x)$ be the solution of heat flow (\ref{Heat}) with initial data $u(t,x)$. Then
\begin{align}\label{FD}
\|\partial^{j+1}_xv\|_{L^{\infty}_tL^2_{x}}\lesssim s^{-\frac{j}{2}}\epsilon_1,
\end{align}
and the corresponding differential fields and connection coefficients satisfy
\begin{align}
s^{\frac{j}{2}}\|\partial^{j}_x\phi_i\|_{L^{\infty}_tL^2_x}&\lesssim \epsilon_1\\
s^{\frac{j}{2}}\|\partial^{j}_xA_i\|_{L^{\infty}_{t}L^2_x}&\lesssim \epsilon_1\\
s^{\frac{j}{2}}\|\partial^{j}_x\phi_i\|_{L^{\infty}_{t}L^{\infty}_x}&\lesssim \epsilon_1\\
s^{\frac{j+1}{2}}\|\partial^{j}_xA_i\|_{L^{\infty}_tL^{\infty}_x}&\lesssim \epsilon_1,
\end{align}
for all $s\in[0,\infty)$, $i=1,2$ and any nonnegative integer $j$.
\end{Lemma}

\section{Iteration for Heat Flows}

\subsection{Main results on the extrinsic map $v$ solving HF}

For $u\in \mathcal{H}_{Q}(T)$, define
\begin{align}\label{Jkmnbn}
\gamma_k(\sigma)=\sup_{k'\in\Bbb Z}2^{-\delta|k-k'|}2^{\sigma k'+k'}\|P_{k'}u\|_{L^{\infty}_tL^2_{x}}, \mbox{  }\sigma\ge 0, \mbox{  }\delta=\frac{1}{800}.
\end{align}
Denote $\{\gamma_k\}$ the frequency envelope for the energy norm, i.e.,
\begin{align*}
\gamma_k=\gamma_k(0).
\end{align*}
Thus
\begin{align*}
2^{k}\|P_{k }u\|_{L^{\infty}_tL^2_{x}}\le 2^{-\sigma k}\gamma_k(\sigma), \mbox{  }\forall \sigma\ge 0.
\end{align*}

Before going ahead, we recall the extrinsic formulations of heat flows.
Assume that the target manifold $\mathcal{N}$ is isometrically embedded into $\Bbb R^{N}$, then the heat flow equation can be formulated as
\begin{align}\label{1OOT}
\partial_s v^{l}-\Delta v^{l}=\sum^2_{a=1}\sum^N_{i,j=1}S^{l}_{ij}\partial_{a}v^{i}\partial_{a}v^{j},\mbox{ }l=1,...,N,
\end{align}
where $S=\{S^{l}_{ij}\}$ denotes the second fundamental form of the embedding $\mathcal{N}\hookrightarrow \Bbb R^N$.

Recall that $\{\gamma_k(\sigma)\}$ are defined to be  the frequency envelopes of $u\in \mathcal{H}_Q(T)$. ( See (\ref{Jkmnbn}) and $u$ {\bf{neednot}} solve SMF.)
The frequency localized estimates of the extrinsic map $v$ solving (\ref{1OOT}) with initial data $u$ are given below.
\begin{Proposition}\label{mei}
Assume that $u\in \mathcal{H}_Q(T)$  satisfies
\begin{align}\label{BOOT}
\|\partial_xu\|_{L^{\infty}_tL^2_{x}}=\epsilon_1\ll1.
\end{align}
And let $v(s,t,x)$ be the solution of heat flow (\ref{1OOT}) with initial data $u(t,x)$. Then $v$ satisfies
\begin{align*}
\sup_{s\ge 0}(1+s2^{2k})^{31}2^{k}\|P_{k }v\|_{L^{\infty}_tL^2_{x}}\lesssim  2^{-\sigma k} \gamma_k(\sigma)
\end{align*}
for all $\sigma\in[0,\frac{99}{100}]$, $k\in\Bbb Z$.
Moreover, for any $\sigma\in[\frac{99}{100},\frac{5}{4}]$,  $k\in\Bbb Z$, we have
\begin{align*}
\sup_{s\ge 0}(1+s 2^{2k})^{30}2^{\sigma k+k}\|P_{k }v\|_{L^{\infty}_tL^2_{x}}\lesssim  \gamma_k(\sigma)+\gamma_k(\sigma-\frac{3}{8})\gamma_k(\frac{3}{8}).
\end{align*}
\end{Proposition}
\begin{Remark} The power of $(1+s2^{2k})$ in  Proposition \ref{mei} can be chosen to be any $M\in \Bbb Z_+$  with additionally assuming that $\epsilon_1$ is sufficiently small depending on $M$. See Proposition 3.4 below.
\end{Remark}

\subsection{Before Iteration }

Given an initial data  $v_0\in \mathcal{H}_Q$ with energy sufficiently small, by Lemma  \ref{kuai} the corresponding heat flow is global with (\ref{FD}) holding. Then combining this with local Cauchy theory in Sobolev spaces for heat flows yields the following lemma.
\begin{Lemma}\label{KJm1}
Assume that $v_0\in \mathcal{H}_Q$ has sufficiently small energy.
And let $v(s,x)$ be the solution of heat flow (\ref{1OOT}) with initial data $v_0$.
Given arbitrary  $L\in\Bbb Z_+$, there exist constants $C_{L}>0$, $C_{s}>0$, such that for any $s\ge 0, 0\le j\le L$,
\begin{align*}
\|\partial^{j+1}_xv(s)\|_{H^L_{x}}\le C_{L}(1+s)^{-\frac{j}{2}};\mbox{ } \|v(s)-Q\|_{L^2_{x}}\le C_{s}.
\end{align*}
\end{Lemma}

\begin{Proposition}\label{4.2}
Assume that $u\in \mathcal{H}_Q(T)$   satisfies
\begin{align}\label{BOOT}
\sum_{k\in\Bbb Z}2^{2k}\|P_{k }u\|^2_{L^{\infty}_tL^2_{x}}=\epsilon^2_1\ll1.
\end{align}
And let $v(s,t,x)$ be the solution of heat flow (\ref{1OOT}) with initial data $u(t,x)$. Then for $\sigma\in[0,\frac{99}{100}]$ and all $k\in\Bbb Z$, $v$ satisfies
\begin{align}\label{2JHB}
\sup_{s\in[0,\infty)}(1+s2^{2k})^{31}2^{k+\sigma k}\|P_{k}v\|_{L^{\infty}_tL^2_{x}}\lesssim \gamma_k(\sigma).
\end{align}
\end{Proposition}

\begin{proof}
Since $v$ converges to a fixed point $Q\in\mathcal{N}$ as $s\to\infty$, we put
\begin{align*}
S^{l}_{ij}(v)=S^{l}_{ij}(Q)+(S^{l}_{ij}(v)-S^{l}_{ij}(Q)).
\end{align*}
The $\{S^{l}_{ij}(Q)\}$ part is constant and makes acceptable contribution to the final estimates by [Lemma 8.3,\cite{BIKT}].  Moreover, by Lemma \ref{kuai}, for all nonnegative integers $L$ the remained part satisfies
\begin{align*}
\|\partial^{L+1}_x(S^{l}_{ij}(v)-S^{l}_{ij}(Q))\|_{L^{\infty}_tL^2_x}\lesssim_L s^{-\frac{L}{2}}\|\nabla u\|_{L^{\infty}_tL^2_x}.
\end{align*}
Thus for all $j\in\Bbb Z_+$ we have the bound
\begin{align}\label{GHv}
(1+2^{2k}s)^{j}\left\|\partial_x\left[P_k(S^{l}_{ij}(v)-S^{l}_{ij}(Q))\right]\right\|_{L^{\infty}_tL^2_x}\lesssim_j \epsilon_1.
\end{align}
Now, let's use [Lemma 8.3,\cite{BIKT}]'s arguments.
Given $\sigma\in[0,\frac{99}{100}]$, define $B_{1,\sigma}(S)$ to be
\begin{align*}
B_{1,\sigma}(S)=\sup_{k\in\Bbb Z,s\in[0,S)} \gamma^{-1}_k(\sigma)(1+s2^{2k})^{31}2^{\sigma k}2^{k}\|P_{k }v\|_{L^{\infty}_tL^2_{x}}.
\end{align*}
By Lemma \ref{KJm1} and the fact  $\{\gamma_k(\sigma)\}$ is a frequency envelope,  $B_{1,\sigma}(S)$ is well-defined for $S\ge 0$ and continuous in $S$ with
$\mathop {\lim }\limits_{S \to 0} B_{1,\sigma}(S)=1$.

Then by trilinear Littlewood-Paley decomposition (see (\ref{c6FC}) in Lemma \ref{A1}), we have
\begin{align*}
2^{k}\|P_{k }S^{l}_{ij}(f)\partial_af^{i}\partial_a f^{j}\|_{L^{\infty}_tL^2_x}&\lesssim 2^{k}\sum_{k_1\le k}\mu_{k_1}2^{k_1}\mu_k+\sum_{k_2\ge k}2^{2k}\mu^2_{k_2}+ a_{k}(\sum_{k_1\le k}2^{k_1}\mu_{k_1})^2\\
&+\sum_{k_2\ge k}2^{2k}2^{-k_2}a_{k_2}\mu_{k_2}\sum_{k_1\le k_2}2^{k_1}\mu_{k_1},
\end{align*}
where $\{a_k\}$, $\{\mu_k\}$  denote
\begin{align}
a_k&:=\sum_{|k-k'|\le 20}\sum^{N}_{l,i,j=1}\|\partial_x P_{k'}(S^{l}_{ij}(v))\|_{L^{\infty}_tL^2_x}; \mbox{ }\mu_k :=\sum^{N}_{l=1}\sum_{|k'-k|\le 20}2^{k'}\|P_{k'}v^{l}\|_{L^{\infty}_tL^2_x}.\label{N08}
\end{align}
Then by definition of $B_1(S)$ and slow variation of envelopes, for $s\in[0,S], \sigma\in[0,\frac{99}{100}]$, we get
 \begin{align}
&2^{k}\|P_{k }S^{l}_{ij}(f)\partial_af^{i}\partial_a f^{j}\|_{L^{\infty}_tL^2_x}\label{CbmXZ}\\
&\lesssim  B_{1,\sigma}B_{1,0}(1+s2^{2k})^{-31} \gamma_k\sum_{k_1\le k}2^{k_1-\sigma k_1+k}\gamma_{k_1}(\sigma)+B_{1,\sigma}B_{1,0}\sum_{k_2\ge k}2^{2k-\sigma k_2}(1+s2^{2k_2})^{-62}\gamma_{k_2}\gamma_{k_2}(\sigma)\nonumber\\
&+B_{1,\sigma}B_{1,0}2^{ k} a_{k} (\sum_{k_1\le k}2^{k_1}(1+s2^{2k_1})^{-31}\gamma_{k_1}) (\sum_{k_1\le k}2^{k_1-\sigma k_1}(1+s2^{2k_1})^{-31}\gamma_{k_1}(\sigma)) \nonumber\\
&+B_{1,\sigma}B_{1,0}\sum_{k_2\ge k}2^{2k}(1+s2^{2k_2})^{-31}2^{-\sigma k_2}a_{k_2}\gamma_{k_2}(\sigma)\gamma_{k_2}\nonumber\\
&\lesssim B_{1,\sigma}B_{1,0} (1+s2^{2k})^{-62}2^{2k-\sigma k}\gamma_k\gamma_k(\sigma)+ B_{1,\sigma}B_{1,0} \sum_{k_2\ge k}2^{2k-\sigma k_2}(1+s2^{2k_2})^{-31}\gamma_{k_2}\gamma_{k_2}(\sigma)\nonumber\\
&+B_{1,\sigma}B_{1,0}\sum_{k_2\ge k}2^{-\sigma k_2}2^{2k}(1+s2^{2k_2})^{-31}a_{k_2}\gamma_{k_2}\gamma_{k_2}(\sigma)+ a_{k}2^{-\sigma k}B_{1,\sigma}B_{1,0}2^{2k} \gamma_{k}\gamma_{k}(\sigma) \label{CmXZ}
\end{align}
Applying (\ref{GHv}) to $\{a_k\}$, (\ref{CmXZ}) is further bounded by
\begin{align*}
(\ref{CmXZ})\lesssim 2^{-\sigma k} B_{1,\sigma}B_{1,0}2^{2k}\sum_{k_2\ge k}(1+s2^{2k_2})^{-31}\gamma_{k}\gamma_{k_2}(\sigma).
\end{align*}
Therefore, for $s\ge 0$, we conclude that  (\ref{CbmXZ}) is dominated by
\begin{align*}
2^{k}\|P_{k }S^{l}_{ij}(v)\partial_av^{i}\partial_a v^{j}\|_{L^{\infty}_tL^2_x}\lesssim 2^{-\sigma k}2^{2k} B_{1,\sigma}B_{1,0} \sum_{k_2\ge k}(1+s2^{2k_2})^{-31}\gamma_{k_2}\gamma_{k_2}(\sigma).
\end{align*}
Hence by Duhamel principle:
\begin{align*}
(1+s2^{2k})^{31}&2^{k+\sigma k}\|P_{k }v\|_{L^{\infty}_tL^2_x}\lesssim (1+s2^{2k})^{31}e^{-s2^{{2k}}}2^{k+\sigma k}\|P_ku\|_{L^{\infty}_tL^2_x}\\
&+B_{1,\sigma}B_{1,0}(1+s2^{2k})^{M}\int^{s}_0e^{-(s-\tau)2^{2k}}
2^{k+\sigma k}\|P_{k }S^{l}_{ij}(v)\partial_av^{i}\partial_a v^{j}\|_{L^{\infty}_tL^2_x}d\tau,
\end{align*}
and the inequality
\begin{align}\label{KJn}
\int^{s}_0e^{-(s-\tau)\lambda}(1+\tau \lambda_1)^{-31}d\tau\lesssim s (1+\lambda s)^{-31}(1+\lambda_1 s)^{-1},
\end{align}
we get
\begin{align*}
(1+s2^{2k})^{31}2^{k+\sigma k}\|P_{k }v\|_{L^{\infty}_tL^2_x}&\lesssim \gamma_k(\sigma)+ B_{1,\sigma}B_{1,0}(S)2^{2k} s\sum_{k_2\ge k}\gamma_{k_2}\gamma_{k_2}(\sigma )(1+2^{2k_2}s)^{-1}\\
 &\lesssim \gamma_k(\sigma)+ B_{1,\sigma}(S)B_{1,0}(S)\epsilon_1\gamma_k(\sigma ).
\end{align*}
Then $B_{1,0}\lesssim 1+\epsilon_1 B^2_{1,0}$. By $B_{1,0}(0)\le 1$ and $\epsilon_1$ is sufficiently small, we have $B_{1,0}(S)\lesssim 1$ for all $S\ge 0$. Then
using $B_{1,\sigma}\lesssim 1+\epsilon_1 B_{1,0}B_{1,\sigma}$ and $B_{1,\sigma}(0)\le 1$, we get   $B_{1,\sigma}(S)\lesssim 1$ for any $\sigma\in[0,\frac{99}{100}]$ and any $S\ge 0$ provided that  $\epsilon_1$ is sufficiently small.
Thus (\ref{2JHB}) has been proved.
\end{proof}

\begin{Remark} The power of $(1+s2^{2k})$ in  Proposition \ref{4.2} can be chosen to be any $M\in \Bbb Z_+$  with additionally assuming that $\epsilon_1$ is sufficiently small depending on $M$. See Proposition 3.4 below.
\end{Remark}

\subsection{First time iteration }

We state the first time iteration in the the following Proposition.
\begin{Proposition}\label{uyhgvbL}
Assume that $u\in \mathcal{H}_Q(T)$   satisfies (\ref{BOOT}).
And let $v(s,t,x)$ be the solution of heat flow (\ref{1OOT}) with initial data $u(t,x)$. Then for $\sigma\in[\frac{99}{100},\frac{5}{4}]$ and any $k\in\Bbb Z$, $v$ satisfies
\begin{align}\label{4JHB}
\sup_{s\in[0,\infty)}(1+s2^{2k})^{30}2^{k+\sigma k}\|P_{k}v\|_{L^{\infty}_tL^2_{x}}\lesssim \gamma_k(\sigma)+\gamma_k(\sigma-\frac{3}{8})\gamma_k(\frac{3}{8}).
\end{align}
\end{Proposition}
\begin{proof}
The key point is to improve the bounds of $\{a_k\}$ defined by (\ref{N08}). For this, we use dynamic separation again. One has
 \begin{align}\label{Uy}
S^{l}_{ij}(v)(s)=S^{l}_{ij}(Q)-\int^{\infty}_s(DS^{l}_{ij})(v)\cdot \partial_s vds'.
\end{align}
By Proposition \ref{4.2}, for $\sigma\in[0,\frac{99}{100}]$ and any $k\in\Bbb Z$,  we get
 \begin{align*}
2^{k+\sigma k}\|P_k\Delta v\|_{L^{\infty}_tL^2_x}\lesssim (2^{2k}s+1)^{-31}2^{2k}\gamma_k(\sigma).
\end{align*}
And repeating the proof of  Proposition \ref{4.2} gives
 \begin{align*}
\sum_{a=1,2}2^{k+\sigma k}\|S^{l}_{ij}(\partial_a v^i,\partial_a v^j)\|_{L^{\infty}_tL^2_x}\lesssim 2^{2k}\sum_{k_1\ge k}(2^{2k_1}s+1)^{-31}\gamma_{k_1}\gamma_{k_1}(\sigma).
\end{align*}
Thus, given $s\in [2^{2k_0-1}, 2^{2k_0+1})$, by the heat flow equation we get for all $k\in\Bbb Z$ and $\sigma\in[0,\frac{99}{100}]$
 \begin{align}
 2^{k+\sigma k}\|P_k\partial_s v\|_{L^{\infty}_tL^2_x}
&\lesssim (2^{2k}s+1)^{-31}2^{2k}\gamma_k(\sigma)+ \sum_{k_1\ge k}(2^{2k}s+1)^{-31}2^{2k}\gamma_{k_1}\gamma_{k_1}(\sigma)\label{Gan}\\
&\lesssim (2^{2k}s+1)^{-31}2^{2k}\gamma_k(\sigma)+ 1_{k+k_0\ge 0}(2^{2k}s+1)^{-31}2^{2k}\gamma_{k}\gamma_{k}(\sigma)\nonumber\\
&+2^{2k}1_{k+k_0\le 0}\sum_{k\le l\le -k_0}\gamma_{l}\gamma_{l}(\sigma).\nonumber
\end{align}
Recall the  bound
\begin{align}\label{GbM1}
2^{k}\|P_k[(DS)(v)]\|_{L^{\infty}_tL^2_x}\lesssim  \epsilon_1(2^{2k}s+1)^{-j}
\end{align}
for all $j\in \Bbb Z_+$ and $k\in \Bbb Z$. Then for  $s\in[2^{2k_0-1},2^{2k_0+1})$, repeating bilinear arguments, (\ref{Uy}) shows that if $k+k_0\ge 0$ then
 \begin{align}
\left\|P_{k}[S^{l}_{ij}(v)(s)]\right\|_{L^{\infty}_tL^2_x}
&\lesssim \int^{\infty}_s\sum_{ |k_1-k|\le 4}\|P_{\le k-4}(DS(v))\|_{L^{\infty}_{t,x}} \|P_{k_1}\partial_s v\|_{L^{\infty}_tL^2_x}ds'\nonumber\\
&+\int^{\infty}_s2^{k}\sum_{ |k_1-k_2|\le 8,k_1,k_2\ge k-4} \|P_{ k_2}(DS(v))\|_{L^{\infty}_{t}L^2_x}\|P_{k_1}\partial_s v\|_{L^{\infty}_tL^2_x}ds'
\nonumber\\
&+\int^{\infty}_s\sum_{ |k_2-k|\le 4, k_1\le k-4}2^{k_1}\|P_{k_2}(DS(v))\|_{L^{\infty}_{t}L^2_x} \|P_{k_1}\partial_s v\|_{L^{\infty}_tL^2_x}ds'\nonumber\\
&\lesssim 2^{-\sigma k-k} (2^{2k+2k_0}+1)^{-31}2^{2k+2k_0}\gamma_k(\sigma)\gamma_k,\label{U1y}
\end{align}
provided $\sigma\in [0,\frac{99}{100}]$, where we applied  (\ref{GbM1}), (\ref{Gan}) in the last line.
Moreover, for any $\sigma\in[0,\frac{99}{100}]$, $k_0\in\Bbb Z$, and $s\in[2^{2k_0-1},2^{2k_0+1}]$,  in the case  $k+k_0\le 0$ one has
 \begin{align}
 \left\|P_{k}[S^{l}_{ij}(v)(s)]\right\|_{L^{\infty}_tL^2_x}\lesssim \sum_{k_0\le j\le -k}2^{-\sigma k} 2^{j+2\delta|k+j|}\gamma_{k}(\sigma)\gamma_{k}\lesssim  2^{-\sigma k-k} \gamma_k(\sigma)\gamma_k. \label{U2y}
\end{align}
Thus (\ref{U2y}), (\ref{U1y}) yield the following  bounds for $\{a_k\}$:
 \begin{align}\label{Niubility}
{2^{\sigma k}}{a_k} \lesssim  {(1 + {2^{2k}}s)^{ - 30}}{\gamma _k}(\sigma ){\gamma _k},
\end{align}
provided that $\sigma\in[0,\frac{99}{100}]$. Now define the function $B_{2,\sigma}(S)$ for a given $\sigma\in[\frac{99}{100},\frac{5}{4}]$ by
 \begin{align*}
B_{2,\sigma}(S)=\sup_{k\in\Bbb Z,s\in[0,S)} \left(\gamma^{(1)}_k(\sigma)\right)^{-1}2^{\sigma k}(1+s2^{2k})^{30}2^{k}\|P_{k }v\|_{L^{\infty}_tL^2_{x}},
\end{align*}
where we denote
 \begin{align*}
\gamma^{(1)}_k(\sigma):=\left\{
                          \begin{array}{ll}
                            \gamma_k(\sigma), & \hbox{ }\sigma\in[0,\frac{99}{100}] \\
                            \gamma_k( \sigma)+\gamma_k(\sigma-\frac{3}{8})\gamma_{k}(\frac{3}{8}), & \hbox{ }\sigma\in(\frac{99}{100},\frac{5}{4}]
                          \end{array}
                        \right..
\end{align*}
Moreover, by Lemma \ref{KJm1} and the fact that  $\{\gamma^{(1)}_k(\sigma)\}$ is a frequency envelope of order $2\delta$, it is clear that  $B_{2,\sigma}:[0,\infty)\to \Bbb R^+$ is well-defined and continuous in $S\ge 0$ with
$\mathop {\lim }\limits_{S \to 0}B_{2,\sigma}(S)=1$.
Then by trilinear Littlewood-Paley decomposition (see (\ref{c6FC}) in Lemma \ref{A1}), the definition of $B_{2,\sigma}$ and slow variation of envelopes, we get for $s\in[0,S], \sigma\in[\frac{99}{100},\frac{5}{4}]$ that
 \begin{align}
&2^{k}\left\|P_{k }[S^{l}_{ij}(v)\partial_av^{i}\partial_a v^{j}]\right\|_{L^{\infty}_tL^2_x}\label{CbmXZ1}\\
&\lesssim B_{2,\sigma}B_{1,0}(1+s2^{2k})^{-30}2^{-\sigma k} \gamma^{(1)}_k(\sigma)\sum_{k_1\le k}2^{k_1+k}\gamma_{k_1}\\
&+ B_{2,\sigma}B_{1,0}\sum_{k_2\ge k}2^{2k-\sigma k_2}(1+s2^{2k_2})^{-60}\gamma_{k_2}\gamma^{(1)}_{k_2}(\sigma)\nonumber\\
&+B_{1,0}B_{1,\frac{3}{8}} a_{k} (\sum_{k_1\le k}2^{k_1}(1+s2^{2k_1})^{-30}\gamma_{k_1}) (\sum_{k_1\le k}2^{k_1-\frac{3}{8}\sigma k_1}(1+s2^{2k_1})^{-30}\gamma_{k_1}(\frac{3}{8})) \nonumber\\
&+ B_{2,\sigma}B_{1,0}\sum_{k_2\ge k}2^{2k}(1+s2^{2k_2})^{-30}2^{-\sigma k_2}a_{k_2}\gamma^{(1)}_{k_2}(\sigma)\gamma_{k_2}\nonumber\\
&\lesssim   B_{2,\sigma}B_{1,0} \sum_{k_2\ge k}2^{2k-\sigma k_2}(1+s2^{2k_2})^{-30}\gamma_{k_2}\gamma^{(1)}_{k_2}(\sigma)\nonumber\\
&+ B_{2,\sigma}B_{1,0}\sum_{k_2\ge k}2^{-\sigma k_2}2^{2k}(1+s2^{2k_2})^{-30}a_{k_2}\gamma_{k_2}\gamma^{(1)}_{k_2}(\sigma)+ a_{k}2^{-\frac{3}{8}\sigma k}B_{1,0}B_{1,\frac{3}{8}}2^{2k} \gamma _{k}\gamma_{k}(\frac{3}{8}). \label{CmXZ1}
\end{align}
Then applying the trivial bound (\ref{GHv}) to the RHS of (\ref{CmXZ1}) except the last term and applying (\ref{Niubility}) to $\{a_k\}$ in the last term, we get for all $\sigma\in (\frac{99}{100},\frac{5}{4}]$  that
 \begin{align*}
2^{k+\sigma k}\left\|P_{k }[S^{l}_{ij}(v)\partial_av^{i}\partial_a v^{j}]\right\|_{L^{\infty}_tL^2_x}&\lesssim B_{1,0}B_{2,\sigma} 2^{2k}\sum_{k_2\ge k} (1+s2^{2k_2})^{-30}\gamma_{k_2}\gamma^{(1)}_{k_2}(\sigma)\\
&+B_{1,0}B_{1,\frac{3}{8}}2^{2k} 2^{-\sigma k} (1+s2^{2k})^{-30}\gamma_{k}(\sigma -\frac{3}{8})\gamma_{k}(\frac{3}{8})\gamma_k1_{k+k_0\ge 0}
\\
&+B_{1,0}B_{1,\frac{3}{8}}2^{2k} 2^{-\sigma k}2^{2\delta|k+k_0|}\gamma_{k}(\sigma -\frac{3}{8})\gamma_{k}(\frac{3}{8})\gamma_k1_{k+k_0\le 0}
\end{align*}
if $s\in[2^{2k_0-1},2^{2k_0+1})$.
Then using Duhamel principle, (\ref{KJn}) and the following inequality
 \begin{align*}
(1+2^{2k}s)^{30}e^{-2^{2k}s}\int^{s}_0e^{s'2^{2k}}(s'2^{2k})^{-\delta}1_{s'\le 2^{-2k}}ds'\lesssim 2^{-2k},
\end{align*}
we obtain
 \begin{align*}
2^{k+\sigma k}(1+2^{2k}s)^{30}\left\|P_{k }v\right\|_{L^{\infty}_tL^2_x}\lesssim
(1+\epsilon_1 B_{1,0}B_{1,\frac{3}{8}}+\epsilon_1 B_{2,\sigma}B_{1,0}) \gamma^{(1)}_{k}(\sigma).
\end{align*}
Since $B_{1,\tilde{\sigma}}\lesssim 1$ for $\tilde{\sigma}\in[0,\frac{99}{100}]$ has been proved in Proposition 3.2, we arrive at
\begin{align*}
B_{2,\sigma}\lesssim 1+\epsilon_1 B_{2,\sigma},\mbox{  }\forall \sigma\in(\frac{99}{100},\frac{5}{4}],
\end{align*}
which  shows $B_{2,\sigma}\lesssim 1$,
thus finishing our proof.
\end{proof}

We define the frequency envelope $\gamma^{(j)}_k(\sigma)$, $j=0,1$ by :
 \begin{align}\label{fd1}
  \gamma _k^{(0)}(\sigma ) &= {\gamma _k}(\sigma ),0 \leqslant \sigma  < \frac{{99}}
{{100}} \hfill \\
  \gamma _k^{(1)}(\sigma ) &= \left\{ \begin{gathered}
  \gamma _k^{({\text{0}})}(\sigma ),0 \leqslant \sigma  \le \frac{{99}}
{{100}} \hfill \\
 \gamma_k(\sigma)+ \gamma _k^{({\text{0}})}(\sigma  - \frac{3}
{8}){\gamma _k}(\frac{3}
{8}),\frac{{99}}
{{100}} <\sigma  \le \frac{{\text{5}}}
{4} \hfill \\
\end{gathered}  \right.
\end{align}
And the frequency envelopes $\gamma^{(j)}_k(\sigma)$, $j\ge 2$ are defined by induction:
 \begin{align}\label{fd2}
  \gamma _k^{(j)}(\sigma )=\left\{
                          \begin{array}{ll}
                           {\gamma _k}^{(j-1)}(\sigma ),&0 \leq  \sigma  \le  \frac{j+3}
{4}   \\
                          \gamma _k(\sigma ) +\gamma _k^{(j-1)}(\sigma-\frac{3}{8} ){\gamma _k}(\frac{3}
{8}),&\frac{j+3}{4} <\sigma  \le \frac{j+4}{4}
                          \end{array}
                        \right.
\end{align}
Define the sequence $\{\gamma^{(j)}_{k,s}(\sigma)\}_{k\in\Bbb Z}$ with $j\in\Bbb N$ by
 \begin{align}\label{fd3}
\gamma^{(j)}_{k,s}(\sigma)=\left\{
                     \begin{array}{ll}
                       2^{k+k_0}\gamma^{(j)}_{k}(\sigma)\gamma^{(j)}_{-k_0}(0), & \hbox{ }k+k_0\ge 0 \\
                      \sum^{-k_0}_{l=k}\gamma^{(j)}_{l}(\sigma)\gamma^{(j)}_{l}(0), & \hbox{ }k+k_0\le 0
                     \end{array}
                   \right.
\end{align}
for $s\in[2^{2k_0-1},2^{2k_0+1})$, $k,k_0\in\Bbb Z$.

We state the $j$-th time iteration in the the following proposition.
\begin{Proposition}($j$-th iteration)  \label{Jth}
Let $j\in \Bbb N, M\in\Bbb Z_+ $. Assume that $u\in \mathcal{H}_Q(T)$  satisfies (\ref{BOOT}) with $\epsilon_1$ sufficiently small depending on $j+M$.
Let $v(s,t,x)$ be the solution of heat flow (\ref{1OOT}) with initial data $u(t,x)$. Then for $\sigma\in[0,1+\frac{j}{4}]$ and all $s\ge 0$, $v$ satisfies
  \begin{align}\label{j4JHB}
\sup_{s\in[0,\infty)}(1+s2^{2k})^{M}2^{k+\sigma k}\|P_{k}v\|_{L^{\infty}_tL^2_{x}}\lesssim \gamma^{(j)}_k(\sigma).
\end{align}
\end{Proposition}
\begin{proof}
Define intervals $\{\mathbb{I}_{l}\}^{\infty}_{l=0}$ by
  \begin{align*}
\mathbb{I}_{0}:=[0,\frac{99}{100}]; \mbox{ }\mathbb{I}_1=(\frac{99}{100},\frac{5}{4}];  \mbox{ }\mathbb{I}_{l}=(\frac{3+l}{4}, \frac{4+l}{4}], l\ge 2.
\end{align*}
Given $K\in\Bbb Z_+$, $\sigma\in \mathbb{I}_{l}$, $l\in\Bbb N$, we denote
\begin{align*}
B_{l+1,\sigma,K}(S):=\sup_{s\in[0,S),k\in\Bbb Z}\frac{1}{ \gamma^{(l)}_k(\sigma)}
(1+s2^{2k})^{K}2^{k+\sigma k}\|P_{k}v\|_{L^{\infty}_tL^2_{x}}.
\end{align*}
And let
\begin{align*}
\mathbb{B}_{l+1,K}(S):=\sup_{\sigma \in \bigcup\limits_{\ell\le l} {\mathbb I}_{\ell}} B_{l+1,\sigma, K}(S).
\end{align*}
(In this notation,  Proposition \ref{4.2} and  Proposition  \ref{uyhgvbL} have proved $\mathbb{B}_{2,30}(S) \lesssim 1$.)

Moreover, the argument of Proposition  \ref{uyhgvbL}  indeed shows:

(i) For all $K_0\ge 2, j\in\Bbb N$, $0\le a\le j+1$,
 \begin{align*}
 {2^{k}\left\| {{P_k}[D^{a}S(v)]} \right\|_{L_t^\infty L_x^2}} \lesssim C_{K_0,j}\epsilon{(1 + {2^{2k}}s)^{ -K_0-(j+1)}}.
 \end{align*}
(ii) For all $K_0\ge 2, j\in\Bbb N$,  if
 \begin{align*}
  \left\{
 \begin{array}{lll}
  {2^{k}\left\| {{P_k}[{D^{j+1}}S (v)]} \right\|_{L_t^\infty L_x^2}} \lesssim  \epsilon{(1 + {2^{2k}}s)^{ -K_0-j-1 }}\\
  {\left\| {{P_k}v} \right\|_{L_t^\infty L_x^2}} \lesssim {2^{ - \sigma k}}\gamma^{(0)}_k(\sigma ){(1 + {2^{2k}}s)^{ -K_0-j-1 }} \\
 \end{array}
\right.,
 \end{align*}
then
 \begin{align*}
  {2^{k}\left\| {{P_k}[D^{j}S (v)]} \right\|_{L_t^\infty L_x^2}} \lesssim {2^{ - \sigma k}}\gamma _k^{(0)}(\sigma ){(1 + {2^{2k}}s)^{ -K_0-j}}.
 \end{align*}
  where the implicit constant in the conclusion is of the form $C(1+C^2_1+C^2_2)$ if we denote $C_1,C_2$ to be the implicit constants in the conditions of (ii). Here, $C$ is universal and $C_1,C_2$ may depend on $j,K_0$.\\
(iii) For all $K_0\ge 2, j\in\Bbb N$, $0\le a\le j+1$, if
 \begin{align*}
 &\left\{
  \begin{array}{ll}
      {2^{k}\left\| {{P_k}[D^{a+1}S(v)]} \right\|_{L_t^\infty L_x^2}} \lesssim {2^{ - \sigma k}}\gamma^{(j-(a+1))}_k(\sigma )(1 + {2^{2k}}s)^{-K_0-(a+1)},   \hbox{ } \\
    {\left\| {{P_k}v} \right\|_{L_t^\infty L_x^2}}  \lesssim  {2^{ - \sigma k}}\gamma^{(j-a)} _k(\sigma ){(1 + {2^{2k}}s)^{-K_0-(a+1)}} \hbox{ }
  \end{array}
\right.
\end{align*}
  then
 \begin{align*}
  {2^{k}\left\| {{P_k}[D^{a}S(v)]} \right\|_{L_t^\infty L_x^2}} \lesssim {2^{ - \sigma k}}\gamma^{(j-a)}_k(\sigma ){(1 + {2^{2k}}s)^{-K_0-a}},
  \end{align*}
  where the implicit constant in the conclusion is of the form $C(1+C^2_1+C^2_2)$ if we denote $C_1,C_2$ to be the implicit constants in the conditions of (iii). Here, $C$ is universal and $C_1,C_2$ may depend on $j,K_0$.\\
(iv) For any $K\ge 2, j\ge 1$, $1\le a\le j+1$, $\sigma\in{\Bbb I}_{a}$, if
 \begin{align*}
 \left\{
\begin{array}{ll}
  {2^{k}\left\| {{P_k}[S (v)]} \right\|_{L_t^\infty L_x^2}} \le C_{K} {2^{ - \sigma k}}\gamma^{(a - 1)} _k(\sigma ){(1 + {2^{2k}}s)^{ -K}} \\
 { \left\| {{P_k}v} \right\|_{L_t^\infty L_x^2}} \le  B_{a+1,\sigma,K}{2^{ - \sigma k}}\gamma _k^{(a)}(\sigma ){(1 + {2^{2k}}s)^{ -K}}  \\
  \end{array}
\right.,
\end{align*}
then for all  $S\in[0,\infty)$ there holds
 \begin{align*}
B_{a+1,\sigma,K}(S) \le C_* ( 1+  \epsilon_1{\Bbb B}_{a,K}B_{a+1,\sigma,K}(S)+C_{K}\sum^{a}_{l=1}\sup_{S\in[0,\infty)}{\mathbb{B}}^2_{l,K+l+1}(S)).
\end{align*}
where $C_*$ depends only on $d$ and emerges from the Littlewood-Paley trilinear decomposition. Then our proposition follows by iteration. To be concrete, we make several remarks.  First, in order to get the $M$-power decay in (\ref{j4JHB}), it suffices to set $K_0=M+4$ and the top involved derivative order is $D^{j+1}S$. Second, let us describe the iteration in a clearer way: In the first step, one verifies
 \begin{align}\label{FvBn}
\sup_{S\in[0,\infty)}B_{1,K_0+j+1}(S) \le C_{K_0,j},
\end{align}
i.e. the second conditions in (ii).
This was presented in Proposition \ref{4.2}. (We emphasize that in this step $\epsilon_1$ shall be sufficiently small depending on $K_0+j$.) In the second step, one verifies $
\sup_{S\in[0,\infty)}B_{2,K_0+j}(S) \le C_{K_0,j}$, and in the $a$-th step  one verifies $
\sup_{S\in[0,\infty)}B_{a,K_0+j+2-a}(S) \le C_{K_0,j}$. This is presented as (iii) and (iv). Thus in the $j$-th step, we get (\ref{j4JHB}).
\end{proof}

\subsection{Rough dynamical separation} \label{VNM}

Recall the notations $\psi^{\alpha}_{i}=\langle \partial_iv,e_{\alpha}\rangle$, $\psi^{\overline{\alpha}}_{i}=\langle \partial_iv,J e_{\alpha}\rangle$, $\alpha=1,...,n, i=0,...,2,3$, and $\phi^{\alpha}_{i}=\psi^{\alpha}_i+\sqrt{-1}\psi^{\overline{\alpha}}_i$. Here, $i=0$ refers to the $t$ variable and $i=3$ refers to the $s$ variable.

We aim to bound connection coefficients in the localized frequency  spaces.
As a preparation, we first derive a suitable form of connection coefficients.  By definitions, we see
\begin{align*}
&{\bf R}( \mathbf{E}\phi_i,{\bf E}\phi_s)={\bf R}\left((\Re \phi^\alpha_i)e_{\alpha}+(\Im \phi^\alpha_i)e_{\overline{\alpha}}, (\Re \phi^\alpha_s)e_{\beta}+(\Im \phi^\alpha_s)e_{\overline{\beta}}\right)\\
&=(\phi^\alpha_i\wedge \phi^{\beta}_s){\bf R}(e_{\alpha},e_{\overline{\beta}})
+(\phi^\alpha_i\cdot \phi^{\beta}_s){\bf R}(e_{\alpha},e_{{\beta}}).
\end{align*}
where we denote $z_1\wedge z_2=-\Im(z_1\overline{z_2})$, $z_1\cdot z_2=\Re z_1\Re z_2+\Im z_1\Im {z_2}$ for complex numbers $z_1,z_2$.
Thus schematically under the frame ${\bf E}=\{e_{\alpha},e_{\overline{\alpha}}\}^{n}_{\alpha=1}$ we can write
\begin{align}\label{PSD2}
\left\{
  \begin{array}{ll}
    \left(\Re[A_i]\right)^{\gamma}_{\theta}=\sum\int^{\infty}_s(\phi^\alpha_i\diamond\phi^{\beta}_s)\langle{\bf R}(e_{\alpha},e_{\beta,\overline{\beta}}) (e_{\gamma}),e_{\theta}\rangle ds', & \hbox{ } \\
    \left(\Im[A_i]\right)^{\gamma}_{\overline{\theta}}=\sum\int^{\infty}_s(\phi^\alpha_i\diamond\phi^{\beta}_s)\langle{\bf R}(e_{\alpha},e_{\beta,\overline{\beta}}) (e_{\gamma}),e_{\overline{\theta}}\rangle ds', & \hbox{ }
  \end{array}
\right.
\end{align}
where $\diamond=``\wedge"$ when $ e_{\beta,\overline{\beta}}=e_{\overline{\beta}}$, and $\diamond= ``\cdot"$ when $ e_{\beta,\overline{\beta}}=e_{ {\beta}}$.
For simplicity, we schematically write
\begin{align*}
A_i(s)= \int^{\infty}_s\left(\phi_i\diamond\phi_s\right)\langle{\bf R}(e_{j_0},e_{j_1}) (e_{j_2}),e_{j_3}\rangle ds',
\end{align*}
where $\{j_c\}^{3}_{c=0}$ run in $\{1,...,2n\}$, and $i$ runs in $\{0,1,2\}$. Recall also that $\phi_s=\sum^{2}_{l=1}D_l\phi_l$.

With abuse of notations, denote
\begin{align}
\mathcal{G}(s)= \langle {\bf R}(e_{j_0},e_{j_1})e_{j_2},e_{j_3}\rangle(s).
\end{align}
for any given $j_0,...,j_3\in\{1,2,...,2n\}$. We expand $\mathcal{G}$ as
\begin{align*}
 \langle {\bf R}(e_{j_0},e_{j_1})e_{j_2},e_{j_3}\rangle(s)&=\lim_{s\to\infty} \langle {\bf R}(e_{j_0},e_{j_1})e_{j_2},e_{j_3}\rangle-\int^{\infty}_s\partial_s\langle {\bf R}(e_{j_0},e_{j_1})e_{j_2},e_{j_3}\rangle ds'\\
&=\Gamma^{\infty}-\int^{\infty}_s\psi^{l}_s(\widetilde{\nabla} {\bf R})(e_{l};e_{j_0},...,e_{j_3})ds',
\end{align*}
where  $\Gamma^{\infty}$ denotes the limit part which is constant, and we used the identity $\nabla_s{e}_p=0$ for all $p=1,...,2n$ in the last line. Here, we view  $\bf R$ as a type $(0,4)$ tensor.

With the above notations, we write
\begin{align}\label{PSD}
A_i(s)=\sum_{j_0,j_1,j_2,j_3}\int^{\infty}_s\left(\phi_i\diamond\phi_s\right)\mathcal{G}ds',
\end{align}
and $\mathcal{G}$ is decomposed as
\begin{align*}
\mathcal{G}=\Gamma^{\infty}-\int^{\infty}_s\psi^{l}_s(\widetilde{\nabla} {\bf R})(e_{l};e_{j_0},...,e_{j_3})ds'.
\end{align*}
Of course, one can perform this separation for any time as desired.
Denote
\begin{align*}
\mathcal{G}^{(j)}= (\widetilde{\nabla}^{j}{\bf R})(\underbrace{e,...,e}_{j};e_{j_0},...,e_{j_3}),
\end{align*}
and
\begin{align*}
\Gamma^{\infty,(j)}=\lim\limits_{s\to\infty}\mathcal{G}^{(j)}(s).
\end{align*}
Then we can schematically write
\begin{align*}
\mathcal{G}=\Gamma^{\infty}-\int^{\infty}_s\psi_s(s_1)ds_1\left(\Gamma^{\infty,(1)}-\int^{\infty}_{s_1}\psi_s(s_2)ds_2(\Gamma^{\infty,(2)}+....)\right)
\end{align*}
For simplicity we also denote
\begin{align*}
\widetilde{\mathcal{G}} =\mathcal{G} -\Gamma^{\infty}, \mbox{ }\widetilde{\mathcal{G}}^{(j)}=\mathcal{G}^{(j)}-\Gamma^{\infty,(j)}.
\end{align*}

\subsection{ Intrinsic v.s. Extrinsic formulations in localized frequency pieces }

\begin{Proposition}\label{7.2}
Let $d=2$. And let $u \in\mathcal{H}_{Q}(T)$ satisfy
\begin{align}\label{qXaC}
\|\partial_x u \|_{L^{\infty}_tL^2_x}= \epsilon_1\ll 1.
\end{align}
Here,  we do not require $u$ to solve SMF.
Denote $v(s,t,x)$ the solution to heat flow with data $u(t,x)$, and denote $\{\phi_i\}$ the corresponding  differential fields under the caloric gauge.
Assume that $\{\eta_{k}(\sigma)\}$ is a frequency envelope of order $\delta$ such that for all $i=1,2$, $k\in\Bbb Z$,
\begin{align}\label{NMl0mnm}
2^{\sigma k}\|\phi_i(\upharpoonright_{s=0})\|_{L^{\infty}_tL^2_x}\le \eta_k(\sigma).
\end{align}
Then we have
\begin{align}
\gamma_k(\sigma)&\lesssim \eta_{k}(\sigma)\label{hbyu77}\\
(1+s2^{2k})^{30}2^{\sigma k}\|P_kA_i\|_{L^{\infty}_tL^2_x}&\lesssim \eta^{(0)}_{k,s}(\sigma).\label{shujiu1}
\end{align}
for any $\sigma\in[0,\frac{99}{100}]$,  $k\in\Bbb Z$.
Furthermore, assume that for  $\sigma\in [0,\frac{5}{4}]$, $\{\eta_{k}(\sigma)\}$ is a frequency envelope of order $\frac{1}{2}\delta$ such that for all $i=1,2$, $k\in\Bbb Z$,
(\ref{NMl0mnm}) holds.
Then for any $\sigma\in[0,\frac{5}{4}]$, $k\in\Bbb Z$, one has
\begin{align}
\gamma^{(1)}_k(\sigma)&\lesssim \eta^{(1)}_{k}(\sigma)\label{hbyu78}\\
(1+s2^{2k})^{29}2^{\sigma k}\|P_kA_i\|_{L^{\infty}_tL^2_x}&\lesssim \eta^{(1)}_{k,s}(\sigma).\label{shujiu2}
\end{align}
\end{Proposition}
\begin{proof}
{\bf Step 1.1.  $\sigma\in[0,\frac{99}{100}]$.}
Let $\mathcal{P}:\mathcal{N}\to \Bbb R^N$ be the isometric embedding. By definition, we see
\begin{align}\label{6.5a}
\partial_iv=\sum^{2n}_{l=1}\psi^{l} d\mathcal{P} (e_l)=\sum^{2n}_{l=1}\psi^{l}_i\chi^{\infty}_l +\sum^{2n}_{l=1}\psi^{l}_i\left( d\mathcal{P} (e_l)-\chi^{\infty}_l\right),
\end{align}
where $\{\chi^{\infty}_l\}$ are the corresponding limit of $(d\mathcal{P})(e_l)$ as $s\to\infty$ which are constant vectors belonging to $\Bbb R^N$. Denote
\begin{align}\label{jnb9va}
\omega_k(s)=\sum_{|k-k'|\le 20}\|P_{k'}\psi_i(s)\|_{L^{\infty}_tL^2_x}, \mbox{  }\nu_k(s)=\sum_{|k-k'|\le 20}2^{k'}\|P_{k'}\left( d\mathcal{P} (e_l)-\chi^{\infty}_l\right)\|_{L^{\infty}_tL^2_x}.
\end{align}
Then we see by Lemma \ref{kuai} that
\begin{align*}
\|\{\nu_k\}\|_{\ell^2}&\lesssim  \|\partial_i\left( d\mathcal{P} (e_l)-\chi^{\infty}_l\right)\|_{L^2_x}\lesssim  \|\partial_i v\|_{L^2_x}+\|A_i\|_{L^2_x}\lesssim \epsilon_1.
\end{align*}
Moreover, direct calculations give the inequality
\begin{align}
\|\partial^{L}_x\left((d\mathcal{P})(e_l)-\chi^{\infty}_l\right)\|_{L^{\infty}_tL^2_x}&\lesssim \sum_{0\le p, q\le L}\sum_{\mathcal{A}}|\partial^{\alpha_1}_x\phi_x|^{l_1}....|\partial^{\alpha_p}_x \phi_x|^{l_p}
|\partial^{\beta_1}_xA|^{n_1}....|\partial^{\beta_q}_x A|^{n_q},
\end{align}
where  $\mathcal{A}$ is the set of nonnegative indexes $l_1,...,n_q\in\Bbb Z$ and $(\alpha_1,...,\beta_q)\in \Bbb Z^2\times...\times\Bbb Z^2$  which satisfy
\begin{align*}
l_1(|\alpha_1|+1)+...+l_p(|\alpha_p|+1)+n_1(|\beta_1|+1)+...+n_{q}(|\beta_{q}|+1)=L.
\end{align*}
Suppose $l_1\ge 1$, by H\"older and Lemma \ref{kuai}, we get
\begin{align*}
&\|\partial^{L}_x\left((d\mathcal{P})(e_l)-\chi^{\infty}_l\right)\|_{L^2_x}\\
&\lesssim \epsilon_1 \sum s^{-\frac{\alpha_1}{2}}s^{-(l_1-1)\frac{\alpha_1+1}{2}}s^{-\frac{l_2(\alpha_2+1)}{2}-...- \frac{l_p(\alpha_p+1)}{2}}
s^{-\frac{(|\beta_1|+1)n_1}{2}-...-\frac{(|\beta_q|+1){n_q}}{2}}\\
&\lesssim \epsilon_1 s^{-\frac{L-1}{2}}.
\end{align*}
Suppose that $n_1\ge 1$, then we also obtain the same bound as above. Thus
we arrive at
\begin{align}\label{lkmna}
\|\{(1+s 2^{2k})^{M}\nu_k(s)\}\|_{\ell^2}&\lesssim_M \epsilon_1,
\end{align}
for $M\in\Bbb Z_+$.
Meanwhile, we see $\| (d\mathcal{P})(e_l)-\chi^{\infty}_l \|_{L^{\infty}}\lesssim 1$. Thus by [\cite{BIKT},(8.4)], we obtain
\begin{align}\label{KoJa}
2^{k}\|P_k(\partial_iv)\|_{L^{\infty}_tL^2_x}\lesssim 2^k\omega_k +\nu_k\sum_{k_1\le k}\omega_{k_1}2^{k_1}
+\sum_{k_1\ge k}2^{-2|k_1-k|}\omega_{k_1}2^{k_1}\nu_{k_1}.
\end{align}
Since when $s=0$, $\omega_k(0)\le 2^{-\sigma k}\eta_k(\sigma)$,  by slow variation of envelopes one deduces
\begin{align*}
&2^{k}\|P_k(\partial_iv)\|_{L^{\infty}_tL^2_x}\\
&\lesssim 2^k2^{-\sigma k}\eta_k(\sigma) +\nu_k \sum_{k_1\le k}2^{k_1+\delta|k-k_1|}2^{-\sigma k_1}\eta_{k_1}(\sigma)
+\sum_{k_1\ge k}2^{-2(k_1-k)}2^{\delta|k_1-k|} 2^{k_1}\nu_{k_1}2^{-\sigma k_1}\eta_{k_1}(\sigma)\\
&\lesssim 2^k 2^{-\sigma k}\eta_k(\sigma) (1 + \epsilon_1),
\end{align*}
for $\sigma\in[0,\frac{99}{100}]$, $s=0$.
Thus since $\{\eta_k(\sigma)\}$ is an envelope, by the definition of $\{\gamma_k(\sigma)\}$  we obtain
\begin{align}\label{kjln6a}
\gamma_k(\sigma)\lesssim  \eta_k(\sigma).
\end{align}
Hence, (\ref{hbyu77}) has been done for $\sigma\in[0,\frac{99}{100}]$.

{\bf Step 1.2.  $\sigma\in(\frac{99}{100},\frac{5}{4}]$.}
 Recall $\mathcal{P}:\mathcal{N}\hookrightarrow \Bbb R^N$ is the given isometric embedding.
Viewing  $d\mathcal{P}$ as a section of $T^*\mathcal{N} \otimes T\Bbb R^N$, then the connection on $\mathcal{N}$ induces a covariant derivative  ${\bf D}$ on the bundle $T^*\mathcal{N} \otimes T\Bbb R^N$.  We have the identity
\begin{align}\label{Hahaob2}
 d\mathcal{P}(e_l)-\chi^{\infty}_l=\int^{\infty}_s \psi^j_s\mathbf{D}d\mathcal{P}(e_j;e_l)ds'.
\end{align}
where we used the caloric condition $\nabla_s e_{l}=0$ for all $l=1,...,2n$. Similar to (\ref{lkmna}), direct calculations  give
\begin{align}\label{Hahaob}
\|P_{k}(\mathbf{D}d\mathcal{P}(e_j;e_l))\|_{L^{\infty}_tL^2_x}\lesssim_{M} \epsilon_1 2^{-k}(1+s2^{2k})^{-M}
\end{align}
for any $M\in \Bbb Z_+$, $k\in\Bbb Z$.

By (\ref{Gan}), we have the bound for $\partial_sv$:
\begin{align}
&2^{\sigma k+k}\|P_{k}(\partial_sv) \|_{L^{\infty}_tL^2_x} \lesssim 2^{2k}\left[(1+2^{2k}s)^{-31}1_{k+k_0\ge 0}\gamma_{k}(\sigma)+\sum_{k\le l\le -k_0}\gamma_{l}(\sigma)\gamma_l\right],\label{JIO1}
\end{align}
if $s\in [2^{2k_0-1}, 2^{2k_0+1})$, $k_0\in\Bbb Z$.
And using the identity $\psi^{l}_s=(d\mathcal{P}e_l)\cdot \partial_sv$, (\ref{lkmna}) and  (\ref{JIO1}) instead yield
\begin{align}\label{uyg}
\|P_{k}\psi_s\|_{L^{\infty}_tL^2_x}\lesssim 2^{k-\sigma k}\left(1_{k+k_0\ge 0}(1+s2^{2k})^{-31}\gamma_{k}(\sigma)+1_{k+k_0\le 0}\sum_{k\le l\le -k_0}\gamma_{l}(\sigma)\gamma_l\right)
\end{align}
for all $k\in\Bbb Z$ and $\sigma\in[0,\frac{99}{100}]$,  $s\in [2^{2k_0-1}, 2^{2k_0+1})$, $k_0\in\Bbb Z$.

Then applying bilinear Littlewood-Paley decomposition to (\ref{Hahaob2}), (\ref{uyg}) and (\ref{Hahaob}) yield for any $\sigma\in[0,\frac{99}{100}]$
\begin{align}\label{sdffmna}
(1+s 2^{2k})^{30}2^k\|P_k\left((d\mathcal{P})(e_l)-\chi^{\infty}_l\right)\|_{L^{\infty}_tL^2_x} \lesssim 2^{-\sigma k}\gamma_{k}(\sigma)\lesssim 2^{-\sigma k}\eta_{k}(\sigma),
\end{align}
where we applied (\ref{kjln6a}) in the last inequality.
Then by the identity (\ref{6.5a}),  (\ref{sdffmna}) and bilinear Littlewood-Paley decomposition, one has when $s=0$
\begin{align}
& \left\|P_{k}\partial_iv\right\|_{L^{\infty}_tL^2_x}\lesssim\|P_{k}\psi_i\|_{L^{\infty}_tL^2_x}\|P_{\le k-4}[(d\mathcal{P})(e_l)-\chi^{\infty}_l]\|_{L^{\infty}}\nonumber\\
&+\|P_{k}[(d\mathcal{P})(e_l)-\chi^{\infty}_l]\|_{L^{\infty}_tL^2_x}\sum_{k_1\le k-4}2^{k_1}\|P_{k_1}\psi_i\|_{L^{\infty}_tL^2_x}\nonumber\\
&+2^{k}\sum_{|k_1-k_2|\le 8,k_1,k_2\ge k-4}\|P_{k_1}\psi_i\|_{L^{\infty}_tL^2_x}\|P_{k}[(d\mathcal{P})(e_l)-\chi^{\infty}_l]\|_{L^{\infty}_tL^2_x}\nonumber\\
&\lesssim 2^{-\sigma k}\eta_{k}(\sigma)+  2^{-\sigma k}\eta_{k}(\frac{3}{8}) \eta_{k}(\sigma-\frac{3}{8}).\label{wesaa}
\end{align}
Thus  (\ref{wesaa}) gives $\gamma^{(1)}_k(\sigma)\lesssim \eta^{(1)}_k(\sigma)$ for $\sigma\in(\frac{99}{100},\frac{5}{4}]$. Combining Step 1.1 and Step 1.2, we have  proved (\ref{hbyu78}) and (\ref{hbyu77}).

{\bf Step 2.1. Bounds of connections for $\sigma\in[0,\frac{99}{100}]$.}
Applying Proposition \ref{mei} gives
\begin{align}\label{JIO}
(1+2^{2k}s)^{31}2^{\sigma k+k}\|P_{k}v \|_{L^{\infty}_tL^2_x} \lesssim \gamma_{k}(\sigma)
\end{align}
for all $\sigma\in[0,\frac{99}{100}]$ and $s\ge 0$.
Then using the identity $\psi^{l}_i=(d\mathcal{P}e_l)\cdot \partial_iv$  and the bound (\ref{lkmna}), by (\ref{JIO}) we infer from the bilinear Littlewood-Paley decomposition that
\begin{align*}
2^{\sigma k}\|P_{k}\phi_i\|_{L^{\infty}_tL^2_x}\lesssim (1+s2^{2k})^{-31}\eta_{k}(\sigma)
\end{align*}
for all $k\in\Bbb Z$ and $\sigma\in[0,\frac{99}{100}]$.
 Then, by (\ref{JIO1}), applying bilinear Littlewood-Paley decomposition again gives
\begin{align}\label{kjiu89P}
\|P_{k}(\phi_s\diamond\phi_i)\|_{L^{\infty}_tL^2_x}\lesssim 2^{-\sigma k}(1+2^{2j+2k})^{-31}\left(2^{-j+k}\eta_{-j}\eta_{k}(\sigma)+2^{-2j}\eta_{-j}\eta_{-j}(\sigma)\right)
\end{align}
for $j+k\ge 0$, $s\in[2^{2j-1},2^{2j+1})$, $\sigma\in[0,\frac{99}{100}]$.

Recall that Section \ref{VNM} shows $A_i$ can be written in the form of
\begin{align}\label{XZ98}
A_i(s)=\int^{\infty}_s(\phi_s \diamond \phi_i)\mathcal{G}ds'.
\end{align}
Direct calculations and Lemma \ref{kuai} imply $\mathcal{G}$ modulate a constant part $\Gamma^{\infty}$ satisfies
\begin{align}\label{Nh}
2^{k}\|P_k\mathcal{\widetilde{G}}\|_{L^{\infty}_tL^{2}_x}\lesssim_{M_1} (1+s2^{2k})^{-M_1}\epsilon_1,
\end{align}
for all $M_1\in\Bbb Z_+$, $k\in\Bbb Z$ and $s\ge 0$.
Then applying  bilinear Littlewood-Paley decomposition and (\ref{kjiu89P}) leads to
\begin{align}
\int^{\infty}_s\|P_{k}((\phi_i\diamond\phi_s)\mathcal{G})\|_{L^{\infty}_tL^{2}_x}ds'\lesssim 2^{-\sigma k}(1+s2^{2k})^{-30}\eta_{k}(\sigma)\eta_{-j}
\end{align}
for all $j\in\Bbb Z$, $s\in[2^{2j-1},2^{2j+1})$ and $k+j\ge 0$. Moreover, one has
\begin{align*}
\int^{\infty}_s\|P_{k}((\phi_i\diamond \phi_s)\mathcal{G})\|_{L^{\infty}_tL^{2}_x}ds'\lesssim 2^{-\sigma k}\sum_{k\le l\le -j}\eta_{l}(\sigma)\eta_{l}
\end{align*}
for any $j\in\Bbb Z$, $s\in[2^{2j-1},2^{2j+1})$ and $k+j\le 0$.

Hence, (\ref{shujiu1}) is done.

{\bf Step 2.2. $\sigma\in[\frac{99}{100},\frac{5}{4}]$.}
Recall that  (\ref{hbyu77}), (\ref{hbyu78}) has given
\begin{align*}
\gamma_{k}(\sigma) \lesssim \eta_{k}(\sigma), \mbox{ } \gamma^{(1)}_{k}(\sigma) \lesssim \eta^{(1)}_{k}(\sigma).
\end{align*}
Now, we are ready to estimate $A_x$ for  $\sigma\in[\frac{99}{100},\frac{5}{4}]$. By the identity  $\psi^{l}_i=(d\mathcal{P}e_l)\cdot \partial_iv$,  the bound (\ref{sdffmna}) and
\begin{align*}
\|P_{k}\partial_i v\|_{L^{\infty}_tL^2_x}\lesssim 2^{-\sigma k}(1+s2^{2k})^{-30}\eta^{(1)}_{k}(\sigma),
\end{align*}
one obtains by bilinear Littlewood-Paley decomposition that
\begin{align} \label{1zNb}
\|P_{k}\phi_i\|_{L^{\infty}_tL^2_x}\lesssim 2^{-\sigma k}(1+s2^{2k})^{-30}\eta^{(1)}_{k}(\sigma),
\end{align}
for any $\sigma \in[\frac{99}{100},\frac{5}{4}]$, $k\in\Bbb Z$, $s\ge 0$.  For any $\sigma \in[\frac{99}{100},\frac{5}{4}]$, the proof of Proposition \ref{uyhgvbL} yields
the bound
\begin{align*}
\|P_{k}\partial_s v\|_{L^{\infty}_tL^2_x}\lesssim 2^{-\sigma k+k}\left[1_{k+j\ge 0}(1+s2^{2k})^{-30}\eta^{(1)}_{k}(\sigma)+1_{k+j\le 0}\sum_{k\le l\le -j}\eta^{(1)}_{l}(\sigma)\eta_l\right],
\end{align*}
which combined with (\ref{sdffmna})  instead gives
\begin{align}\label{zNb}
\|P_{k}\phi_s \|_{L^{\infty}_tL^2_x}\lesssim 2^{-\sigma k+k}\left[1_{k+j\ge 0}(1+s2^{2k})^{-30}\eta^{(1)}_{k}(\sigma)+1_{k+j\le 0}\sum_{k\le l\le -j}\eta^{(1)}_{l}(\sigma)\eta_l\right],
\end{align}
for any $k\in\Bbb Z$, $s\in [2^{2j-1}, 2^{2j+1})$, $j\in\Bbb Z$., $\sigma \in[\frac{99}{100},\frac{5}{4}]$.

In order to apply (\ref{XZ98}), we also need to improve the bound of $\mathcal{\widetilde{G}}$ stated in (\ref{Nh}).
Recall the formula
\begin{align*}
\mathcal{G}&:= \langle {\bf R}(e_{j_0},e_{j_1})(e_{j_2}),e_{j_3}\rangle =\Gamma^{\infty}-\int^{\infty}_s\psi^{p}_s (\widetilde{\nabla } {\bf R})(e_{p};e_{j_0},...,e_{j_3})ds'.
\end{align*}
By Lemma \ref{kuai} and the direct calculations (see Step 1.1  for instance) we have the bounds:
\begin{align}\label{chui}
2^{k}\|P_{k}\left(  (\widetilde{\nabla} {\bf R})(e_{l};e_{j_0},...,e_{j_3})-\Gamma^{\infty,(1)}_l\right)\|_{L^{\infty}_tL^{2}_x}\lesssim_{M} (1+s2^{2k})^{-M}
\end{align}
for all $M\in\Bbb Z_+$, $k\in\Bbb Z$.
Hence one obtains by (\ref{uyg}) and bilinear Littlewood-Paley decomposition that
\begin{align}\label{98uh}
&2^{k}\|P_{k}\left(\mathcal{G}-\Gamma^{\infty}\right)\|_{L^{\infty}_tL^{2}_x}\lesssim 2^{-\sigma k}(1+s2^{2k})^{-30}\gamma_{k}(\sigma)
\end{align}
for any $k\in\Bbb Z$, $\sigma\in[0,\frac{99}{100}]$.

Then by (\ref{98uh}), (\ref{zNb}), (\ref{1zNb}) using trilinear Littlewood-Paley decomposition as Step 2.1 with additional modifications in the $\rm Low\times \rm High$ interaction of $P_k((\phi_s\diamond\phi_i)\mathcal{\widetilde{G}})$  (see Proposition \ref{4.2} for instance) we conclude that
\begin{align*}
&\int^{\infty}_s\|P_{k}\left((\phi_s\diamond\phi_i)\mathcal{G} \right)\|_{L^{\infty}_tL^{2}_x}ds'\lesssim 2^{-\sigma k}(1+s2^{2k})^{-29}\eta^{(1)}_{k,s}(\sigma)
\end{align*}
for any $k\in\Bbb Z$, $\sigma\in(\frac{99}{100},\frac{5}{4}]$.
Thus, (\ref{shujiu2}) is done.
\end{proof}

\begin{Lemma}\label{Ncccmkijnl}
Let $d=2$, $j,M\in\Bbb Z_+$. And let $u \in \mathcal{H}_Q(T)$, and $v(s,t,x)$ be the solution of heat flow with data $u(t,x)$. Denote $\{\phi_i\}$ the differential fields of
$v$ under the caloric gauge. Assume that $\sigma\in[0,1+\frac{j}{4}]$, and $\{\eta_{k}(\sigma)\}$ are frequency envelopes of order $\frac{1}{2^{j}}\delta$ such that for all $i=1,2$, $k\in\Bbb Z$,
\begin{align}
2^{\sigma k}\|\phi_i(\upharpoonright_{s=0})\|_{L^{\infty}_tL^2_x}\le \eta_k(\sigma).
\end{align}
Then, given   $j,M\in\Bbb Z_+$,  there exists a sufficiently small constant $\epsilon_j$ depending only on $M,j$ such that
if $\|\nabla u \|_{L^{\infty}_tL^2_x}\le  \epsilon_j$, then  we have
\begin{align}
\gamma^{(j)}_k(\sigma)&\lesssim \eta^{(j)}_{k}(\sigma)\label{asshujiu4}
\end{align}
for  $\sigma\in[0,\frac{j}{4}+1]$, $k\in\Bbb Z$.
Moreover, for $l=0,...,j$,  we have
\begin{align}
(1+s2^{2k})^{M}\|P_k\phi_i\|_{L^{\infty}_tL^2_x}&\lesssim 2^{-\sigma k}\eta^{(j)}_{k}(\sigma), \mbox{ }i=1,2,  \mbox{ }\sigma\in[0,1+j/4)\label{qashujiu3}\\
(1+s2^{2k})^{M}\|P_k [\widetilde{d\mathcal{P}}]^{(l)}\|_{L^{\infty}_tL^2_x}&\lesssim 2^{-\sigma k}\eta^{(j-l)}_{k}(\sigma), \mbox{ }\sigma\in[0,1+(j-l)/4) \label{qbshujiu3}\\
(1+s2^{2k})^{M}\|P_k\widetilde{\mathcal{G}}^{(l)}\|_{L^{\infty}_tL^2_x}&\lesssim 2^{-\sigma k}\eta^{(j-l)}_{k}(\sigma),\mbox{ }\sigma\in[0,1+(j-l)/4) \label{qcshujiu3}\\
(1+s2^{2k})^{M}\|P_kA_i\|_{L^{\infty}_tL^2_x}&\lesssim 2^{-\sigma k}\eta^{(j)}_{k,s}(\sigma),\mbox{ }\sigma\in[0,1+j/4) \label{shujiu3}
\end{align}
where we denote  $[d\mathcal{P}]^{(l)} =({\bf D}^{l} {d\mathcal{P}})(\underbrace{e,...,e}_l;e)$, and $[\widetilde{d\mathcal{P}}]^{(l)} =  [d\mathcal{P}]^{(l)} -\lim\limits_{s\to\infty} [d\mathcal{P}]^{(l)} $.
\end{Lemma}
\begin{proof}
The case  $\sigma\in[0,5/4]$ has been done in Proposition \ref{7.2}.
Let $\sigma\in[1+j/4 ,1+(j+1)/4]$. The general case of (\ref{asshujiu4}) follows by iteration.
The highest covariant  derivative order of $\mathcal{G}$ and $d\mathcal{P}(e)$ one needs for the $j$-th iteration is $j+1$, and it suffices to take the decay power $M+2+j$, i.e.,
\begin{align*}
 \left\|  \partial^{L+1}_x \mathcal{G}^{(j+1)}  \right\|_{L_t^\infty L_x^2}&\lesssim_{L,j}  \epsilon s^{-\frac{L}{2}}, \forall L\in [0,M+2+j] \\
 \left\| \partial^{L+1}_x [d\mathcal{P}]^{(j+1)}  \right\|_{L_t^\infty L_x^2}&\lesssim_{L,j} \epsilon s^{-\frac{L}{2}},  \forall L\in [0,M+2+j],
\end{align*}
where we denote $\mathcal{G}^{(k)}=(\widetilde{\nabla}^{k} {\bf R})(\underbrace{e,...,e}_k;\underbrace{e,...,e}_4)$, and   $[d\mathcal{P}]^{(k)} =({\bf D}^{k} {d\mathcal{P}})(\underbrace{e,...,e}_k;e)$. These decay estimates are easy to check by using Lemma \ref{kuai}.
If these decay estimates along heat direction are verified,
then (\ref{asshujiu4}) follows by repeating  the arguments of Step 2 in Proposition \ref{7.2} for $j$ times. Moreover, the left (\ref{qashujiu3}), (\ref{qbshujiu3}), (\ref{qcshujiu3}), (\ref{shujiu3}) follow along  the same path by applying dynamical separation and iterations  for $j$ times.
\end{proof}

Similar to Proposition \ref{7.2}, one also have
\begin{Corollary}\label{JkL}
Let $d=2$. And let $v_0 \in\mathcal{H}_Q$ satisfy
\begin{align*}
\|\partial_x v_0 \|_{L^2_x}= \epsilon_1\ll 1.
\end{align*}
\begin{itemize}
  \item Let $\{d_{k}(\sigma)\}$ with $k\in\Bbb Z$, $\sigma\in[0,\frac{5}{4}]$ be frequency envelopes of order $\frac{1}{2}\delta$ satisfying
\begin{align}\label{sddNMl0mnm}
2^{\sigma k+k}\|P_{k}v_0\|_{L^2_x}\le d_k(\sigma).
\end{align}
Denote $v(s,x)$ the solution to heat flow with data $v_0$, and denote $\{\phi_i\}$ the corresponding  differential fields under the caloric gauge.
Then we have
\begin{align}\label{sdNMl0mnm}
\|\phi_i(\upharpoonright_{s=0})\|_{L^2_x}\le 2^{-\sigma k} d^{(1)}_k(\sigma).
\end{align}
  \item Let $j\in\Bbb Z_+$. Assume that $\{d_{k}(\sigma)\}$ with $k\in\Bbb Z$, $\sigma\in[0,1+\frac{1}{4}j]$ are frequency envelopes of order $\frac{1}{2^{j}}\delta$. Then for $\epsilon_1$ sufficiently small depending only on $j$,  similar results hold with  $d^{(1)}_k(\sigma)$ replaced by $d^{(j)}_k(\sigma)$.
\end{itemize}
\end{Corollary}
\begin{proof}
Using (\ref{sddNMl0mnm})  Proposition \ref{4.2} and Proposition \ref{uyhgvbL} show for $\sigma\in[0,\frac{5}{4})$
\begin{align}
(1+s2^{2k})^{30}2^{k+\sigma k}\|P_{k}v\|_{L^2_x}&\lesssim d^{(1)}_{k}(\sigma)\label{sdhbyu77}.
\end{align}
Let's first consider $\sigma\in[0,\frac{99}{100}]$.
Recall
\begin{align}
2^{k} \|P_k(d\mathcal{P}(e_l)-\chi^{\infty}_l)\|_{L^2_x}&\lesssim \epsilon_1(1+s2^{2k})^{-29}.\label{shujxc1}
\end{align}
Then by the identity $\psi^{l}_i=d\mathcal{P}(e_{l})\cdot \partial_i v$,  (\ref{sdhbyu77}), (\ref{shujxc1}), we obtain from the  bilinear Litttlewood-Paley decomposition
\begin{align*}
& \|P_k(d\mathcal{P}(e_{l})\cdot \partial_i v)\|_{L^2_x}\\
 &\lesssim  2^{-\sigma k}d_k(\sigma)\|P_{\le k-4}d\mathcal{P}(e_{l})\|_{L^{\infty}}+
2^{k}\sum_{k_1\ge k-4,|k_1-k_2|\le 8} 2^{-\sigma k_1}d_{k_1}(\sigma)\|P_{k_2}d\mathcal{P}(e_{l})\|_{L^{\infty}_tL^2_x}\\
&+ \sum_{|k-k_2|\le 4} \|P_{k_2}(d\mathcal{P}(e_l))\|_{L^2_x}\sum_{k_1\le k-4}2^{k_1-\sigma k_1} d_{k_1}(\sigma)
\end{align*}
that for any $\sigma\in[0,\frac{99}{100}]$,  $k\in\Bbb Z$,
\begin{align*}
 \|P_k\psi_i(\upharpoonright_{s=0})\|_{L^2_x}&\lesssim  2^{-\sigma k}d_k(\sigma).
\end{align*}
Using this $\sigma\in[0,\frac{99}{100}]$ bound and similar arguments as before  one can improve (\ref{shujxc1}) to
\begin{align*}
2^{k} \|P_k(d\mathcal{P}(e_l)-\chi^{\infty}_l)\|_{L^2_x}&\lesssim 2^{-\sigma k} d_{k}(\sigma)(1+s2^{2k})^{-29}, \mbox{ }\sigma\in[0,\frac{99}{100}],
\end{align*}
and thus giving (\ref{sdNMl0mnm}).  The second item claimed in our corollary follows by similar arguments and Lemma \ref{Ncccmkijnl}.
\end{proof}

\begin{Lemma}\label{FFF}
Let $u\in\mathcal{H}_{Q}(T)$ solve $SMF$. And let $v(s,t,x)$ be the solution of heat flow  (\ref{1OOT})  with initial data $u(t,x)$. Then given $L\in\Bbb Z_+$, $L\ge 200$, for any $0\le \sigma\le 2L$, there exist  constants $\epsilon_{L}>0$, $C_{L}>0$, $C_{L,T}$, such that if $\|\partial_xu\|_{L^{\infty}_tL^2_x}\le \epsilon_{L}\ll 1$,
then for any $s\ge 0$, $i=1,2$, $\rho=0,1$, $m=0,1,...,L$,
 \begin{align}
 \|\partial^{\rho}_{t}\partial^{m}_x(v-Q)\|_{L^{\infty}_tH^{L}_x}&\le  C_{L,T}(s+1)^{-\frac{m}{2}}\label{RTg1}\\
(2^{-\frac{1}{2}k}1_{k\le 0}+2^{\sigma k})\| P_{k} \phi_i\|_{L^{\infty}_tL^2_x}&\le  C_L(2^{2k}s+1)^{-30}\label{RTg2}\\
(2^{-\frac{1}{2}k}1_{k\le 0}+2^{\sigma k})\| P_{k} A_i\|_{L^{\infty}_tL^2_x}&\le C_L(2^{2k}s+1)^{-28}\label{RTg3}\\
 2^{m k}\| P_{k} \partial_t\phi_i\|_{L^{\infty}_tL^2_x}&\le  C_{L,T}(2^{2k}s+1)^{-25}\label{RTg4}\\
 2^{m k}\| P_{k} \partial_tA_i\|_{L^{\infty}_tL^2_x}&\le  C_{L.T}(2^{2k}s+1)^{-25}.\label{RTg5}
\end{align}
\end{Lemma}
\begin{proof}
Fix arbitrary $L\in\Bbb N$, $L\ge 200$. Let $\lambda_{k}(\sigma)$ be the frequency envelope
 \begin{align*}
&\lambda_{k}(\sigma):=\sup_{k'\in\Bbb Z}2^{-\frac{1}{2^j}\delta|k-k'|}2^{\sigma k'}\|P_{k'}(u-Q)\|_{L^{\infty}_tL^2_x}.
\end{align*}
for $\sigma\in{\Bbb I}_{j}\cap [0,2L]$. And define
\begin{align*}
\widetilde{B}_{j,\sigma,K}(S)=\sup_{k\in\Bbb Z,s\in[0,S)} [\lambda^{(j)}_k(\sigma)]^{-1}(1+s2^{2k})^{K}2^{\sigma k}\|P_{k }v\|_{L^{\infty}_tL^2_{x}},
\end{align*}
for $\sigma\in{\Bbb I}_{j}\cap [0,2L]$, $j\in\Bbb N$, $K\in \Bbb Z_+$.
By Lemma \ref{KJm1} and the fact  $\{\lambda^{(j)}_k(\sigma)\}$ are frequency envelopes,  $\widetilde{B}_{j,\sigma,K}(S)$ is well-defined for $S\ge 0$ and continuous in $S$ with
$\mathop {\lim }\limits_{S \to 0} \widetilde{B}_{j,\sigma,K}(S)=1$. Then applying Proposition \ref{4.2},  Proposition  \ref{uyhgvbL}, Proposition  \ref{Jth} and their arguments,
we get $\widetilde{B}_{j,\sigma,30}(S)\lesssim 1$. Then applying similar arguments of Corollary \ref{JkL} one obtains
\begin{align*}
2^{\sigma k} \| P_{k} \phi_i(s=0)|_{L^{\infty}_tL^2_x}&\lesssim \lambda^{(j)}_k(\sigma)
\end{align*}
for $\sigma\in{\Bbb I}_{j}\cap [0,2L]$.
Then similar arguments of Proposition \ref{7.2} yield for any $\sigma\in  [0,L]$
\begin{align*}
 2^{\sigma k} \| P_{k} \phi_i\|_{L^{\infty}_tL^2_x}&\lesssim C_L(2^{2k}s+1)^{-30} \\
 2^{\sigma k} \| P_{k} A_i\|_{L^{\infty}_tL^2_x}&\lesssim C_L(2^{2k}s+1)^{-28}
\end{align*}
which verifies  half of (\ref{RTg2})-(\ref{RTg3}). Moreover, one has
\begin{align}\label{Mdfghjk}
2^{\sigma k} \|P_k(d\mathcal{P}e)\|_{L^{\infty}_tL^2_x}\lesssim  2^{-k}(1+s2^{2k})^{-30}
\end{align}
for $\sigma\in [0,2L]$.
Let's check the left half  of (\ref{RTg2})-(\ref{RTg3}).
Recall the bounds
\begin{align*}
  \|P_kv\|_{L^2_x}&\lesssim 2^{-\sigma k} \lambda^{(0)}_k(\sigma) (1+s2^{2k})^{-30}
\end{align*}
for $\sigma\in [0,\frac{99}{100}]$.
Then (\ref{shujxc1}) and bilinear Littlewood-Paley decomposition show that
\begin{align*}
&\|P_k(\phi_i)\|_{L^{\infty}_tL^2_x}=\sum^{2n}_{l}\|P_k(d\mathcal{P}(e_{l})\cdot \partial_i v)\|_{L^{\infty}_tL^2_x}\\
 &\lesssim  2^{ k}\lambda^{(0)}_k(0)(1+s^{2k})^{-30}\|P_{\le k-4}d\mathcal{P}(e_{l})\|_{L^{\infty}}+
2^{k}\sum_{k_1\ge k-4,|k_1-k_2|\le 8} 2^{\frac{1}{2}k_1} \lambda^{(0)}_{k_1}(\frac{1}{2}) \|P_{k_2}d\mathcal{P}(e_{l})\|_{L^{\infty}_tL^2_x}\\
&+ \sum_{|k-k_2|\le 4} \|P_{k_2}(d\mathcal{P}(e_l))\|_{L^2_x}\sum_{k_1\le k-4}2^{2k_1} \lambda^{(0)}_{k_1}(0) \\
&\lesssim  2^{\frac{k}{2}}(1+s^{2k})^{-30}
\end{align*}
for any $i=1,2$, $k\le 0$. Similar arguments give
\begin{align}
&\|P_k(A_i)\|_{L^{\infty}_tL^2_x}\lesssim  2^{\frac{k}{2}}(1+s2^{2k})^{-28}
\end{align}
for any $i=1,2$, $k\le 0$. Hence, (\ref{RTg2}) and (\ref{RTg2}) have been done.

Since $\partial_t v=J(v)(\sum_{i=1,2}D_i\phi)$ at $s=0$, we observe from $u\in\mathcal{H}_Q(T)$ that
\begin{align*}
 \|\partial_t v(\upharpoonright_{s=0}) \|_{L^{\infty}_tH^l_x}\le  C_{l,T}, \mbox{ }\forall \mbox{ }l\in\Bbb N.
\end{align*}
Thus using the smoothing estimates of heat semigroups and applying $\partial_t$ to (\ref{1OOT}), one obtains (\ref{RTg1}).
From (\ref{RTg1}), (\ref{shujxc1}), (\ref{Mdfghjk}) and the identity
\begin{align*}
 \partial^{l}_x\phi^{a}_t =\sum^{l}_{l_1=0}\partial^{l_1}_x(d\mathcal{P}(e_a))\partial^{l-l_1}_x(\partial_t v)
\end{align*}
we get
\begin{align*}
2^{mk} \|P_{k}\phi_t \|_{L^{\infty}_tL^2_x}\lesssim (1+s2^{2k})^{-28}   \mbox{ }\forall \mbox{ }0\le m\le L,
\end{align*}
which further gives bounds of $\|P_{k}A_t\|_{L^{\infty}_tL^2_x}$.
Then applying similar bounds of $A_i\phi_i, A_t\phi_i$ and the identity $\partial_t\phi_i=-A_t\phi_i+D_i\phi_t$, we obtain (\ref{RTg4}).
For (\ref{RTg5}), we use $\phi_s=D_i\phi_i$ and
\begin{align*}
 |\partial^{l}_x \partial_tA_i |\le\sum^{l}_{l_1=0} \int^{\infty}_s|D^{l-l_1}_xD_t\phi_i||D^{l_1}_x \phi_s|ds'+\int^{\infty}_s|D^{l-l_1}_x\phi_i||D^{l_1}_xD_t \phi_s|ds'.
\end{align*}

\end{proof}

\subsection{Additional decay estimates for dynamical caloric gauge}

\begin{Proposition}\label{7.2q}
Let $d=2$. And let $u \in\mathcal{H}_Q(T)$ be solution of SMF.
Denote $v(s,t,x)$ the solution to heat flow with data $u(t,x)$, and denote $\{\phi_i\}^{2}_{i=0}$ the corresponding  differential fields under the caloric gauge.
Assume that $\{\beta_{k}(\sigma)\}$ is a frequency envelope of order $\delta$ such that for all $i=1,2$, $k\in\Bbb Z$,
\begin{align}\label{avNMl0mnm}
2^{\sigma k}\|\phi_i(\upharpoonright_{s=0})\|_{L^{\infty}_tL^{2}_x\cap L^{4}_{t,x}}\le \beta_k(\sigma).
\end{align}
\begin{itemize}
  \item There exists a sufficiently small constant $\epsilon>0$ such that if
\begin{align}\label{bvNMl0mnm}
\sum_{k\in\Bbb Z} |\beta_k(0)|^2<\epsilon,
\end{align}
then we have for any $l\in\Bbb N$
\begin{align}
&\|P_{k}\widetilde{\mathcal{G}}^{(l)}\|_{L^4\cap L^{\infty}_tL^2_x}\lesssim_l 2^{-\sigma k-k} \beta_{k}(\sigma)(2+s2^{2k})^{-30}\label{3.16ab}\\
&\|P_{k}\phi_s \|_{L^4\cap L^{\infty}_tL^2_x}\lesssim 2^{-\sigma k+k}\left[1_{k+j\ge 0}(1+s2^{2k})^{-30}\beta_{k}(\sigma)+1_{k+j\le 0}\sum_{k\le l\le -j}\beta_{l}(\sigma)\beta_l\right]\label{wwbbzNb}\\
&(1+s2^{2k})^{29}2^{\sigma k+k}\|P_{k}(d\mathcal{P}(e))\|_{L^{4}}\lesssim \beta_{k}(\sigma)\label{We1b}
\end{align}
for any $\sigma\in[0,\frac{99}{100}]$, $s\in [2^{2j-1}, 2^{2j+1})$,  $j,k\in\Bbb Z$.
Furthermore, assume that for  $\sigma\in [0,\frac{5}{4}]$, $\{\beta_{k}(\sigma)\}$ is a frequency envelope of order $\frac{1}{2}\delta$ such that for all $i=1,2$, $k\in\Bbb Z$,
(\ref{avNMl0mnm}), (\ref{bvNMl0mnm}) hold.
Then for any $\sigma\in[0,\frac{5}{4}]$, $k\in\Bbb Z$, one has
\begin{align}
(1+s2^{2k})^{27}2^{\sigma k}\|P_kA_i\|_{L^{\infty}_tL^2_x}&\lesssim \beta^{(1)}_{k,s}(\sigma).\label{qashujiu2}
\end{align}
  \item If $\{\beta_{k}(\sigma)\}$ is a frequency envelope of order $\frac{1}{2^{j}}\delta$, then similar results hold for $\sigma\in[0,\frac{1}{4}j+1)$ and $\epsilon$ sufficiently small depending only on $j\in\Bbb Z_+$. (See Prop. 1.4 for instance )
\end{itemize}
\end{Proposition}
\begin{proof}
By Proposition \ref{7.2} and its proof, we have
\begin{align}
&(1+s2^{2k})^{31}2^{\sigma k}2^{k}\|P_kv\|_{L^{\infty}_tL^2_x} \lesssim \beta_{k}(\sigma).\label{asdjka}\\
&(1+s2^{2k})^{30}2^{\sigma k}2^{k}\|P_kS(v)\|_{L^{\infty}_tL^2_x} \lesssim \beta_{k}(\sigma).\label{asdjkl}\\
&\|P_{k}\phi_s \|_{L^{\infty}_tL^2_x}\lesssim 2^{-\sigma k+k}\left[1_{k+j\ge 0}(1+s2^{2k})^{-30}\beta_{k}(\sigma)+1_{k+j\le 0}\sum_{k\le l\le -j}\beta_{l}(\sigma)\beta_l\right]\label{bbzNb}\\
&(1+s 2^{2k})^{30}2^k\|P_k\left((d\mathcal{P})(e_l)-\chi^{\infty}_l\right)\|_{L^2_x}  \lesssim 2^{-\sigma k}\beta_{k}(\sigma)\label{bsdffmna}\\
&\sum_{i=1,2}(1+s2^{2k})^{29}2^{\sigma k}\|P_kA_i\|_{L^{\infty}_tL^2_x} \lesssim \beta^{(0)}_{k,s}(\sigma).\label{ashujiu1}
\end{align}
for any $\sigma\in[0,\frac{99}{100}]$, $j,k\in\Bbb Z$, $s\in[2^{2j-1}, 2^{2j+1})$,  if $\{\beta_k(\sigma)\}$ is a frequency envelope of order $\delta$.
And  Proposition \ref{7.2} and its proof also give similar results as above
for $\sigma\in[0,\frac{5}{4}]$, if $\{\beta_k(\sigma)\}$ is a frequency envelope of order $\frac{1}{2}\delta$.

{\bf Step 1.} When $s=0$, using $\partial_i v=\psi^{l}_i\mathcal{P}(e_{l})$, we get from the bilinear Littlewood-Paley decomposition
\begin{align}
&\|P_k(fg)\|_{L^4}\lesssim  \sum_{|k-k_2|\le 4} \|P_{\le k-4}f\|_{L^{\infty}_{t,x}}\|P_{k_2}g\|_{L^4_xL^{\infty}_t}
+2^{k} \sum_{k_1,k_2\ge k-4, |k_1-k_2|\le 8} \|P_{k_1}f\|_{L^{\infty}_tL^2_x}\|P_{k_2}g\|_{L^{4}}\nonumber\\
&+\sum_{k_2\le k-4, |k_1-k|\le 4}2^{\frac{1}{2}k_1} \|P_{k_1}f\|_{L^{\infty}_{t}L^{2}_x}2^{\frac{1}{2}k_2}\|P_{k_2}g\|_{L^{4}}\label{Dfr}
\end{align}
and (\ref{avNMl0mnm}), (\ref{bsdffmna}) that
\begin{align}\label{xdtyv}
\|P_k(\partial_i v)\|_{L^4}\lesssim  2^{-\sigma k}\beta_{k}(\sigma)
\end{align}
for $s=0$ and any $\sigma\in[0,\frac{99}{100}]$, $k\in\Bbb Z$.
Then we turn to study  the heat flow equation
 (\ref{Heat}) to obtain
\begin{align}\label{Fcv7yhbn}
(1+s2^{2k})^{30}\|P_k(\partial_i v)\|_{L^4}\lesssim  2^{-\sigma k}\beta_{k}(\sigma)
\end{align}
 for any $s\ge 0$, $\sigma\in[0,\frac{99}{100}]$, $k\in\Bbb Z$.
 In fact, define
\begin{align*}
Z_1(S):=\sup_{s\in[0,S),k\in\Bbb Z} \frac{1}{\beta_{k}(\sigma)}2^{\sigma k} (1+s2^{2k})^{30}\|P_k(\partial_i v)\|_{L^4}.
\end{align*}
$Z_1(S)$ is  well-defined, continuous and tends to $1$ as $S\to 0$ by (\ref{xdtyv}).
Using the following trilinear Littlewood-Paley decomposition
\begin{align}
&2^{k}\|P_{k }S(v)(\partial_x v,\partial_x v)\|_{L^{4}_{t,x}}\lesssim 2^{k}\widetilde{\beta}_k \sum_{k_1\le k}\widetilde{\beta}_{k_1}2^{k_1}+\sum_{k_2\ge k}2^{-2|k-k_2|}2^{2k_2}\widetilde{\beta}^2_{k_2}\nonumber\\
&+ 2^{\frac{1}{2}k}\widetilde{\alpha}_{k}(\sum_{k_1\le k}2^{k_1}\widetilde{\beta}_{k_1})^2+\sum_{k_2\ge k}2^{2k-k_2}\widetilde{\alpha}_{k_2}\widetilde{\beta}_{k_2}\sum_{k_1\le k_2}2^{k_1}\widetilde{\beta}_{k_1}.
\end{align}
where we denote
\begin{align*}
\widetilde{\beta}_k&=\sum_{|k'-k|\le 30}2^{ k'}\|P_{k'}v\|_{L^{\infty}_t{L^2_x}\bigcap L^{4}_{t,x}}; \mbox{ } \widetilde{\alpha}_k=\sum_{|k'-k|\le 30}2^{ k'}\|  P_{k'}(S(v))\|_{L^{\infty}_t{L^2_x}},
\end{align*}
using similar arguments as Proposition 3.2 we obtain by (\ref{asdjkl}), (\ref{asdjka}) that
\begin{align*}
Z_1(S)\lesssim 1+\epsilon Z^2_1(S)
\end{align*}
for any $S\ge 0$. Then $Z_{1}(S)\lesssim 1$ since $\lim\limits_{S\to 0}Z_{1}(S)=1$.
Thus (\ref{Fcv7yhbn}) follows.

{\bf Step 2.} With (\ref{Fcv7yhbn}) in hand, using the heat flow equation one obtains bounds for $\|P_{k}\partial_s v\|_{L^4}$. Then the bound of  $\|P_{k}\phi_s\|_{L^4}$ follows by the bilinear Littlewood-Paley decomposition (\ref{Dfr}) and (\ref{bsdffmna}). By performing dynamical separation for $d\mathcal{P}(e)$, we get bounds of  $\|P_{k}(d\mathcal{P}(e))\|_{L^4}$ from  $\|P_{k}\phi_s\|_{L^4}$ and $\|P_{k}({\bf D}d\mathcal{P}(e;e))\|_{L^{\infty}_tL^2_x}$. And then one obtains bounds of $\|P_{k}(\mathcal{G})\|_{L^4}$, $\|P_{k}\phi_i\|_{L^4}$ for all $s\ge 0$,  which further yields $\|A_i\|_{L^4}$ for any  $s\ge 0$.
See Proposition \ref{7.2} for the details.
\end{proof}

\section*{Outline of Proof before Iteration }

One of the ingredients of proof before iteration is the framework due to \cite{BIKT}. The other main  ingredient is the decomposition of curvatures mentioned in Step 2 of Section 1.4. And in the technic level  we need bootstrap assumption of $\|P_{k}\widetilde{\mathcal{G}}^{(1)}\|_{L^4_xL^{\infty}_{t}(T)}$ and ideas to improve this bound, see Step 4.1 and Step 4.2 of Section 1.4 for instance.

We outline the framework of \cite{BIKT} for reader's convenience. Since the gauged equation is now not self-closed due to the curvature terms, several key new ideas as mentioned above shall be  used. But to leave the reader a whole picture, we just sketch the framework of \cite{BIKT} rather than presenting all technically complex issues.

The proof is a bootstrap argument. Let $\sigma\in [0,\frac{99}{100})$ be given.  Given $\mathcal{L}\in \Bbb Z_+$, $Q\in \mathcal{N}$, let $T\in(0,2^{2\mathcal{L}}]$. Assume that $\{c_k\}$ is an $\epsilon_0$-frequency envelope of order $\delta$ and $\{c_k(\sigma)\}$ is another frequency envelope of order $\delta$. Let $u_0$ be the initial data  of SMF which satisfies
\begin{align}\label{zFDc65}
\|P_{k}\nabla u_0\|_{L^2_x}\le c_k(\widetilde{\sigma})2^{-\widetilde{\sigma}  k},\mbox{  } \widetilde{\sigma}\in [0,\frac{99}{100}].
\end{align}
Denote $u$ the solution to SMF with initial data $u_0$. Assume that $u$ satisfies
\begin{align*}
{\bf Bootstrap \mbox{  }I.}\mbox{  } \|P_{k}\nabla  u\|_{L^{\infty}_tL^2_x}\le\epsilon^{-\frac{1}{2}}_0c_k.
\end{align*}
Denote $v(s,t,x)$ the solution of heat flow with initial data $u(t,x)$, and $A_{i},A_t,A_s$ the corresponding connection coefficients. And denote the heat tension field  by $ \phi_s$ and the differential fields by $\{\phi_i\}$, $\phi_t$ respectively.  Suppose that $\{\phi_i\}^2_{i=1}$ satisfy
the bootstrap condition at $s=0$:
\begin{align*}
{\bf Bootstrap \mbox{  }II.} \mbox{  }\|P_{k}\phi_i \upharpoonright_{s=0}\|_{G_k(T)}&\le\epsilon^{-\frac{1}{2}}_0c_k.
\end{align*}

In Step 1, by studying the heat equations (\ref{heat1}), (\ref{heat2}), we prove {\bf Bootstrap I},{\bf II} in fact give parabolic estimates for $A_i,A_t$ and $\phi_{i,t}$ along the heat flow direction:
\begin{align*}
\|P_{k}\phi_i(s)\|_{F_k(T)}&\le c_k( {\sigma})2^{- {\sigma} k}(1+s2^{2k})^{-4}, \sigma\in [0,\frac{99}{100}]\\
\|P_{k}\phi_t(s)\|_{L^4_{t,x}}&\le c_k(\sigma)2^{-\sigma k}(1+s2^{2k})^{-2}, \sigma\in [0,\frac{99}{100}].\\
\|P_kA_i\upharpoonright_{s=0}\|_{L^{4}_{t,x}}&\le c_k(\sigma)2^{-\sigma k}, \sigma\in [0,\frac{99}{100}]\\
\|P_kA_t\upharpoonright_{s=0}\|_{L^{2}_{t,x}}&\lesssim \epsilon_0.
\end{align*}

In Step 2, by studying the Schr\"odinger equations (\ref{xcgfrt}), we prove {\bf Bootstrap I, II} indeed yield improved estimates for  $\phi_{i}$ along the Schr\"odinger flow direction:
\begin{align*}
\|P_{k}\phi_i\upharpoonright_{s=0}\|_{G_k(T)}\lesssim c_k(\sigma)2^{-\sigma k},\mbox{  } \sigma\in [0,\frac{99}{100}].
\end{align*}

In Step 3, we prove
\begin{align}\label{C098}
\|P_{k}\phi_i\upharpoonright_{s=0}\|_{G_k(T)}\lesssim c_{k}
\end{align}
with {\bf Bootstrap \mbox{  }I, II} dropped.

\section{ Evolution of SMF solutions along the heat direction }

\subsection{Parabolic Estimates for differential fields}

The main result of this  section  is the following Proposition.

\begin{Proposition}\label{Parabolic}
Let $\{b_k\}$ be a $\varepsilon$-frequency envelope.
Assume that for $i=1,2$, there hold
\begin{align}
\|P_{k }\phi_i\upharpoonright_{s=0}\|_{F_k(T)}&\le b_k(\sigma)2^{-\sigma k}, \mbox{  }\sigma\in [0,\frac{99}{100}]\label{tian}\\
\|P_{k }\phi_t\upharpoonright_{s=0}\|_{L^{4}_{t,x}}&\lesssim b_k(\sigma)2^{-(\sigma -1)k}, \mbox{  }\sigma\in [0,\frac{99}{100}]\label{tian2}.
\end{align}
and
\begin{align}\label{4.1}
\|P_{k }\phi_i(s)\|_{F_k(T)}&\le \varepsilon^{-\frac{1}{2}}b_k(1+s2^{2k})^{-4}.
\end{align}
Then if $\varepsilon>0$ is sufficiently small, for $\sigma\in[0,\frac{99}{100}]$ one has
\begin{align}
\|P_{k }\phi_i(s)\|_{F_k(T)}&\lesssim b_k(\sigma) 2^{-\sigma k} (1+s2^{2k})^{-4} \label{4.18}\\
\|P_k A_i\upharpoonright_{s=0}\|_{L^{4}_{t,x}}&\lesssim b_k(\sigma)2^{-\sigma k}, \mbox{ }i=1,2.\label{4.19}\\
\|P_k\phi_t(s)\|_{L^{4}_{t,x}}&\lesssim b_k(\sigma)2^{-(\sigma-1)k}(1+2^{2k}s)^{-2}\nonumber\\
\|P_kA_t\upharpoonright_{s=0}\|_{L^{2}_{t,x}}&\lesssim \varepsilon b_{k}(\sigma)2^{- \sigma k},\mbox{ }\sigma\in[\frac{1}{100},\frac{99}{100}]\\
\|P_kA_t\upharpoonright_{s=0}\|_{L^{2}_{t,x}}&\lesssim \varepsilon^2, \mbox{ } \sigma\in[0,\frac{99}{100}].\nonumber
\end{align}
\end{Proposition}

\noindent {\bf Remark.} Assumption (\ref{4.1}) can be dropped. It suffices to apply Sobolev embeddings, Lemma \ref{FFF} and [\cite{BIKT}, Page 1463]'s argument.

\subsection{Proof of Proposition \ref{Parabolic}}

Now we turn to prove the parabolic estimates in Proposition \ref{Parabolic}.

Denote
\begin{align}
h(k)=\sup_{s\ge 0}(1+s2^{2k})^4\sum^{2}_{i=1}\|P_k\phi_i(s)\|_{F_k(T)}.
\end{align}
Define the corresponding envelope by
\begin{align}
h_k(\sigma)&=\sup_{k'\in\Bbb Z}2^{\sigma k'}2^{-\delta|k'-k|}h(k').
\end{align}
Assume that
\begin{align}\label{FIL}
2^{\frac{1}{2}k}\|P_{k}(\widetilde{\mathcal{G}}^{(1)})\|_{L^{4}_xL^{\infty}_t(T)}\le \varepsilon^{-\frac{1}{4}}h_k[(1+2^{2k}s)^{-9}1_{j+k\ge 0}+ 1_{j+k\le 0}2^{\delta|k+j|}]
\end{align}
for any  $s\in [2^{2j-1}, 2^{2j+1})$, $k,j\in\Bbb Z$.

\begin{Lemma}\label{connection}
Under the assumptions  of Proposition \ref{Parabolic} and (\ref{FIL}), for any $k\in \Bbb Z$, $s\ge 0$, $i=1,2$, we have
\begin{align}\label{xX56}
\|P_k(A_i(s))\|_{F_k(T)\bigcap S^{\frac{1}{2}}_k(T)}\lesssim 2^{-\sigma k}(1+s2^{2k})^{-4}h_{k,s}(\sigma),
\end{align}
where the sequences $\{h_{k,s}\}$ when $ 2^{2k_0-1}\le s< 2^{2k_0+1}, k_0\in\Bbb Z$ are defined by
\begin{align}\label{GH}
{h_{k,s}}\left( \sigma  \right) = \left\{ \begin{gathered}
  {2^{k + {k_0}}}{h_{ - {k_0}}}{h_k}(\sigma )\mbox{  }{\rm if}\mbox{ }k + {k_0} \geqslant 0 \hfill \\
  \sum\limits_{l = k}^{ - {k_0}}  {h_l}{h _l}(\sigma )\mbox{  }{\rm if}\mbox{ }k + {k_0} \leqslant 0. \hfill \\
\end{gathered}  \right.
\end{align}
\end{Lemma}
\begin{proof}
 By assumption (\ref{4.1}) of Proposition \ref{Parabolic} and noticing $\{b_k\}$ is a  $\varepsilon$-envelope, we have
\begin{align}
\|\{h_k\}\|^2_{\ell^2}\le \varepsilon.
\end{align}

In order to prove (\ref{xX56}), let $B_1$ denote the smallest number in $[1,\infty)$ for which it holds for all $\sigma\in[0,\frac{99}{100}]$, $s\ge 0$, $k\in \Bbb Z$, $i=1,2$,
\begin{align}\label{PSD2}
\|P_k(A_i(s))\|_{F_k(T)\bigcap S^{\frac{1}{2}}_k(T)}\le B_1 2^{-\sigma k}(1+s2^{2k})^{-4}h_{k,s}(\sigma).
\end{align}

Recall that Section \ref{VNM} shows  $A_i$ is schematically written as
\begin{align}\label{PSD}
A_i(s)=\sum_{j_0,j_1,j_2,j_3}\int^{\infty}_s\left(\phi_i\diamond\phi_s\right)\langle{\bf R}(e_{j_0},e_{j_1}) (e_{j_2}),e_{j_3}\rangle ds',
\end{align}
where $\{j_c\}^{3}_{c=0}$ run in $\{1,...,2n\}$, and $i$ runs in $\{1,2\}$. Recall also that $\phi_s=\sum^{2}_{l=1}D_l\phi_l$.
Applying $P_k$ to (\ref{PSD}) we have
 \begin{align}
&\|P_k(A_i(s))\|_{F_k(T)\bigcap S^{\frac{1}{2}}_k(T)}\nonumber\\
&\le \sum_{|k_1-k_2|\le 8, k_1,k_2\ge k-4}\int^{\infty}_s\left\|P_{k}[P_{k_1}\left(\phi_i\diamond\phi_s\right)P_{k_2} \langle{\bf R}(e_{j_0},e_{j_1}) (e_{j_2}),e_{j_3}\rangle ]\right\|_{F_k(T)\bigcap S^{\frac{1}{2}}_k(T)}ds'\nonumber\\
&+\sum_{|k_1- k|\le 4}\int^{\infty}_s\left\|P_{k}[P_{k_1}\left(\phi_i\diamond\phi_s\right)P_{\le k-4} \langle{\bf R}(e_{j_0},e_{j_1}) (e_{j_2}),e_{j_3}\rangle ]\right\|_{F_k(T)\bigcap S^{\frac{1}{2}}_k(T)}ds'\nonumber\\
&+\sum_{|k_2-k|\le 4, k_1\le k-4}\int^{\infty}_s\left\|P_{k}[P_{k_1}\left(\phi_i\diamond\phi_s\right)P_{k_2} \langle{\bf R}(e_{j_0},e_{j_1}) (e_{j_2}),e_{j_3}\rangle ]\right\|_{F_k(T)\bigcap S^{\frac{1}{2}}_k(T)}ds'.\label{PSD2}
\end{align}
The above three subcases according to their order are usually called
(a) ${\rm High\times High \to Low}$,
(b) ${\rm High\times Low \to High}$,
(c) ${\rm Low   \times High \to High}$.

{\bf Case b. { High$\times$ Low $\to$ High}.}
In [\cite{BIKT}, Lemma 5.2, Page 1470], the authors have proved
\begin{align}\label{aFsD}
\sum^{2}_{i=1}\int^{\infty}_s\left\|P_{k}\left(\phi_i\diamond \phi_s\right)\right\|_{F_k(T)\bigcap S^{\frac{1}{2}}_k(T)}ds'
\lesssim \varepsilon B_12^{-\sigma k}(1+s2^{2k})^{-4}h_{k,s}(\sigma).
\end{align}
with slightly different notations. Thus in the case (b), by (\ref{PDE}) and applying the trivial bound
\begin{align}
\|\langle{\bf R}(e_{j_0},e_{j_1}) (e_{j_2}),e_{j_3}\rangle\|_{L^{\infty}_{t,s,x}}\lesssim K(\mathcal{N}),
\end{align}
to the $P_{\le k}$ part and (\ref{aFsD}) to the $P_{k_1}$ part, we obtain that
\begin{align}
\sum_{|k_1- k|\le 4}&\int^{\infty}_s\left\|P_{k}\left(P_{k_1}(\phi_i\diamond \phi_s)P_{\le k-4}\langle{\bf R}(e_{j_0},e_{j_1} )e_{j_2},e_{j_3}\rangle\right)\right\|_{F_k(T)\bigcap S^{\frac{1}{2}}_k(T)}\nonumber\\
&\lesssim \varepsilon B_12^{-\sigma k}(1+s2^{2k})^{-4}h_{k,s}(\sigma).\label{A1}
\end{align}

{\bf Step 2. Refined Dynamic Separation.} For the $\rm Low\times High$ and $\rm High\times High$ part, we need to further decompose the curvature term. The  dynamic separation performed in Section \ref{VNM} also needs to be refined.
Recall the notation
\begin{align*}
\mathcal{G}(s)= \langle {\bf R}(e_{j_0},e_{j_1})e_{j_2},e_{j_3}\rangle(s),
\end{align*}
for any given $j_0,...,j_3\in\{1,...,2n\}$, and the decomposition of $\mathcal{G}$ in Section \ref{VNM}.
Thus using $\psi_s=\sum_{i=1,2}(\partial_i+A_i)\psi_i$,  after the second time dynamic separation
$\mathcal{G}$ can be decomposed into
\begin{align}
&\mathcal{G}(s)= \langle {\bf R}(e_{j_0},e_{j_1})e_{j_2},e_{j_3}\rangle(s)\label{xoz}\\
&=\Gamma^{\infty}-\Gamma^{\infty,(1)}_l\int^{\infty}_s\psi^{l}_s(\widetilde{s})d\widetilde{s}
-\int^{\infty}_s \psi^{l}_s(\widetilde{s})\left(\int^{\infty}_{\widetilde{s}}\psi^{p}_s(s') ( \widetilde{\nabla}^2{\bf R})(e_{l},e_{p};e_{j_0},...,e_{j_3} ) ds'\right)d\widetilde{s}.\label{cov}\\
&:=\Gamma^{\infty}+\mathcal{U}_{00}+\mathcal{U}_{01}+\mathcal{U}_{I}+\mathcal{U}_{II}.\nonumber
\end{align}
where we denote
 \begin{align*}
\mathcal{U}_{00}&:=-\Gamma^{\infty,(1)}_l\int^{\infty}_s\sum^{2}_{i=1}(\partial_i\psi_i)ds'\\
\mathcal{U}_{01}&:=-\int^{\infty}_s\sum^{2}_{i=1}(\partial_i\psi_i)^{l}\left(  (\widetilde{\nabla}{\bf R})( {e_l};e_{j_0},...,e_{j_3})-\Gamma^{\infty,(1)}_l\right)ds'\\
\mathcal{U}_{I}&:=-\Gamma^{\infty,(1)}_l\int^{\infty}_s\sum^{2}_{i=1}(A_i\psi_i)^{l}ds'\\
\mathcal{U}_{II}&:=-\int^{\infty}_s\sum^{2}_{i=1} (A_i\psi_i)^{l}(\widetilde{s})\left(\int^{\infty}_{\widetilde{s}}\psi^{p}_s(s')  ( \widetilde{\nabla}^2 {\bf R})(e_{l},e_{p};e_{j_0},...,e_{j_3} )ds'\right)d\widetilde{s}\nonumber\\
&=- \int^{\infty}_s\sum^{2}_{i=1}(A_i\psi_i)^{l}(\widetilde{s})\left(  (\widetilde{\nabla}{\bf R})({e_l};e_{j_0},...,e_{j_3}) -\Gamma^{\infty,(1)}_l\right)d\widetilde{s}.
\end{align*}
It is easy to prove
 \begin{align}
&\left\|\int^{\infty}_{2^{2k_0-1}}\sum^{2}_{i=1}(\partial_i\psi_i)ds'\right\|_{F_{k}(T)}\nonumber\\
&\lesssim 2^{-\sigma k}h_{k}(\sigma)(1_{k_0+k\le 0}2^{-k}+1_{k_0+k\ge 0}2^{2k_0+k})(1+2^{2k_0+2k})^{-4}\label{kbnjnm}.
\end{align}
And recall that [\cite{BIKT}, Lemma 5.2] has shown for $s\in[2^{2j-1},2^{2j+2})$ there holds
\begin{align}\label{BVzm}
&\begin{array}{*{20}{c}}
  {{{\left\| {{P_k}(\phi_i\diamond {\phi _s})(s)} \right\|}_{{F_k}(T) \cap S_k^{\frac{1}
{2}}(T)}} } \hfill  \\
 \end{array} \\
&\lesssim \left\{ \begin{gathered}
  {2^{ - \sigma k}}{(1 + {2^{2k}}s)^{ - 4}}\left({\widetilde h_{k,s}}(\sigma )+B_1\varepsilon {2^{ - 2j}}{h_{k,s}}(\sigma )\right),\mbox{  }{\rm if}\mbox{  }k + j \ge 0 \hfill \\
 {2^{ - \sigma k}}\left( {\widetilde h_{k,s}}(\sigma )+ B_1\varepsilon  {2^{ - 2j}}{h_{ - j}}h_{ - j}(\sigma) \right),\mbox{  }{\rm if}\mbox{  }k + j \le   0 \hfill \\
\end{gathered}  \right.
\end{align}
where $\widetilde{h}_{k,s}(\sigma)$ is defined by
\begin{align*}
\widetilde{h}_{k,s}(\sigma)=2^{-j}h_{-j}\left(2^kh_k(\sigma)+2^{-j}h_{-j}(\sigma)\right).
\end{align*}
Then repeating the bilinear estimates [Lemma 5.1, \cite{BIKT}], we have
\begin{align}\label{vzgc}
\int^{\infty}_s\left\|P_k(\mathcal{U}_{00}(\phi_s\diamond \phi_i))\right\|_{F_{k}(T)\bigcap S^{\frac{1}{2}}_k(T)}ds'\lesssim (1+\varepsilon B_1)2^{-\sigma k}h_{k,s}(\sigma)(1+2^{2k_0+2k})^{-4}.
\end{align}
The $\Gamma^{\infty}$ part
\begin{align*}
\int^{\infty}_s\left\|P_k(\Gamma^{\infty}(\phi_s\diamond \phi_i))\right\|_{F_{k}(T)\bigcap S^{\frac{1}{2}}_k(T)}ds'\lesssim (1+\varepsilon B_1)2^{-\sigma k}h_{k,s}(\sigma)(1+2^{2k_0+2k})^{-4}
\end{align*}
follows by directly applying Lemma 5.2 of [\cite{BIKT}],
since $\Gamma^{\infty}$ is just a constant.

Recall the notation  $\widetilde{\mathcal{G}}^{(1)}=  (\widetilde{\nabla}{\bf R})(e;e_{j_0},e_{j_1},e_{j_2},e_{j_3})-\Gamma^{\infty,(1)}$.
For $\mathcal{U}_{01}$, applying  (\ref{3.16ab}) which says
\begin{align}\label{hvzgc}
2^{k}\|P_{k}\widetilde{\mathcal{G}}^{(1)}\|_{L^{\infty}_tL^2_x}\lesssim 2^{-\sigma k}h_{k}(\sigma)(1+2^{2k}s)^{-30},
\end{align}
and (\ref{FIL}) which says for $s\in[2^{2j-1},2^{2j+1})$
\begin{align}\label{hvzgc}
2^{\frac{k}{2}}\|P_{k}\widetilde{\mathcal{G}}^{(1)}\|_{L^4_xL^{\infty}_t}\lesssim 2^{-\sigma k}h_{k}(\sigma)[(1+2^{2k}s)^{-9}1_{j+k\ge 0}+2^{\delta|j+k|} 1_{k+j\le 0}],
\end{align}
we obtain by Lemma \ref{mJ} that
 \begin{align*}
&\|P_{k}\left((\partial_i\psi_i)\widetilde{\mathcal{G}}^{(1)}\right)\|_{F_{k}(T)}\\
&\lesssim \sum_{|k_1-k|\le 4}\|P_{k_1}\partial_i\phi_i\|_{F_{k_1}(T)}\|P_{\le k-4}\widetilde{\mathcal{G}}^{(1)}\|_{L^{\infty}}
+ \sum_{|k_2-k|\le 4,k_1\le k-4}2^{\frac{k_1}{2}}\|P_{k_1}\partial_i\phi_i\|_{F_{k_1}(T)}\|P_{k_2}\widetilde{\mathcal{G}}^{(1)}\|_{L^4_xL^{\infty}_t}\\
&+ \sum_{|k_2-k|\le 4,k_1\le k-4} 2^{k_1}\|P_{k_1}\partial_i\phi_i\|_{F_{k_1}(T)}\|P_{k_2}\widetilde{\mathcal{G}}^{(1)}\|_{L^4}  \\
&+ \sum_{|k_2-k_1|\le 8,k_1,k_2\ge k-4}\|P_{k_1}\partial_i\phi_i\|_{F_{k_1}(T)}(\|P_{k_2}\widetilde{\mathcal{G}}^{(1)}\|_{L^{\infty}}
 +2^{\frac{1}{2}k_1}\|P_{k_2}\widetilde{\mathcal{G}}^{(1)}\|_{L^4_xL^{\infty}_t}).
\end{align*}
Thus by the slow variation of envelopes we further have
\begin{align*}
\|P_{k}\left((\partial_i\psi_i)\widetilde{\mathcal{G}}^{(1)}\right)\|_{F_{k}(T)}\lesssim 2^{-\sigma k}h_{k}(\sigma)\left( 1_{k+j\ge 0}2^{k }(1+2^{2k+2j})^{-4}
+ 1_{k+j\le 0}2^{-j}2^{\delta|j+k|}\right)
\end{align*}
for $s\in[2^{2j-1},2^{2j+1})$, $k,j\in\Bbb Z$. Notice that the large constant $ \varepsilon^{-\frac{1}{4}}$ is absorbed by $\|\{h_{k}\}\|_{\ell^{\infty}}\lesssim \varepsilon^{\frac{1}{2}}$.  And notice that in the ${\rm Low\times High}$ interaction of $(\partial_i\psi_i)\widetilde{\mathcal{G}}^{(1)}$ it is possible to deduce $2^{-\sigma k}$  for $\sigma\in[0,\frac{99}{100}]$ from $\partial_i \psi_i$ due to fact that the series $\sum_{k_1\le k-4}2^{2k-\sigma k}h_{k}(\sigma)$  is summable for $\sigma<2$. Thus for $s\in[2^{2k_0-1},2^{2k_0+1})$ summing the above formula in $j\ge k_0$  yields
\begin{align*}
&\int^{\infty}_s\|P_{k}\left((\partial_i\psi_i)\widetilde{\mathcal{G}}^{(1)}\right)\|_{F_{k}(T)}ds'\lesssim  2^{-\sigma k}h_{k}(\sigma)\left( 1_{k+k_0\ge 0}2^{k+2k_0 }(1+2^{2k+2k_0})^{-4}
+ 1_{k+k_0\le 0}2^{-k} \right)
\end{align*}
which is the same as (\ref{kbnjnm}). Thus (\ref{BVzm}) and  bilinear estimates give
\begin{align*}
\int^{\infty}_s\left\|P_k(\mathcal{U}_{01}(\phi_s\diamond \phi_i))\right\|_{F_{k}(T)\bigcap S^{\frac{1}{2}}_k(T)}ds'\lesssim (1+\varepsilon B_1)2^{-\sigma k}h_{k,s}(\sigma)(1+2^{2k_0+2k})^{-4}.
\end{align*}

Therefore, it remains to estimate $\mathcal{U}_{I},\mathcal{U}_{II}$.

{\bf Step 3. Prove our lemma with a bootstrap condition.}
We first prove our lemma with additional bootstrap condition. And in the final step we will drop the bootstrap condition and finish the whole proof.

{\bf Bootstrap Assumption A.}
Assume that  for all $k,j\in\Bbb Z$, $s\in[2^{2j-1},2^{2j+1})$, there holds
\begin{align}
\|P_{k}\mathcal{U}_{I}\|_{F_{k}(T)\bigcap S^{\frac{1}{2}}_k(T)}&\le \varepsilon^{-\frac{1}{2}}(1+2^{k+j})^{-7}T_{k,j}h_{k}c^*_0\label{1YTN}\\
\|P_{k}\mathcal{U}_{II}\|_{F_{k}(T)}&\le\varepsilon^{-\frac{1}{2}}(1+2^{k+j})^{-7}T_{k,j}h_{k}c^*_0,\label{YTN}
\end{align}
where $c^*_0:=\|\{h_k\}\|_ {\ell^2}$ and we denote
\begin{align}\label{FFgB}
T_{k,j}=1_{k+j\le 0}2^{-k}+1_{k+j\ge 0}2 ^{j}.
\end{align}
Our aim for this step  is to prove $B_1$ defined by (\ref{PSD2}) satisfies
\begin{align*}
B_1\lesssim 1+\varepsilon B_1,
\end{align*}
assuming {Bootstrap Assumption A}.

For $s\in [2^{2j-1},2^{2j+1})$, $j\in\Bbb Z$, (\ref{1YTN}) and (\ref{YTN})  show $\mathcal{U}:=\mathcal{U}_{I}+\mathcal{U}_{II}$ satisfies
\begin{align}
2^{k}(1+2^{k+j})^6\|P_{k}\mathcal{U }\|_{ F_{k}(T)}
&\lesssim 1.\label{VXZb}
\end{align}

{\bf Case a. ${\rm High\times High \to Low}$.} The bound  (\ref{VXZb}) suffices to control the ${\rm High\times High}$ interaction.
For $k+k_0\ge 0$, applying the  bounds (\ref{VXZb}), (\ref{BVzm}) and (\ref{FC}) of Lemma \ref{FCq} with $\omega=\frac{1}{2}$,
one obtains that in ${\rm High\times High}$ case
 \begin{align}
&\sum_{j\ge k_0}\int^{2^{2j+1}}_{2^{2j-1}}\sum_{|k_1-k_2|\le 8, k_1,k_2\ge k-4}
{\left\| {{P_k}\left( P_{k_1}{\left( {{\phi _i} \diamond {\phi _s}} \right)P_{k_2}\mathcal{U}} \right)} \right\|_{{F_k}(T) \cap S_k^{\frac{1}
{2}}(T)}}ds'\nonumber\\
&\lesssim
{2^{ - \sigma k}}\sum_{j \ge  {k_0}}{\sum _{{k_1} \ge  k-4}}{{2^{\frac{{k - {k_1}}}
{2}}} }   {(1 + {2^{{2k_1} + 2j}})^{ -6}}{h_{ - j}}\left( {{2^{{k_1} + j}}{h_{{k_1}}}(\sigma ) + {h_{ - j}}(\sigma )} \right) \nonumber\\
  &+ {2^{ - \sigma k}}\sum_{j \ge  {k_0}}{\sum _{{k_1} \ge  k-4}} {{2^{\frac{{k - {k_1}}}
{2}}}  } B_1\varepsilon {(1 + {2^{{2k_1+2j}}})^{ -6}}{h_{{k_1},{2^{2j}}}}(\sigma ) \label{0IYC}
\end{align}
which by slow variation of envelopes is further bounded by
 \begin{align*}
&{2^{ - \sigma k}}\sum_{j \ge  {k_0}}{\sum _{{k_1} \ge  k-4}} {2^{\frac{k-k_1}
{2}}} {(1 + {2^{{k_1} + j}})^{ -10}}{2^{\delta |j- {k_0}|}}{h_{ - {k_0}}}{h_k}(\sigma )\left( {{2^{{k_1} + j + \delta \left| {{k_1} - k} \right|}} + {2^{\delta |k + j|}}} \right) \hfill \\
  &+ {2^{ - \sigma k}}\sum_{j \ge  {k_0}}{\sum _{{k_1} \ge  k-4}} {2^{\frac{k-k_1}
{2}}} {{{B}}_1}\varepsilon {(1 + {2^{{k_1} + j}})^{ -10}}{2^{\delta |j - {k_0}|}}
{2^{\delta \left| {{k_1} - k} \right|}}{h_{ - {k_0}}}{h_k}(\sigma ).
\end{align*}
Since $k_1+j\gtrsim k +j\ge k+k_0\ge  0$, it is easy to see the above formula is acceptable by
\begin{align*}
&\sum_{j\ge k_0} 2^{-\sigma k}(1+\varepsilon B_1)h_{-k_0} h_{k}(\sigma)2^{\delta|j-k_0|} 2^{\delta|k+j|}(1+2^{k+j})^{-8}\nonumber\\
&\lesssim   (1+\varepsilon B_1)2^{-\sigma k}h_{-k_0}h_{k}(\sigma) 2^{k_0+k} (1+2^{k+k_0})^{-8}.
\end{align*}
Therefore, the $k+k_0\ge 0$ case has been done for the ${\rm High\times High}$ interaction.

Assume $k+k_0\le 0$. Applying the  bounds  (\ref{BVzm}), (\ref{VXZb}) with $\sigma=0$ and (\ref{FC}) of Lemma \ref{FCq} with $\omega=\frac{1}{2}$, for $j+k\le 0$, by slow variation of envelopes we have
\begin{align*}
&\sum_{k_0\le  j\le -k}\int^{2^{2j+1}}_{2^{2j-1}}
\sum_{|k_1-k_2|\le 8, k_1,k_2\ge k-4} {\left\| {{P_k}\left(P_{k_1}( {{\phi _i}\diamond {\phi _s}} )P_{k_2}\mathcal{U} \right)} \right\|_{{F_k}(T) \cap S_k^{\frac{1}
{2}}(T)}}ds'\nonumber\\
&\lesssim\sum_{k_0\le j\le -k}\sum_{k_1\ge k-4}2^{\frac{k-k_1}{2}}2^{-\sigma k_1} (1+2^{2k_1+2j})^{-6}  h_{-j}\left(2^{k_1+j}h_{k_1}(\sigma)+ h_{-j}(\sigma)\right)\\
&+B_1\varepsilon  {2^{ - \sigma k}} \sum_{k_0\le j\le -k}\left(\sum_{k-4\le k_1\le -j}2^{\frac{k-k_1}{2}}{h_{ - j}}h_{ - j}(\sigma )+\sum_{k_1\ge -j}2^{\frac{k-k_1}{2}} (1+2^{2j+2k_1})^{-4}h_{k_1,2^{2j}}(\sigma)\right) \\
&\lesssim \sum_{ k_0 \le j\le -k}2^{-\sigma k} (1+\varepsilon B_1) h_{-j}h_{-j}(\sigma).
\end{align*}
Therefore, for $k_0+k\le 0$, the ${\rm High\times High}$ part is bounded by
 \begin{align}
&\sum_{j\ge k_0}\int^{2j+1}_{2j-1}\sum_{|k_1-k_2|\le 8, k_1,k_2\ge k-4}\|P_k(P_{k_1}(\phi_i\diamond\phi_s)P_{k_2}\mathcal{U})\|_{F_k(T)\bigcap S^{\frac{1}{2}}_k(T)}d\tau\nonumber\\
&\lesssim \sum_{  k_0\le j\le -k}C (1+2^{k+j})^{-12}2^{-\sigma k }h_{-j}h_{-j}(\sigma) 2^{(1\pm \delta)(k +j) }+\sum_{j\ge -k}(...)\label{Ket}\\
&\lesssim(1+B_1\varepsilon) \sum_{k_0\le j\le -k}2^{-\sigma k }h_{-j}h_{-j}(\sigma)+(1+B_1\varepsilon) 2^{-\sigma k }h_{k}h_{k}(\sigma)\nonumber\\
&\lesssim \sum_{k_0 \le j\le -k  }(1+B_1\varepsilon)2^{-\sigma k }h_{-j}h_{-j}(\sigma).\nonumber
\end{align}
where the $\sum_{j\ge -k}(...)$ part in (\ref{Ket}) is bounded by $(1+B_1\varepsilon)2^{-\sigma k }h_{k}h_{k}(\sigma)$ via directly using results of the $k+k_0\ge 0$ case.
Therefore, we conclude that
 \begin{align*}
&\sum_{j\ge k_0}\int^{2j+1}_{2j-1}\sum_{|k_1-k_2|\le 8, k_1,k_2\ge k-4}\|P_k(P_{k_1}(\phi_i\diamond\phi_s)P_{k_2}\mathcal{U})\|_{F_k(T)\bigcap S^{\frac{1}{2}}_k(T)}d\tau\\
&\lesssim (1+\varepsilon B_1)2^{-\sigma k}(1+2^{j+k})^{-8}h_{k,s}(\sigma).
\end{align*}

{\bf Case c. $\rm Low\times High \to High$. }
In Case c, we assume $|k-k_2|\le 4$ and $k\ge k_1+4$, i.e. $\rm Low\times High \to High$. We applied the bound
\begin{align*}
\|P_{k}(f_{k_1}g_{k_2})\|_{F_k\bigcap S^{\frac{1}{2}}_k}\lesssim 2^{k_1}\|P_{k_1}f \|_{F_{k_1}\bigcap S^{\frac{1}{2}}_{k_1}}\|P_{k_2}g \|_{F_{k_2}(T)},
\end{align*}
provided that $|k-k_2|\le 20$. To avoid too long formula, we recall the notation
\begin{align*}
{T}_{k,j}=(1_{k+j\ge 0}2^j+1_{k+j\le 0}2^{-k}).
\end{align*}
Thus by (\ref{BVzm}), (\ref{1YTN}) and (\ref{YTN}), for $j\ge k_0$, $\sigma\in[0,\frac{99}{100}]$, we obtain that  the $\rm Low\times High \to High$ part is bounded by
\begin{align*}
&\sum_{j\ge k_0}\int^{2^{2j+1}}_{2^{2j-1}}\sum_{|k_2-k|\le 4, k_1\le  k-4}\|P_k\left(P_{k_1}(\phi_s \diamond \phi_i)P_{k_2}{\mathcal{U}}\right)\|_{F_{k }(T)\bigcap S^{\frac{1}{2}}_{k }(T)} \nonumber\\
&\lesssim \sum_{j\ge k_0}\sum_{|k-k_2|\le 4}\sum_{k_1\le k-4}\int^{2^{2j+1}}_{2^{2j-1}}2^{k_1}\|P_{k_1}(\phi_s \diamond \phi_i)\|_{F_{k_1}\bigcap S^{\frac{1}{2}}_{k_1}(T)}\|P_{k_2}(\mathcal{U})\|_{F_{k_2} (T)} \\
&\lesssim  \sum_{j\ge k_0} h_{k}T_{k,j}(1+2^{j+k})^{-7} \sum\limits_{{k_1} \le k} 2^{k_1 -\sigma k_1}{{h_{ - j}}\left( {{2^{{k_1} + j}}{h_{{k_1}}(\sigma) } + {h_{ - j}(\sigma)} } \right){{(1 + {2^{{k_1} + j}})}^{ - 7}}}  \\
 & + \sum_{j\ge k_0}  h_{k}T_{k,j}(1+2^{j+k})^{-7}{B_1}\varepsilon  1_{k+j\ge 0}\sum\limits_{{k_1} \le  {-j}} 2^{k_1-\sigma k_1 }{h_{-j}}{h_{-j}}(\sigma) \\
&+ \sum_{j\ge k_0} h_{k}T_{k,j}(1+2^{j+k})^{-7}{B_1}\varepsilon
1_{k+j\ge 0}\sum\limits_{-j\le {k_1} \le k} 2^{k_1 -\sigma k_1}{{2^{{k_1} + j}}{h_{ - j}}{h_{k_1}(\sigma)}{{(1 + {2^{k_1 + j}})}^{ - 7}}} \\
&+ \sum_{j\ge k_0}  h_{k}T_{k,j}(1+2^{j+k})^{-7}{B_1} \varepsilon
 1_{k+j\le 0}\sum\limits_{ {k_1} \le k} 2^{k_1-\sigma k_1 } {h_{-j}}{h_{-j}}(\sigma).
\end{align*}
Therefore, for  $k+k_0\ge 0$ we conclude
\begin{align*}
&\sum_{j\ge k_0} \int^{2^{2j+1}}_{2^{2j-1}}\sum_{|k_2-k|\le 4, k_1\le k-4}\|P_k\left(P_{k_1}(\phi_s \diamond \phi_i)P_{k_2}{\mathcal{U}}\right)\|_{F_{k }(T)\bigcap S^{\frac{1}{2}}_{k }(T)} \nonumber\\
&\lesssim 2^{-\sigma k}\sum_{j\ge k_0}
(1+B_1\varepsilon) {(1 + {2^{k + j}})^{ - 7}}2^{\delta|j-k_0|} {h_{ - k_0}}{h_{{k}}}(\sigma)\\
&\lesssim {2^{ - \sigma k}}\left(1+B_1\varepsilon\right){(1 + {2^{2k + 2j}})^{ - 4}}h_{k,2^{2k_0}}(\sigma).
\end{align*}
And for $k+k_0\le 0$, we also have
\begin{align*}
&\sum_{j\ge k_0}\sum_{|k_2-k|\le 8, k_1\le k}\int^{2^{2j+1}}_{2^{2j-1}}\|P_k\left(P_{k_1}(\phi_s \diamond \phi_i)P_{k_2}{\mathcal{U}}\right)\|_{F_{k }(T)\bigcap S^{\frac{1}{2}}_{k }(T)} \\
&\lesssim {2^{ - \sigma k}}\left(1+B_1\varepsilon ^{\frac{1}{2}}\right){(1 + {2^{2k + 2j}})^{ - 4}}h_{k,2^{2k_0}}(\sigma).
\end{align*}
Thus the $\rm Low\times High$ part has been done for $\mathcal{U}$ as well.

Therefore, combining the three cases, we summarize
\begin{align}
\|P_k(A_i(s))\|_{F_k\bigcap S^{\frac{1}{2}}_k}&\lesssim  (\varepsilon^{\frac{1}{2}} B_1+1)2^{-\sigma k } (1+2^{k+k_0})^{-8}   h_{k,2^{2k_0}}(\sigma)\label{1ABC}.
\end{align}
Then (\ref{1ABC}) shows
\begin{align}
B_1 \lesssim  \varepsilon^{\frac{1}{2}} B_1 +1.
\end{align}
Hence $B_1 \lesssim 1$. So, we have obtained our lemma for $\sigma\in[0,\frac{99}{100}]$ with assuming Bootstrap Assumption A.

{\bf Step 4.}
In this step, we prove our lemma remains valid as if we drop the Bootstrap condition A in (\ref{YTN}), (\ref{1YTN}).
First, we prove a claim.

{\bf Claim A.} If (\ref{1YTN})-(\ref{YTN}) hold, then for all $k,j\in \Bbb Z$, $\sigma\in[0,\frac{99}{100}]$, $s\in[2^{2j-1},2^{2j+1})$,
\begin{align}
\|P_{k}\mathcal{U}_{I}\|_{F_{k_2}(T)\bigcap S^{\frac{1}{2}}_k(T)}&\le c^*_0 2^{-\sigma k}(1+2^{k+j})^{-7}  T_{k,j}h_{k}(\sigma)\label{4YTN}\\
\|P_{k}\mathcal{U}_{II}\|_{F_{k_2}(T)}&\le c^*_02^{-\sigma k}(1+2^{k+j})^{-7}T_{k,j}h_{k}(\sigma).\label{3YTN}
\end{align}
Recall the definition of $\mathcal{U}_{I}$
\begin{align*}
\mathcal{U}_{I}&=-\Gamma^{\infty,(1)}_{l}\int^{\infty}_s\sum_{i=1,2}(A_i\psi_i)^{l}(s')ds',
\end{align*}
For $U_{II}$, it is better to use
\begin{align*}
\mathcal{U}_{II}= -\int^{\infty}_s\sum_{i=1,2}(A_i\psi_i)^{p}(s')\left[  (\widetilde{\nabla}{\bf R})(e_p;e_{j_0},e_{j_1},e_{j_2},e_{j_3}) -\Gamma^{\infty,(1)}_p\right] ds'.
\end{align*}
Recall the notation
\begin{align*}
\widetilde{\mathcal{G}}^{(1)}=(\widetilde{\nabla} {\bf R})(e ;e_{j_0},...,e_{j_3})-\Gamma^{\infty,(1)}.
\end{align*}
Moreover, we have by (\ref{3.16ab}) and $c^{*}_0:=\|\{h_k\}\|_{\ell^2}$ that
\begin{align}\label{3Jgvm}
\|P_k(\widetilde{\mathcal{G}}^{(1)})(s)\|_{L^{\infty}_{t,x}}\lesssim  2^{-\sigma k} h_{k}(\sigma)(1+s2^{2k})^{-20}, \mbox{ }\forall \sigma\in[0,\frac{99}{100}].
\end{align}
Thus in order to prove Claim A for $\mathcal{U}_{II}$, it suffices to prove
\begin{align}\label{Jgv}
\int^{\infty}_s\|(A_i\psi_i)\widetilde{\mathcal{G}}^{(1)} \|_{F_k(T)}ds'\lesssim
(1+2^{k}s^{\frac{1}{2}})^{-7}c^*_0T_{k,j}.
\end{align}
The bound claimed for $\mathcal{U}_{I}$ is easier to verify. Since now $B_1 \lesssim 1$,
applying bilinear Lemma 8.2 to $A_i\psi_i$, one has for $j+k\ge 0$, $s\in[2^{2j-1},2^{2j+1})$
\begin{align}
&\|(A_i\psi_i)\|_{F_k(T)\bigcap S^{\frac{1}{2}}_{k_1}(T)}\label{poknjuy}\\
&\lesssim \left(1_{k+j\ge 0}2^{-j}+1_{k+j\le 0}2^{\frac{k-j}{2}}\right)(1+2^{2k+2j})^{-4}2^{-\sigma k}c^*_0 h_{k,s}(\sigma).\nonumber
\end{align}
Summing the above formula w.r.t. $j\ge k_0$, we get
\begin{align}\label{1Jgvm}
\int^{\infty}_{s}\left\| A_i\psi_i \right\|_{{F_k}(T) \cap S_k^{\frac{1}
{2}}(T)}ds'\lesssim c^*_0 2^{-\sigma k}T_{k,k_0}(1+2^{k_0+k})^{-7}h_{k}(\sigma).
\end{align}
Thus Claim A has been verified for $\mathcal{U}_{I}$.

For $\mathcal{U}_{II}$,  we will use the following inequality (see (\ref{2Jgvm}))
\begin{align*}
\|P_k(P_{k_1}f P_{k_2} g)\|_{F_k(T)}\lesssim \|P_{k_2}g \|_{L^{\infty}_x}\|P_{k_1} f\|_{F_{k_1}(T)\bigcap S^{\frac{1}{2}}_{k_1}(T)}.
\end{align*}
Then by Littlewood-Paley bilinear decomposition,
\begin{align*}
&\|P_k ( (A_i\psi_i) \widetilde{\mathcal{G}}^{(1)})\|_{F_k(T)}\lesssim \sum_{|k_1-k|\le 4}\|P_{k_1} (A_i\psi_i)\|_{F_{k_1}(T)\bigcap S^{\frac{1}{2}}_{k_1}(T)}\|P_{\le k-4} \widetilde{\mathcal{G}}^{(1)}\|_{ L^{\infty} }\\
&+\sum_{|k_1-k_2|\le 8,k_1,k_2\ge k-4}\|P_{k_1} (A_i\psi_i)\|_{F_{k_1}(T)\bigcap S^{\frac{1}{2}}_{k_1}(T)}\|P_{ k_2} \widetilde{\mathcal{G}}^{(1)}\|_{ L^{\infty} }\\
&+\sum_{|k_2-k|\le 4,k_1 \le  k-4}\|P_{k_1}(A_i\psi_i)\|_{F_{k_1}(T)\bigcap S^{\frac{1}{2}}_{k_1}(T)}\|P_{ k_2} \widetilde{\mathcal{G}}^{(1)}\|_{ L^{\infty} }.
\end{align*}
Thus by (\ref{poknjuy}),  the ${\rm High\times Low}$ part of $(A_i\psi_i)\widetilde{\mathcal{G}}^{(1)}$ is dominated by
 \begin{align*}
&\sum_{|k-k_1|\le 4}\|P_k (P_{k_1 }(A_i\psi_i)P_{\le k-4}\widetilde{\mathcal{G}}^{(1)})\|_{F_k(T)}\\
&\lesssim c^*_0{2^{ - \sigma k}}\left(1_{k+j\le 0}2^{\frac{k-j}{2}+\delta|k+j|} {h_k}(\sigma )+1_{k+j\ge 0}2^{-j}(1+2^{k+j})^{-7}h_{k}(\sigma)h_{-j}\right).
\end{align*}
Summing in $j\ge k_0$ yields
 \begin{align*}
&\sum_{j\ge k_0}2^{2j}\sum_{|k-k_1|\le 4}\|P_k (P_{k_1 }(A_i\psi_i)P_{\le k}\widetilde{\mathcal{G}}^{(1)})\|_{F_k(T)}
\\
&\lesssim c^*_0 2^{-\sigma k}h_{k}(\sigma) \left(1_{k+k_0\ge 0}2^{k_0}(1+2^{k+k_0})^{-7}
+   1_{k+k_0\le 0}2^{-k} \right).
\end{align*}
Using (\ref{3Jgvm}) and (\ref{poknjuy}), the $\rm High\times High$ part of $(A_i\psi_i)\widetilde{\mathcal{G}}^{(1)}$ is dominated by
 \begin{align*}
&\sum_{|k_2-k_1|\le 8,k_1,k_2\ge k-4}\|P_k (P_{k_1 }(A_i\psi_i)P_{k_2}\widetilde{\mathcal{G}}^{(1)})\|_{F_k(T)}\\
&\lesssim  c^*_0 {1_{k + j \ge  0}}\sum\limits_{{k_1} \ge  k-4} {{2^{ - \sigma {k_1}}}(1 + {{\text{2}}^{2j + 2k}})^{ - 7}{2^{{k_1}}}{h_{{k_1}}}(\sigma )} \\
 &+c^*_0 {1_{k + j \le 0}}\left[{\sum _{k-4 \le  {k_1} \le   - j}}2^{\frac{k_1-j}{2}}{2^{\delta |{k_1} + j|}}{2^{ - \sigma {k_1}}}{h_{{k_1}}}(\sigma ) + {\sum _{{k_1} \ge - j}}{2^{ - \sigma {k_1}}} (1+{2} ^{2j + 2k})^{ - 7}{2^{{k_1}}}{h_{{k_1}}}(\sigma )\right]\\
&\lesssim c^*_0 1_{k+j\ge 0} 2^{-j}(1+2^{j+k})^{-10}2^{-\sigma k}  h_{k}(\sigma)+
c^*_01_{k+j\le 0} 2^{-j}2^{\delta|k+j|}2^{-\sigma k}  h_{k}(\sigma).
\end{align*}
Summing in $j\ge k_0$ also gives
 \begin{align*}
&\sum_{j\ge k_0}2^{2j}\sum_{|k_2-k_1|\le 8,k_1,k_2\ge k-4}\|P_k (P_{k_1 }(A_i\psi_i)P_{k_2}\widetilde{\mathcal{G}}^{(1)})\|_{F_k(T)}\\
&\lesssim c^*_0 2^{-\sigma k} 1_{k+k_0\ge 0}2^{k_0}(1+2^{k+k_0})^{-7}h_{k}(\sigma)
+c^*_0 1_{k+k_0\le 0}2^{-k}{2^{ - \sigma k}} {h_k}(\sigma ).
\end{align*}
Then using (\ref{3Jgvm}), the $\rm Low\times High$ part of $(A_i\psi_i)\widetilde{\mathcal{G}}^{(1)}$ is bounded as
 \begin{align*}
&\sum_{|k-k_2|\le 4,k_1\le k-4}\|P_k (P_{k_1 }(A_i\psi_i)P_{k_2}\widetilde{\mathcal{G}}^{(1)})\|_{F_k(T)}\\
&\lesssim  h_{k}(\sigma)2^{-\sigma k}1_{k+j\le 0}\sum_{k_1\le k} h_{k_1} 2^{\frac{1}{2}(k_1-j)}2^{\delta|k_1+j|} \\
&+h_{k}(\sigma) 2^{-\sigma k}(1+2^{2j+2k})^{-20}1_{k+j\ge 0}\left[\sum_{-j\le k_1\le k}h_{k_1} 2^{k_1 }(1+2^{2j+2k_1})^{-4}\right]\\
 &+h_{k}(\sigma)2^{-\sigma k}(1+2^{2j+2k})^{-20}1_{k+j\ge 0}\left[\sum_{k_1\le -j} h_{k_1} 2^{\frac{k_1-j}{2}}2^{\delta|k_1+j|}\right]\\
&\lesssim c^*_0 2^{-\sigma k}1_{k+j\ge 0} 2^{-j}(1+2^{j+k})^{-7}2^{-\sigma k}  h_{k}(\sigma)+
c^*_02^{-\sigma k}1_{k+j\le 0} 2^{\frac{k-j}{2}}2^{\delta|k+j|}2^{-\sigma k}  h_{k}(\sigma).
\end{align*}
Summing in $j\ge k_0$  as well yields
 \begin{align*}
&\sum_{j\ge k_0}2^{2j}\sum_{|k-k_2|\le 4,k_1\le k+4}\|P_k (P_{k_1 }(A_i\psi_i)P_{k_2}\widetilde{\mathcal{G}}^{(1)})\|_{F_k(T)}\\
&\lesssim c^*_0 2^{-\sigma k} 1_{k+k_0\ge 0}2^{k_0}(1+2^{k+k_0})^{-7}h_{k}(\sigma)
+c^*_0 1_{k+k_0\le 0}2^{-k}{2^{ - \sigma k}} {h_k}(\sigma ).
\end{align*}
Thus back to the LHS of (\ref{Jgv}),
we conclude if (\ref{YTN}) holds, then
\begin{align*}
\|P_k\mathcal{U}_{II}\|_{F_k(T)}\lesssim c^*_0 2^{-\sigma k} 1_{k+k_0\ge 0}2^{k_0}(1+2^{k+k_0})^{-7}h_{k}(\sigma)
+c^*_0  1_{k+k_0\le 0}2^{-k}{2^{ - \sigma k}} {h_k}(\sigma ).
\end{align*}
Particularly (\ref{YTN}) holds, thus proving Claim A.

Now we are ready to prove our lemma with (\ref{1YTN}) and (\ref{YTN})  being dropped. Define the function on $T'\in [0,T]$
\begin{align*}
\Phi(T')&=\sum_{\{j_c\}\subset\{1,...,2n\}}\sup\limits_{k,j\in\Bbb Z}\sup\limits_{s\in [2^{2j-1},2^{2j+1}]}(c^*_0)^{-1}(1+2^{k}2^{\frac{s}{2}})^{7}T^{-1}_{k,j}  h^{-1}_k \\
&\times\left( \|P_{k}\mathcal{U}_{I}\|_{F_{k}(T')}+\|P_{k}\mathcal{U}_{II}\|_{F_{k}(T')}\right).
\end{align*}
Using Lemma \ref{FFF} and Sobolev embeddings, we find that $\Phi(T')$ is a continuous function in $T'\in [0,T]$.
Then in order to prove our lemma, it suffices to prove $\Phi\lesssim 1$. It is easy to see $\Phi$ is also an increasing continuous  function on $T'\in [0,T]$.  And Claim A shows
\begin{align*}
\Phi(T')\le \varepsilon^{-\frac{1}{2}}\Longrightarrow \Phi(T')\lesssim 1.
\end{align*}
Hence it suffices to verify
\begin{align*}
\mathop {\lim }\limits_{T' \to 0} \Phi (T')\lesssim 1.
\end{align*}
This reduces to prove for all $j,k\in\Bbb Z,s\in[2^{2j-1},2^{2j+1})$
 \begin{align*}
&\sum_{\{j_c\}\subset\{1,...,2n\}}\left\|P_{k}\int^{\infty}_s(A_i\psi_i)^{q}
\left[ (\widetilde{\nabla}{\bf R})({e_q};e_{j_0},...,e_{j_3})-\Gamma^{\infty,(1)}_q\right]\right\|_{L^2_x}
 +\int^{\infty}_{s}\|P_{k}(A_i\psi_i)\|_{L^2_x}ds' \\
&\lesssim c^*_0 (1+2^{k}2^{\frac{s}{2}})^{-7}T_{k,j}h_k,
\end{align*}
where all these fields $\psi_i$, matrices $A_i$ are associated with the heat flow with initial data $u_0$. This can be proved by applying results of Section 3. In fact, by the definition of $h_{k}(\sigma)$ one has
\begin{align*}
2^{\sigma k}\|\phi_i\upharpoonright_{s=0}\|_{L^{\infty}_tL^{2}_x}\le h_k(\sigma),
\end{align*}
if $\sigma\in[0,\frac{99}{100}]$.
Then by Proposition \ref{7.2} with $\eta_k(\sigma)=h_{k}(\sigma)$  we get
\begin{align}\label{njnnhbuy78}
(1+2^{2k}s)^{29}\|P_{k}A_i(s)\|_{L^{\infty}_tL^{2}_x}\le 2^{-\sigma k}h_{k,s}(\sigma).
\end{align}
And proof of Proposition \ref{7.2} shows
\begin{align}\label{njnnhbuy88}
(1+2^{2k}s)^{30}\|P_{k}\psi_i(s)\|_{L^{\infty}_tL^{2}_x}\le 2^{-\sigma k}h_{k}(\sigma).
\end{align}
if $\sigma\in[0,\frac{99}{100}]$.
Then by (\ref{njnnhbuy88}), (\ref{njnnhbuy78}) and bilinear Littlewood-Paley decomposition, one obtains
\begin{align*}
(1+2^{2k}s)^{28}\int^{\infty}_s\|P_k(A_i\psi_i)\|_{L^{\infty}_tL^2_x}ds'\lesssim \|\{h_k\}\|_{\ell^2}T_{k,j} h_{k}(\sigma)2^{-\sigma k}
\end{align*}
for $s\in[2^{2j-1},2^{2j+1})$, $k,j\in\Bbb Z$. The left is to prove
\begin{align}
(1+2^{2k}s)^{28}\int^{\infty}_s\|P_k[A_i\psi_i\widetilde{\mathcal{G}}^{(1)}]\|_{L^{\infty}_tL^2_x}ds'\lesssim \|\{h_k\}\|_{\ell^2}T_{k,j} h_{k}(\sigma)2^{-\sigma k}, \mbox{ }\sigma\in[0,\frac{99}{100}].
\end{align}
This follows by  (\ref{njnnhbuy78}), (\ref{njnnhbuy88}), (\ref{chui}) and bilinear Littlewood-Paley decomposition as well.
\end{proof}

\begin{Remark}
Checking the proof of Lemma 4.1, we see the range of $\sigma\in[0,\frac{99}{100}]$ was only used in the $\rm Low \times High$ interaction of $(\partial_i\psi_i)\widetilde{\mathcal{G}}^{(1)}$, $(A_i\psi_i)(\mathcal{U}_I+\mathcal{U}_{II})$ of Step 2 and Step 3 respectively.
\end{Remark}

\begin{Lemma}\label{mJ}
If $|k_1-k|\le 4$, then
\begin{align}\label{1mJgvm}
\|P_k( P_{k_1} fP_{\le k-4}g )\|_{F_k(T)}\lesssim \|P_{k_1} g\|_{F_{k_1}(T)}\|P_{\le k-4}g \|_{L^{\infty}}.
\end{align}
If $|k_2-k_1|\le 8$, $k_1,k_2\ge k-4$, then
\begin{align}\label{2mJgvm}
\|P_k( P_{k_1}f P_{k_2} g)\|_{F_k(T)}\lesssim \|P_{k_1}f\|_{F_{k_1}(T)}(\|P_{k_2}g \|_{L^{\infty}}+2^{\frac{1}{2}k_2}\|P_{k_2}g \|_{L^4_xL^{\infty}_t}).
\end{align}
If $|k_2-k|\le 4$, $k_1\le k-4$, then
\begin{align}\label{3mJgvm}
\|P_k( P_{k_1}f P_{k_2} g)\|_{F_k(T)}\lesssim  2^{\frac{k_1}{2}}\|P_{k_1} f\|_{F_{k_1}(T)}\|P_{k_2}g \|_{L^4_xL^{\infty}_t}+2^{k_1}\|P_{k_1} f\|_{F_{k_1}(T)} \|P_{k_2}g \|_{L^4}.
\end{align}
For any $k_1,k_2,k\in\Bbb Z$, one has
\begin{align}\label{2Jgvm}
\|P_k(P_{k_1}f P_{k_2} g)\|_{F_k(T)}\lesssim\|P_{k_1} f\|_{F_{k_1}(T)\bigcap S^{\frac{1}{2}}_{k_1}(T)} \|P_{k_2}g \|_{L^{\infty}}.
\end{align}
\end{Lemma}
\begin{proof}
 (\ref{1mJgvm}) has been given in \cite{BIKT}. And (\ref{2mJgvm}) follows by H\"older and [(3.17), \cite{BIKT}]
which says
\begin{align*}
\|P_k f\|_{F_k(T)}\lesssim \|P_k  f\|_{L^{2}_xL^{\infty}_t}+\|P_k f\|_{L^{4}_{t,x}}.
\end{align*}
(\ref{3mJgvm}) follows by the same reason with additionally using Bernstein inequality.
Moreover, by definition
\begin{align*}
\|P_k f\|_{L^{4}_{t,x}}&\le \|f\|_{F_{k}(T)},\mbox{ }\|P_k f\|_{L^{2}_xL^{\infty}_t}\le  \|f\|_{S^{\frac{1}{2}}_{k}(T)}.
\end{align*}
Thus one obtains
\begin{align*}
\|P_k (P_{k_1 }fP_{k_2}g)\|_{F_k(T)}&\lesssim \|P_{k_1} fP_{k_2} g\|_{L^{2}_xL^{\infty}_t}+\|P_{k_1} fP_{k_2} g\|_{L^{4}_{t,x}}\\
&\lesssim \left(\|P_{k_1} f\|_{L^{2}_xL^{\infty}_t}+\| P_{k_1} f\|_{L^{4}_{t,x}}\right)\|P_{k_2}g\|_{L^{\infty}}\\
&\lesssim \|P_{k_2}g\|_{L^{\infty}} \|P_{k_1} f\|_{F_{k_1}(T)\bigcap S^{\frac{1}{2}}_{k_1}(T)}.
\end{align*}
\end{proof}

The proof of Lemma \ref{connection} yields
\begin{Corollary}\label{kjl}
Under the assumptions of Proposition \ref{Parabolic} and (\ref{FIL}), for $s\in[2^{2j-1},2^{2j+1}]$, $\sigma\in[0,\frac{99}{100}]$, $j,k\in\Bbb Z$, there holds
\begin{align*}
\|P_{k}(\mathcal{\widetilde{G}})\|_{F_{k}(T)} \lesssim 2^{-\sigma k} h_k(\sigma)T_{k,j}(1+2^{j+k})^{-7}.
\end{align*}
where $T_{k,j}$ is defined by (\ref{FFgB}).
When $s=0$, we have
\begin{align*}
\|P_{k}(\mathcal{\widetilde{G}})\upharpoonright_{s=0}\|_{F_{k}(T)} \lesssim 2^{-\sigma k} h_k(\sigma)2^{-k}.
\end{align*}
\end{Corollary}
\begin{proof}
Lemma \ref{connection} gives
\begin{align*}
\|P_{k}(\mathcal{U}_{00})\|_{F_{k}(T)}+\|P_{k}(\mathcal{U}_{01})\|_{F_{k}(T)}&\lesssim 2^{-\sigma k} h_k(\sigma)T_{k,j}(1+2^{j+k})^{-7}  \\
 \|P_{k}\mathcal{U}_{I}\|_{F_{k}(T)\bigcap S^{\frac{1}{2}}_k(T)}&\lesssim 2^{-\sigma k}T_{k,j}h_k(\sigma)(1+2^{j+k})^{-7} \\
\|P_{k}\mathcal{U}_{II}\|_{F_{k}(T)}&\lesssim 2^{-\sigma k}T_{k,j}h_k(\sigma)(1+2^{j+k})^{-7}.
\end{align*}
Then our corollary follows by the decomposition
\begin{align*}
\mathcal{\widetilde{G}}=\mathcal{G}-\Gamma^{\infty}=\mathcal{U}_{00}+\mathcal{U}_{01}+\mathcal{U}_{I}+\mathcal{U}_{II},
\end{align*}
and the inequality $(1+2^{j+k})^{-1}T_{k,j}\le 2^{-k}$ for all $j,k\in\Bbb Z$.
\end{proof}

\subsection{ Evolution of differential fields along the heat flow }

Recall the evolution equation for $\phi_i$ along the heat flow:
\begin{align}
(\partial_s-\Delta)\phi_i&=K_{i}\nonumber\\
K_{i}&:=2\sum^{2}_{j=1}\partial_j(A_j\phi_i)+\sum^{2}_{j=1}(A^2_j-\partial_jA_j)\phi_i+\sum^{2}_{j=1} \phi_j\diamond \phi_i \diamond \phi_j\mathcal{G}.\label{Heat1}
\end{align}

Now we control the nonlinearities in the above equations.
\begin{Lemma}
Under the assumptions  of Proposition \ref{Parabolic} and (\ref{FIL}), for all $s\in[0,\infty)$, $i=1,2$ and $\sigma\in [0,\frac{99}{100}]$, we have
\begin{align}\label{Kende}
\|\int^{s}_0e^{(s-\tau)\Delta}P_{k}K_i(\tau)d\tau\|_{F_k(T)}\lesssim \varepsilon(1+s2^{2k})^{-4}2^{-\sigma k}h_k(\sigma).
\end{align}
\end{Lemma}
\begin{proof}
First, we consider the quartic term $\mathcal{G}(\phi_i\diamond \phi_j)\diamond\phi_j$ in $K_i$.
[\cite{BIKT}, (5.25)] have proved that for $\tau\in [2^{2j-1}, 2^{2j+1})$
\begin{align}\label{FCa}
\|P_{k}\left(  \phi_i\diamond \phi_p \diamond\phi_{l}\right)(\tau)\|_{F_k(T)\bigcap S^{\frac{1}{2}}_k(T)}\lesssim \varepsilon 2^{-\sigma k}2^{2k}(1+2^{2k+2j})^{-4}\left[h_k(\sigma)+2^{-\frac{3}{2}(k+j)}h_{-j}(\sigma)\right].
\end{align}
Recall that $\mathcal{G}=\mathcal{\widetilde{G}}+\Gamma^{\infty}$. The constant part follows by (\ref{FCa}).
By bilinear Littlewood-Paley decomposition we have
\begin{align*}
&\|P_{k}\left(  \phi_i\diamond \phi_p \diamond\phi_{l}\mathcal{ \widetilde{G}}\right)\|_{F_k(T)}\\
&\lesssim \sum^{|k_1-k_2|\le 8}_{k_1\ge k-4} \|P_{k_1}\left(  \phi_i\diamond \phi_p \diamond\phi_{l}\right)P_{k_2}\mathcal{ \widetilde{G}}\|_{F_k(T)}
+\sum^{|k_2-k|\le 4}_{k_1\le k-4} \|P_{k_1}\left(  \phi_i\diamond \phi_p\diamond\phi_{l}\right)P_{k_2}\mathcal{ \widetilde{G}}\|_{F_k(T)}\\
&+\sum_{|k_1-k|\le 4} \|P_{k_1}\left(  \phi_i\diamond \phi_p\diamond\phi_{l}\right)P_{\le k-4}\mathcal{\widetilde{ G}}\|_{F_k(T)}
\end{align*}
For the $\rm High\times Low$ term,  directly applying $\|\mathcal{\widetilde{G}}\|_{L^{\infty}_{t,x}}\le K(\mathcal{N})$ gives
 \begin{align*}
\|P_{k_1}\left(  \phi_i\diamond \phi_p \diamond\phi_{l}\right)P_{\le k-4}\mathcal{ \widetilde{G}}\|_{F_k(T)}
&\lesssim  \|P_{k}\left(  \phi_i\diamond \phi_p\diamond\phi_{l}\right)\|_{F_k(T)\bigcap S^{\frac{1}{2}}_k(T)}\\
&\lesssim \varepsilon 2^{-\sigma k}2^{2k}(1+2^{2k+2j})^{-4}\left[h_k(\sigma)+2^{-\frac{3}{2}(k+j)}h_{-j}(\sigma)\right].
\end{align*}
For the $\rm High\times High$ term, denoting $\mathcal{V}:=  \phi_i\diamond \phi_p \diamond\phi_{l}$, Corollary \ref{kjl} and (\ref{FC}) show
 \begin{align}
&\sum_{|k_1-k_2|\le 8, k_1,k_2\ge k-4}\|P_{k_1}\left(  \phi_i\diamond \phi_p \diamond\phi_{l}\right)P_{k_2}\mathcal{ G}\|_{F_k(T)}\nonumber\\
&\lesssim \sum_{|k_1-k_2|\le 8, k_1,k_2\ge k-4}
2^{\frac{k_1+k}{2}}\|P_{k_1}\mathcal{V}\|_{F_{k_1}(T)\bigcap S^{\omega}_{k_1}(T)}\|P_{k_2}\mathcal{G}\|_{F_{k_2}(T)}\nonumber\\
&\lesssim \sum_{k_1\ge k,|k_1-k_2|\le 2} 2^{\frac{k_1+k}{2}} 2^{-\sigma k_1+2k_1 }(1+2^{k_1+j})^{-15}\left[h_{k_1}(\sigma)+2^{-\frac{3}{2}(k_1+j)}h_{-j}({\sigma})\right]T_{k_1,j} h_{k_1} .\label{XzZ1}
\end{align}
If $k+j\ge 0$, then by slow variation of envelopes, (\ref{XzZ1}) is bounded by
\begin{align*}
2^{3k+j}2^{-\sigma k}{(1 + {2^{k + j}})^{ - 14}} {h_{{k}}}(\sigma )h_{k}.
\end{align*}
If $k+j\le 0$,  using $(1+2^{k_1+j})^{-1}T_{k_1,j}\le 2^{-k_1}$ for all $k_1\in \Bbb Z$, (\ref{XzZ1}) is dominated by
\begin{align*}
2^{\frac{k}{2}-\frac{3j}{2}}h_{k}(\sigma)h_{k}2^{\delta|j+k|}.
\end{align*}
Therefore, for the $\rm High\times High$ interaction, we have if $s\in [2^{2j-1},2^{2j+1})$ then
\begin{align}
&\sum_{|k_1-k_2|\le 8, k_1,k_2\ge k-4} \|P_{k_1}\left(  \phi_i\diamond \phi_p \diamond\phi_{l}\right)P_{k_2}\mathcal{ \widetilde{G}}\|_{F_k(T)}\lesssim   {2^{ - \sigma k - \frac{3}
{2}j}}{2^{\frac{1}
{2}k}}{2^{\delta |j + k|}}h_{k}{h_k}(\sigma )(1+2^{2k+2j})^{-5}.\label{Hao}
\end{align}
To finish estimates for the $\rm High\times High$ interaction, we turn to verify the corresponding part in (\ref{Kende}).
We  use (\ref{Hao}) to verify (\ref{Kende}). Let $s\in[2^{2k_0-1},2^{2k_0+1})$ with $k_0\in\Bbb Z$ fixed. For $k+k_0\le 0$, by (\ref{3.18}) one has
 \begin{align*}
 &\|\int^{s}_0e^{ (s-\tau)\Delta}\sum_{k_1\ge k-4,|k_1-k_2|\le 8}P_{k}(P_{k_1}\mathcal{V}P_{k_2}\mathcal{\widetilde{G}})d\tau\|_{F_k(T)}\\
  &\lesssim \sum_{j\le k_0}\int^{2^{2j+1}}_{2^{2j-1}}\sum_{k_1\ge k-4,|k_1-k_2|\le 8}\|P_{k}(P_{k_1}\mathcal{V}p_{k_2}\mathcal{\widetilde{G}})\|_{F_k(T)}d\tau\lesssim \varepsilon  \sum_{j\le k_0}2^{-\sigma k} h_{k}(\sigma)\left(2^{\frac{1}{2}(j+k)}+2^{(\frac{1}{2}\pm \delta)(k+j)}\right)\\
&\lesssim \varepsilon  2^{-\sigma k} h_{k}(\sigma).
\end{align*}
For $k+k_0\ge 0$, by (\ref{3.18}), (\ref{Hao}) one has
 \begin{align*}
& { \int^{s}_0}\|e^{ (s-\tau)\Delta}\sum_{k_1\ge k-4,|k_1-k_2|\le 8}P_{k}(P_{k_1}\mathcal{V}P_{k_2}\mathcal{\widetilde{G}}) \|_{F_k(T)}d\tau\lesssim \int^{\frac{s}{2}}_{0}...d\tau
+\int^{s}_{\frac{s}{2}}...d\tau\nonumber\\
&\lesssim\sum _{j\le -k_0-1}\int^{2^{2j+1}}_{2^{2j-1}}2^{-20(k+k_0)}\sum_{k_1\ge k-4,|k_1-k_2|\le 8}\|P_{k}(P_{k_1}\mathcal{V}P_{k_2}\mathcal{\widetilde{G}})\|_{F_k(T)}d\tau
\\
&+2^{2k_0} \sum_{k_1\ge k-4,|k_1-k_2|\le 8}\|P_{k}(P_{k_1}\mathcal{V}P_{k_2}\mathcal{\widetilde{G}})\|_{F_{k(T)}}\upharpoonright_{\tau\in[2^{2k_0-2},2^{2k_0+1}]}\\
&\lesssim\varepsilon 2^{-20(k+k_0)}\sum_{j\le -k_0-1}2^{-\sigma k} h_{k}(\sigma)\left(2^{\frac{1}{2}(j+k)}+2^{(\frac{1}{2}\pm \delta)(j+k)}\right)
\\
&+ \varepsilon 2^{-\sigma k} h_{k}(\sigma)(1+2^{k+k_0})^{-10}\left[2^{\frac{1}{2}(k_0+k)}+2^{(\frac{1}{2}\pm \delta)(k_0+k)}\right]\\
&\lesssim\varepsilon (1+2^{k+k_0})^{-8} 2^{-\sigma k} h_{k}(\sigma).
\end{align*}
Thus we conclude
\begin{align*}
& \int^{s}_0 \|e^{ (s-\tau)\Delta}\sum_{k_1\ge k-4,|k_1-k_2|\le 8}P_{k}(P_{k_1}\mathcal{V}P_{k_2}\mathcal{\widetilde{G}})\|_{F_k(T)}d\tau\lesssim  \varepsilon^22^{-\sigma k} h_{k}(\sigma)(1+2^{2k }s)^{-4}.
\end{align*}

For the $\rm Low\times High$ term, using
 \begin{align*}
\|P_{k_1}(\mathcal{V})\|_{L^{\infty}_{t,x}}\lesssim 2^{k_1}\varepsilon 2^{-\sigma k_1+2k_1}(1+2^{k_1+j})^{-8}
&\left[h_{k_1}(\sigma)+2^{-\frac{3}{2}(k_1+j)}h_{-j}(\sigma)\right]
\end{align*}
we obtain by Corollary \ref{kjl} and (\ref{FC}) that
  \begin{align*}
&\sum_{k_1\le k-4,|k-k_2|\le 4}\left\|P_{k}\left( P_{k_1}(  \phi _i\diamond {\phi _p}\diamond{\phi _l} )P_{k_2}\mathcal{\widetilde{G}}\right)\right\|_{{F_k}(T)}\nonumber\\
 &\lesssim  2^{-\sigma k}{T_{k,j}}{(1 + {2^{j + k}})^{ - 7}}{h_k}(\sigma)\varepsilon {\sum _{{k_1} \le  k-4}}{2^{ 3{k_1}}}{(1 + {2^{{k_1} + j}})^{ - 8}}\left( 1 + {2^{ - \frac{3}
{2}({k_1} + j)}} \right).
\end{align*}
For $k+j\ge 0$, we have
 \begin{align}
&\sum_{k_1\le k-4,|k-k_2|\le 4}\left\|P_{k}( P_{k_1}\left( ( \phi _i\diamond {\phi _p}\diamond{\phi _l} )P_{k_2}\mathcal{\widetilde{G}}\right)\right\|_{{F_k}(T)} \lesssim \varepsilon 2^{-\sigma}2^{-2j}(1+2^{k+j})^{-7}h_{k}(\sigma).\label{H7ao1}
\end{align}
For $k+j\le 0$, we have
 \begin{align}
&\sum_{k_1\le k-4,|k-k_2|\le 4}\left\|P_{k}\left( P_{k_1}(  \phi _i\diamond {\phi _p}\diamond{\phi _l} )P_{k_2}\mathcal{\widetilde{G}}\right) \right\|_{{F_k}(T)} \lesssim \varepsilon 2^{-\sigma k}2^{\frac{1}{2}k-\frac{3}{2}j}(1+2^{k+j})^{-7}h_{k}(\sigma).\label{H8ao1}
\end{align}
As a summary, we use (\ref{H7ao1}), (\ref{H8ao1}) to verify (\ref{Kende}). Let $s\in[2^{2k_0-1},2^{2k_0+1})$ with $k_0$ fixed. For $k+k_0\le 0$, by (\ref{3.18}) one has
for the $\rm Low\times High$ interaction that
 \begin{align*}
\int^{s}_0 \|e^{ (s-\tau)\Delta}\sum_{k_1\le k-4,|k-k_2|\le 4}P_k(P_{k_1}\mathcal{V}P_{k_2}\mathcal{\widetilde{G}})\|_{F_k(T)}d\tau&\lesssim  \sum_{j\le k_0}\int^{2^{2j+1}}_{2^{2j-1}}...d\tau\\
&\lesssim \varepsilon \sum_{j\le k_0}2^{\frac{j}{2}+\frac{k}{2}}2^{-\sigma k} h_k({\sigma})
\lesssim  2^{-\sigma k} \varepsilon  h_{k}(\sigma).
\end{align*}
For $k+k_0\ge 0$, similarly we have
 \begin{align*}
& \int^{s}_0\|e^{ (s-\tau)\Delta}\sum_{k_1\le k-4,|k-k_2|\le 4}P_k(P_{k_1}\mathcal{V}P_{k_2}\mathcal{\widetilde{G}})\|_{F_k(T)}d\tau\\
&\lesssim \varepsilon 2^{-20(k+k_0)}2^{-\sigma k}h_{k}(\sigma)\sum_{j\le -k_0-1} 2^{\frac{1}{2}(j+k)} + \varepsilon 2^{-\sigma k} h_{k}(\sigma)(1+2^{k+k_0})^{-7}2^{-2k_0-2k}\\
&\lesssim \varepsilon  (1+2^{k+k_0})^{-8} 2^{-\sigma k} h_{k}(\sigma).
\end{align*}
Thus  we conclude
\begin{align*}
\int^{s}_0\|e^{ (s-\tau)\Delta}\sum_{k_1\le k-4,|k-k_2|\le 4}P_k(P_{k_1}\mathcal{V}P_{k_2}\mathcal{\widetilde{G}})\|_{F_k(T)}d\tau\lesssim  \varepsilon (1+2^{k+k_0})^{-8}2^{-\sigma k} h_{k}(\sigma).
\end{align*}
And the $\rm High\times Low$ case is easy by repeating the same argument or directly applying the result of [\cite{BIKT}, Lemma 5.3].
Therefore, the curvature term has been done:
\begin{align*}
\left\|\int^{s}_0e^{ (s-\tau)\Delta}P_{k}\left(  \phi_i\diamond \phi_p \diamond\phi_{l}\mathcal{ \widetilde{G}}\right)d\tau\right\|_{F_k(T)}\lesssim \varepsilon(1+s2^{2k})^{-4}2^{-\sigma k} h_{k}(\sigma).
\end{align*}

{\bf Step 2. Connection coefficient term.} In this step, we turn to estimate the terms $\partial_l(A_l\psi_i)$, $\partial_lA_l\phi_i$ and $A^2_l\phi_i$. With Lemma \ref{connection} in hand,  all these terms follow directly by repeating arguments of [\cite{BIKT},Lemma 5.3].

\end{proof}

\begin{Lemma}\label{Zoom0}
Under the assumptions of Proposition \ref{Parabolic} and (\ref{FIL}), for all $k\in \Bbb Z$, $s\ge 0$, $i=1,2$, we have
\begin{align}
\|P_{k }\phi_i(s)\|_{F_k(T)}&\lesssim b_k(\sigma)2^{-\sigma k} (1+s2^{2k})^{-4}\mbox{  }\sigma\in[0,\frac{99}{100}].\label{4.18}
\end{align}
\end{Lemma}
\begin{proof}
By Duhamel principle and (\ref{Kende}), we get
\begin{align}\label{h6bna}
 \sup_{s\ge 0}(1+s2^{2k})^{4}2^{\sigma k} \|P_{k }\phi_i(s)\|_{F_k(T)}\lesssim  b_{k}(\sigma)+\varepsilon h_{k}(\sigma).
\end{align}
Since the RHS of (\ref{h6bna}) is a frequency envelope of $\delta$ order, by the definition of $\{h_k(\sigma) \}$ we get
\begin{align}\label{h6bn}
h_k(\sigma)\lesssim  b_{k}(\sigma)+\varepsilon h_{k}(\sigma),
\end{align}
which by letting $\varepsilon$ be sufficiently small yields
\begin{align}\label{nsbm}
h_k(\sigma)\lesssim b_{k}(\sigma).
\end{align}
\end{proof}

\begin{Lemma}\label{Zoom1}
Under the assumptions  of Proposition \ref{Parabolic} and (\ref{FIL}),  for all $k\in \Bbb Z$, $s\ge 0$, $i=1,2$, we have
\begin{align}
\|P_k A_i\upharpoonright_{s=0}\|_{L^{4}_{t,x}}&\lesssim b_k(\sigma)2^{-\sigma k}.\label{4.19}
\end{align}
\end{Lemma}
\begin{proof}
Lemma \ref{Zoom0} has shown (\ref{nsbm}).  Then the previous bounds in Lemma \ref{connection} now hold with $h_{k}(\sigma)$ replaced by $b_{k}(\sigma)$.
Recall that [\cite{BIKT}, Page 1473-1474] proved
\begin{align}\label{Bvn}
\|P_{k }\phi_s\|_{L^{4}_{t,x}}\lesssim 2^{k}2^{-\sigma k}b_k(\sigma) (2^{2k}s)^{-\frac{3}{8}}(1+s2^{2k})^{-3}.
\end{align}
We also recall the bilinear estimate of [\cite{BIKT}, Lemma 5.4] in our Appendix A,  Lemma \ref{90}.
Then (\ref{Bvn}) and  (\ref{4.18}) show
\begin{align*}
\|P_{k }(\phi_i\diamond\phi_s)\|_{L^{4}_{t,x}}&\lesssim  2^{-\sigma k}\sum_{l\le k} b_{k}(\sigma)b_{l}2^{l+k}(s2^{2k})^{-\frac{3}{8}}(1+s2^{2k})^{-3}
\\
&+ 2^{-\sigma k}\sum_{l\le k}b_{k}(\sigma) b_{l}2^{2l}2^{\frac{1}{2}(k-l)}(2^{2k}s)^{-\frac{3}{8}}(1+s2^{2k})^{-4}\\
&+  \sum_{l\ge k}2^{-\sigma l}b_{l}(\sigma)b_{l}2^{k+l} (2^{2l}s)^{-\frac{3}{8}}(1+s2^{2l})^{-7}.
\end{align*}
Thus given $s\in[2^{2j-1},2^{2j+2})$ with $j\in\Bbb Z$, we conclude for $k+j\ge 0$,
\begin{align}\label{3.591}
\|P_{k }(\phi_i\diamond\phi_s)\|_{L^{4}_{t,x}}\lesssim  b_{k}(\sigma)b_{k}2^{2k-\sigma k}(s2^{2k})^{-\frac{3}{8}}(1+s2^{2k})^{-3},
\end{align}
and for $k+j\le  0$,
\begin{align}\label{GVfcm}
\|P_{k }(\phi_i\diamond\phi_s)\|_{L^{4}_{t,x}}\lesssim  b_{k}(\sigma)b_{k}2^{2\delta|k+j|}2^{-\sigma k}2^{k}2^{-j}.
\end{align}
Recalling that for $i=1,2$,
\begin{align}
A_i(0)=\int^{\infty}_0((\phi_i\diamond \phi_s)\mathcal{G})ds.
\end{align}
we see the left is to deal with the interaction of $\phi_s\diamond\phi_i$ with $\mathcal{G}$.
Recall that $\mathcal{G}=\Gamma^{\infty}+\mathcal{\widetilde{G}}$ with
\begin{align}\label{CXbZ}
\|P_k (\mathcal{\widetilde{G}}) \|_{F_k(T)}\lesssim 2^{-\sigma k}T_{k,j}(1+2^{k+j})^{-7}h_{k}(\sigma),
\end{align}
for $s\in[2^{2j-1},2^{2j+1})$ with $j\in\Bbb Z$. (\ref{GVfcm}), (\ref{3.591}) show  the constant part $\Gamma^{\infty}$ contributes to
$\|A_i(0)\|_{L^4_{t,x}}$ by $b_k(\sigma)2^{-\sigma k}$. Thus it suffices to control $(\phi_i\diamond \phi_s)\mathcal{\widetilde{G}}$.

As before, we consider three cases according to Littlewood-Paley decomposition.
Using the trivial bound $\|\mathcal{\widetilde{G}}\|_{L^{\infty}_{t,x}}\le K(\mathcal{N})$ in the $\rm High\times Low$ part gives
 \begin{align*}
&\sum_{j\in \Bbb Z}\int^{2^{2j+1}}_{2^{2j-1}}\sum_{|k_1-k|\le 4}\|P_k\left[P_{k_1}(\phi_s\diamond\phi_i)P_{\le k-4}\mathcal{\widetilde{G}}\right]\|_{L^{4}_{t,x}}ds\\
&\lesssim   \sum_{j\ge -k} 1_{k+j\ge 0}b_{k}(\sigma)b_{k}2^{2k+2j-\sigma k}(2^{2k+2j})^{-\frac{3}{8}}(1+2^{2k+2j})^{-3}+\sum_{j\le -k} 1_{k+j\le 0} b_{k}(\sigma)b_{k}2^{2\delta|k+j|}2^{-\sigma k}2^{k}2^{j}\\
&\lesssim  b_{k}(\sigma)b_{k}2^{-\sigma k}.
\end{align*}
Notice that Lemma \ref{90} shows
\begin{align*}
\sum_{|k_1-k_2|\le 8,k_1,k_2\ge k-4}\left\|P_k(P_{k_1}fP_{k_2}g)\right\|_{L^{4}_{t,x}}\lesssim  \sum_{k_1\ge k} 2^{k(1+\omega)}2^{-\omega k_1}\mu_{k_1}\nu_{k_1},
\end{align*}
where $\mu_k=\sum_{|k-k'|\le 20}\|P_{k '}f\|_{S^{\omega}_{k'}}$, $\nu_k= \sum_{|k-k'|\le 20}\|P_{k'} g\|_{L^{4}_{t,x}}$.
Thus using (\ref{CXbZ}), we have by choosing $\omega=0$ that
 \begin{align*}
&\int^{\infty}_{0}\sum_{|k_1-k_2|\le 8,k_1,k_2\ge k-4}\left\|P_k\left[P_{k_1}(\phi_s\diamond\phi_i)P_{k_2}\mathcal{\widetilde{G}}\right]\right\|_{L^{4}_{t,x}}ds\\
&\lesssim \sum_{j\ge -k}2^{2j} \varepsilon  2^{k}\sum_{k_1\ge k-4} b_{k_1}b_{k_1}(\sigma)2^{k_1}2^{-\sigma k_1} (1+2^{2j+2k_1})^{-5}\\
&+ \sum_{j\le -k}\mbox{ } 2^{-\sigma k}2^{k}\sum_{k-4\le k_1\le -j}2^{j} b_{k_1}(\sigma)b_{k_1}2^{2\delta|k_1+j|}
\\
&+2^{-\sigma k}\sum_{j\le -k}2^{2j} 2^{k}\sum_{k_1\ge -j} b_{k_1}(\sigma)b_{k_1}2^{2\delta|k_1+j|}2^{k_1}(2^{2j+2k_1})^{-\frac{3}{8}}(1+2^{2j+2k_1})^{-5}\\
&\lesssim  \sum_{j\ge -k}2^{-\sigma k} 2^{2j+2k}(1+2^{2k+2j})^{-5}b_{k}b_{k }(\sigma)
+\sum_{j\le -k} 2^{-\sigma k}2^{(k+j)+\delta|k+j|}b_{k}b_{k }(\sigma)\\
&\lesssim  b_{k}b_{k }(\sigma)2^{-\sigma k},
\end{align*}
where we used $(1+2^{k+j})^{-1}T_{k,j}\le 2^{-k}$ for all $k\in\Bbb Z$.
In the $\rm Low\times High$ part, Lemma \ref{90} shows
\begin{align*}
\sum_{|k_2-k|\le 4,k_1\le k-4}\left\|P_k(P_{k_1}fP_{k_2}g)\right\|_{L^{4}_{t,x}}\lesssim  \sum_{l\le k} 2^{l}\mu_{l}\nu_{k},
\end{align*}
where $\nu_k= \sum_{|k'-k|\le 20}\|P_{k'} f\|_{S^{\omega}_k}$, $\nu_k=\sum_{|k'-k|\le 20}\|P_{k'} g\|_{L^{4}_{t,x}}$. Then by (\ref{CXbZ}) we have
 \begin{align*}
&\sum_{|k_2-k|\le 4,k_1\le k-4}\left\|P_k\left[P_{k_1}(\phi_s\diamond\phi_i)P_{k_2}\mathcal{\widetilde{G}}\right]\right\|_{L^{4}_{t,x}}\\
&\lesssim  2^{-\sigma k}b_{k}(\sigma)\varepsilon T_{k,j}(1+2^{j+k })^{-7}
 \sum_{l\le k} 2^{l} b_{l}b_{l } \left(1_{{l+j}\le 0}2^{2\delta|l+j|}2^{l-j}+1_{j+l\ge0}2^{2l}2^{-\frac{3}{4}(j+l)}(1+2^{2j+2l})^{-3}\right)\\
&\lesssim  2^{-\sigma k}b_{k}(\sigma)b_{k }\left( 1_{k+j\ge 0} 2^{-2j} (1+2^{k+j })^{-4}+   1_{k+j\le 0} 2^{k-j}2^{\delta|k+j|}\right).
\end{align*}
Hence we conclude for the $\rm Low\times High$ part
\begin{align*}
&\int^{\infty}_0\sum_{|k_2-k|\le 4,k_1\le k-4}\left\|P_k\left[P_{k_1}(\phi_s\diamond\phi_i)P_{k_2}\mathcal{\widetilde{G}}\right]\right\|_{L^{4}_{t,x}}ds\\
&\lesssim  2^{-\sigma k}\sum_{j\in\Bbb Z} b_{k}b_{k }(\sigma)1_{k+j\ge 0} (1+2^{k+j })^{-4}+  2^{-\sigma k}\sum_{j\in\Bbb Z}b_{k}b_{k }(\sigma)  1_{k+j\le 0} 2^{k+j}2^{\delta|k+j|}\\
&\lesssim  b_{k}b_{k }(\sigma)2^{-\sigma k}.
\end{align*}

Therefore, we get
\begin{align*}
\|P_k(A_i(0))\|_{L^{4}_{t,x}}\lesssim  2^{-\sigma k }  b_{k }(\sigma).
\end{align*}
\end{proof}

Now we turn to the bounds for $\phi_t$ stated in Proposition \ref{Parabolic}.
\begin{Lemma}\label{Zoom2}
Assume that the assumptions of Proposition \ref{Parabolic} and (\ref{FIL}) hold.
Then for $\sigma\in[0,\frac{99}{100}]$, one has
\begin{align}
\|P_k\phi_t(s)\|_{L^{4}_{t,x}}&\lesssim b_k(\sigma)2^{-(\sigma-1)k}(1+2^{2k}s)^{-2}\label{cx0}.
\end{align}
\end{Lemma}
\begin{proof}
Recall that $\phi_t$ satisfies
 \begin{align*}
\partial_s\phi_t-\Delta \phi_t&=L(\phi_t)\\
L(\phi_t)&=L_1(\phi_t)+L_2(\phi_t)\\
L_1(\phi_t)&=\sum^{2}_{i=1}2\partial_i(A_i\phi_t)+(\sum^{2}_{l=1} A^2_l-\partial_lA_l )\phi_t\\
L_2(\phi_t)&=\sum^{2}_{i=1} (\phi_t\diamond \phi_i)\diamond\phi_i{\mathcal{G}}.
\end{align*}
By Duhamel principle, $\phi_t$ can be written as
\begin{align}
\phi_t=e^{s\Delta}\phi_t(\upharpoonright_{s=0})+\int^{s}_0e^{(s-\tau)\Delta} L(\phi_t(\tau))d\tau.
\end{align}
By a uniqueness  argument as [\cite{BIKT}, Lemma 5.6], in order to prove (\ref{cx0}), it suffices to show
 \begin{align}
&\|P_k\phi_t(s)\|_{L^{4}_{t,x}}\lesssim b_k(\sigma)2^{-(\sigma-1)k}(1+2^{2k}s)^{-2}\label{hb2}\\
&\Longrightarrow \nonumber\\
&\int^{s}_0\|e^{(s-\tau)\Delta} L(\phi_t(\tau))\|_{{L^4}_{t,x}}d\tau\lesssim \varepsilon^2b_k(\sigma)2^{-(\sigma-1)k}(1+2^{2k}s)^{-2}.\label{hbl}
\end{align}
The $L_1(\phi_t)$ part of (\ref{hbl}) has been done in  [\cite{BIKT}, Lemma 5.6]. It suffices to prove (\ref{hbl}) for $L_2(\phi_t)$ under the assumption of (\ref{hb2}). Recall also that $\mathcal{G}=\Gamma^{\infty}+\mathcal{\widetilde{G}}$ satisfies
\begin{align}\label{Gfb}
&\|P_k(\mathcal{G}-\Gamma^{\infty})\|_{F_{k}(T)}\lesssim 2^{-\sigma k}(1+2^{j+k})^{-7}T_{k,j}.
\end{align}
By the proof of [\cite{BIKT}, Lemma 5.6],
\begin{align}
&\|P_k(\phi_t(s)\diamond\phi_i\diamond\phi_l)\|_{L^{4}_{t,x}}\lesssim b^2_k 2^{-(\sigma-3)k}(1+2^{2k}s)^{-2}(s2^{2k})^{-\frac{7}{8}}b_k(\sigma).\label{kncde}
\end{align}
Then the $\Gamma^{\infty}$ part of $L_2(\phi_t)$ follows directly by proof of [\cite{BIKT}, Lemma 5.6].

Denote ${\bf P}=\phi_t(s)\diamond\phi_i\diamond\phi_l$. In order to control ${\bf P}(\mathcal{G}-\Gamma^{\infty})$, we first control $\|(\phi_i\diamond\phi_l)\mathcal{\widetilde{G}}\|_{S^{\frac{1}{2}}_k(T)}$. We have seen
\begin{align}
&\|P_k(\phi_i\diamond \phi_l)\|_{F_k(T)\bigcap S^{\frac{1}{2}}_k(T)}\lesssim
2^{-\sigma k}(1+2^{2k+2j})^{-4}2^{-j}b_{-j}b_{\max(k,-j)}(\sigma).
\end{align}
Thus applying bilinear Littlewood-Paley decomposition, we obtain by (\ref{Gfb}) that
 \begin{align*}
&\|P_k(\phi_i\diamond \phi_l\mathcal{\widetilde{G}})\|_{F_k(T)\bigcap S^{\frac{1}{2}}_k(T)}\lesssim
\sum_{|k_1-k|\le 4}\|P_{k_1}(\phi_i\diamond \phi_l)\|_{F_k(T)\bigcap S^{\frac{1}{2}}_k(T)}\|P_{\le k-4}\mathcal{\widetilde{G}}\|_{L^{\infty}}\\
&+\sum_{|k_1-k_2|\le 8,k_1,k_2\ge k-4}2^{\frac{k+k_1}{2}}\|P_{k_1}(\phi_i\diamond \phi_l)\|_{F_k(T)\bigcap S^{\frac{1}{2}}_k(T)}\|P_{k_2}\mathcal{\widetilde{G}}\|_{F_{k_2}(T)}\\
&+\sum_{|k_2-k|\le 4,k_1\le k-4}2^{k_1}\|P_{k_1}(\phi_i\diamond \phi_l)\|_{F_k(T)\bigcap S^{\frac{1}{2}}_k(T)}\|P_{k_2}\mathcal{\widetilde{G}}\|_{F_{k_2}(T)}\\
&\lesssim 2^{-\sigma k}b_{k}(\sigma)b_{k}2^{\delta|k+j|}\left(1_{k+j\le 0}2^{-j}+2^{k}1_{k+j\ge 0}(1+2^{k+j})^{-7}\right).
\end{align*}
Then
using Lemma \ref{90} with $\omega=\frac{1}{2}$ and (\ref{Gfb}),  ${\bf P}\mathcal{\widetilde{G}}$ is dominated by
 \begin{align*}
&\|P_k\left(P_{k_1}\phi_tP_{k_2}(\phi_i\diamond \phi_l\mathcal{{\widetilde{G}}})\right)\|_{L^4_{t,x}}\lesssim
\sum_{|k_1-k|\le 4,k_2\le k-4}2^{k_2}\|P_{k_1}\phi_t\|_{L^4_{t,x}}\|P_{k_2}(\phi_i\diamond \phi_l\mathcal{\widetilde{G}})\|_{S^{\frac{1}{2}}_k(T)} \\
&+2^{k}\sum_{|k_1-k_2|\le 8.k_2,k_1\ge k-4}2^{-\frac{1}{2}(k_1-k)} \|P_{k_1}\phi_t\|_{L^{4}_{t,x}}\|P_{k_2}(\phi_i\diamond \phi_l\mathcal{\widetilde{G}})\|_{S^{\frac{1}{2}}_k(T)}\\
&+\sum_{|k_2-k|\le 4,k_1\le k-4}2^{\frac{k+k_1}{2}}\|P_{k_1}\phi_t\|_{L^4_{t,x}}\|P_{k_2}(\phi_i\diamond \phi_l\mathcal{\widetilde{G}})\|_{S^{\frac{1}{2}}_k(T)}\\
&\lesssim 2^{-\sigma k}b_{k}(\sigma)b_{k}2^{\delta|k+j|}1_{k+j\le 0}\left(2^{\frac{3}{2}k-\frac{3}{2}j}+2^{2k-j}\right)\\
&+ 2^{-\sigma k}b_{k}(\sigma)b_{k}1_{k+j\ge 0}\left(2^{\delta|k+j|}2^{3k}(1+2^{k+j})^{-10}
+2^{\frac{3}{2}(k-j)}(1+2^{k+j})^{-7}+2^{k-2j}(1+2^{k+j})^{-4}\right)
\end{align*}
for $s\in[2^{2j-1},2^{2j+1})$, $j,k\in\Bbb Z$.
As a summary, inserting this bound to the following heat estimates
  \begin{align*}
&\int^{\widetilde{s}}_0\left\|e^{(\widetilde{s}-s)\Delta}P_k[\phi_t\diamond \phi_i\diamond \phi_l \mathcal{\widetilde{G}}]\right\|_{L^{4}_{t,x}}ds\lesssim  \int^{\widetilde{s}}_{0}(1+|\widetilde{s}-s|)^{-N} (...)ds,
\end{align*}
we conclude that
  \begin{align*}
\left\|P_k\int^{s}_0e^{(s-\tau)\Delta} L_2(\phi_t(\tau))d\tau\right\|_{L^4_{t,x}}\lesssim \varepsilon(1+2^{2k}s)^{-2}2^{-\sigma k+k}b_{k}(\sigma).
\end{align*}
Thus since $L_1$ has been done before, we finish our proof.
\end{proof}

\begin{Lemma}\label{Zoom3}
With the assumptions of Proposition \ref{Parabolic}, the bootstrap assumption  (\ref{FIL}) can be improved to be
\begin{align*}
2^{\frac{1}{2}k}\|P_{k }\widetilde{\mathcal{G}}^{(1)}\|_{L^{4}_{x}L^{\infty}_t}\lesssim h_k  (1_{k+j\ge 0}(1+s2^{2k})^{-20}+2^{\delta|k+j|}1_{k+j\le 0}),
\end{align*}
for any $k,j\in\Bbb Z$, $s\in[2^{2j-1},2^{2j+1})$.
\end{Lemma}
\begin{proof}
The proof of Lemma \ref{connection} has shown for any $\sigma\in[0,\frac{99}{100}]$, $k,j\in \Bbb Z$, $s\in [2^{2j-1}, 2^{2j+1})$,
\begin{align}\label{Ccdfcv}
\|P_{k }(\mathcal{G}D_i\phi_i)\|_{L^4\bigcap L^{\infty}_tL^2_x}&\lesssim h_k(\sigma)2^{-\sigma k+k}(1+s2^{2k})^{-3}1_{j+k\ge 0}+2^{-j}2^{\delta|k+j|}1_{k+j\le 0}.
\end{align}
Meanwhile, Lemma \ref{Zoom2} yields
\begin{align}
\|P_{k }\phi_t\|_{L^{4}}&\lesssim b_{k}2^{k}(1+s2^{2k})^{-2}.
\end{align}
Recall that $b_{k}\le \varepsilon^{\frac{1}{2}}$, for any $k\in\Bbb Z$.
Then bilinear Littlewood-Paley decomposition shows
\begin{align*}
\|P_{k }(\phi_t(D_i\phi_i\mathcal{G}))\|_{L^{4}}&\lesssim h_k(\sigma)2^{-\sigma k+3k}(1+s2^{2k})^{-2}1_{k+j\ge 0}+2^{-2j}2^{k-\sigma k}2^{2\delta|k+j|}h_{k}(\sigma)h_{k}1_{k+j\le 0}
\end{align*}
for $i=1,2$, any $\sigma\in[0,\frac{99}{100}]$, $k,j\in \Bbb Z$, $s\in [2^{2j-1}, 2^{2j+1})$. Here, in the ${\rm High\times Low}$ interaction of $\phi_t(D_i\phi_i\mathcal{G})$ we use
\begin{align*}
&\sum_{|k_1-k|\le 4,k_2\le k-4}\|P_{k }(P_{k_1}\phi_tP_{k_2}(D_i\phi_i\mathcal{G}))\|_{L^{4}}\\
&\lesssim b_{k}2^{k}\sum_{k_2\le k-4} h_{k_2}(\sigma)2^{-\sigma k_2+2k_2}\left((1+s2^{2k_2})^{-3}1_{j+k_2\ge 0}+2^{-j}2^{\delta|k_2+j|}1_{k_2+j\le 0}\right).
\end{align*}
The other two frequency interactions are standard.
Thus
\begin{align}
\int^{\infty}_s\|P_{k }(\phi_tD_i\phi_i)\mathcal{G}\|_{L^{4}}ds'&\lesssim h_k(\sigma)2^{-\sigma k+k}(1+2^{2k+2k_0})^{-1}1_{k+k_0\ge 0}\nonumber\\
&+ 2^{k-\sigma k}h_{k}(\sigma)(1+2^{2\delta|k+k_0|}h_{k}) 1_{k+k_0\le 0}\label{Pl90mn}
\end{align}
for any $\sigma\in[0,\frac{99}{100}]$, $k,k_0\in \Bbb Z$, $s\in [2^{2k_0-1}, 2^{2k_0+1})$. Recall that
\begin{align*}
A_t=\int^{\infty}_s (\phi_t\diamond\phi_s)\mathcal{G} ds'; \mbox{ }\phi_s=\sum_{i=1,2}D_i\phi_i,
\end{align*}
we see $\|P_{k}A_t\|_{L^4}$ is bounded by the RHS of (\ref{Pl90mn}).
By the schematic formula
\begin{align*}
\partial_t (\widetilde{\mathcal{G}}^{(1)})= \phi_t \mathcal{G}^{(2)} +A_t \mathcal{G}^{(1)},
\end{align*}
and the bounds
\begin{align*}
\|P_{k}(\widetilde{\mathcal{G}}^{(l)})\|_{L^{4}\cap L^{\infty}_tL^2_x} \lesssim (1+s2^{2k})^{-30}2^{-\sigma k-k}h_{k}(\sigma),\mbox{ }\forall \mbox{ } l=1,2,
\end{align*}
we deduce from bilinear Littlewood-Paley decomposition that
\begin{align*}
\|P_{k}\partial_t (\widetilde{\mathcal{G}}^{(1)})\|_{L^4}\lesssim h_k 2^{ k}(1+2^{2\delta|k+k_0|}h_{k}1_{k+k_0\le 0}).
\end{align*}
Then by Gagliardo-Nirenberg inequality we get
\begin{align*}
 2^{\frac{1}{2}k}\|P_{k} (\widetilde{\mathcal{G}}^{(1)})\|_{L^4_xL^{\infty}_t}&\lesssim \|P_{k} (\widetilde{\mathcal{G}}^{(1)})\|^{\frac{3}{4}}_{L^4}\|\partial_tP_{k} (\widetilde{\mathcal{G}}^{(1)})\|^{\frac{1}{4}}_{L^4}\\
&\lesssim
h_k   (1+s2^{2k})^{-20}1_{k+k_0\ge 0}+h_k  2^{\delta|k+k_0|} 1_{k+k_0\le 0}.
\end{align*}
for  $k_0,k\in\Bbb Z$, $s\in [2^{2k_0-1}, 2^{2k_0+1})$.
\end{proof}

\subsection{End of proof of Proposition \ref{Parabolic}}

By Lemma \ref{Zoom3}, the assumption (\ref{FIL}) in Lemma \ref{Zoom0}, Lemma \ref{Zoom1} and Lemma \ref{Zoom2} can be dropped. In fact, let
\begin{align*}
\widetilde{\Phi}(T'):=\sup\limits_{k,j\in\Bbb Z}\sup\limits_{s\in[2^{2j-1},2^{2j+1})}h^{-1}_k \left(1_{k+j\ge 0}(1+s2^{2k})^{-20}+2^{\delta|k+j|}1_{k+j\le 0}\right)^{-1}2^{\frac{1}{2}k}\|P_{k }\widetilde{\mathcal{G}}^{(1)}\|_{L^{4}_{x}L^{\infty}_t(T')}
\end{align*}
Lemma \ref{FFF} and  Sobolev embeddings imply $\widetilde{\Phi}$ is an increasing continuous function on $T'\in[0,T]$. Lemma \ref{Zoom3} shows
$\widetilde{\Phi}(T')\le \varepsilon^{-\frac{1}{4}}\Longrightarrow \widetilde{\Phi}(T')\lesssim 1$. Then by Bernstein inequality and letting $T'\to 0$, it remains to prove
\begin{align*}
 \|P_{k }\widetilde{\mathcal{G}}^{(1)}\|_{L^{2}_{x}}\lesssim h_k \left(1_{k+j\ge 0}(1+s2^{2k})^{-20}+2^{\delta|k+j|}1_{k+j\le 0}\right)
\end{align*}
along the heat flow initiated from $u_0$ for any ${k,j\in\Bbb Z}$, $s\in[2^{2j-1},2^{2j+1})$. This follows by (\ref{3.16ab}).

By Lemma \ref{FFF} and similar arguments, assumption (\ref{4.1}) can be also dropped.
Thus Lemma \ref{Zoom0}, Lemma \ref{Zoom1} and Lemma \ref{Zoom2} all hold with only  assuming (\ref{tian}), (\ref{tian2})  of Proposition \ref{Parabolic}.
The left for Proposition \ref{Parabolic} is to prove the $L^{2}_{t,x}$ bound for $A_t$.
\begin{Lemma}
With the assumptions (\ref{tian}), (\ref{tian2}) of Proposition \ref{Parabolic}, for all $k\in \Bbb Z$, one has
\begin{align}
\|P_kA_t\upharpoonright_{s=0}\|_{L^{2}_{t,x}}&\lesssim \varepsilon {b}_{k}(\sigma)2^{- \sigma k} \mbox{ }{\rm{if}}\mbox{  }\sigma\in[\frac{1}{100},\frac{99}{100}]\label{cx01}\\
\|A_t\upharpoonright_{s=0}\|_{L^{2}_{t,x}}&\lesssim \varepsilon^2\mbox{ }{\rm{if}}\mbox{  }\sigma\in[0,\frac{99}{100}]\label{cx02}.
\end{align}
\end{Lemma}
\begin{proof}
Recall that [\cite{BIKT}, Lemma 5.7] has proved
 \begin{align}
\|P_k(\phi_t\diamond \phi_s)\|_{L^{2}_{t,x}}&\lesssim   \sum_{l\le k}2^{- \sigma l}2^{l+k}b_{k}(\sigma)b_l(s2^{2l})^{-\frac{3}{8}}(1+s2^{2k})^{-2}
+ \sum_{l\ge k}2^{- \sigma l}2^{2l}b_{l}(\sigma)b_l(s2^{2l})^{-\frac{3}{8}}(1+s2^{2l})^{-4}\label{cx03}.
\end{align}
Denote the RHS of (\ref{cx03}) by $\mathbf{a}_k(\sigma)$ for simplicity.

Since $A_t(0)=\int^{\infty}_0 (\phi_t\diamond\phi_s)\mathcal{G}ds$, (\ref{cx02}) follows by directly applying  [\cite{BIKT}, Lemma 5.7]  since $\|\mathcal{G}\|_{L^{\infty}}\lesssim K(\mathcal{N})$. For (\ref{cx01}), we need to clarify the frequency interaction between $(\phi_t\diamond\phi_s)$ and $\mathcal{G}$ as before. The constant part of $\mathcal{G}$ follows by (\ref{cx03}). It remains to deal with the $\mathcal{\widetilde{G}}$ part.  In the $\rm High\times Low$ part of $P_k[(\phi_t\diamond\phi_s)\mathcal{\widetilde{G}}]$, we have
\begin{align*}
&\sum_{|k_1-k|\le 4}\|P_{k_1}(\phi_t\diamond \phi_s)P_{\le k-4}\mathcal{\widetilde{G}}\|_{L^2_{t,x}}
\lesssim  \sum_{|k_1-k|\le 4}\|P_{k_1}(\phi_t\diamond \phi_s)\|_{L^2_{t,x}}\|\mathcal{\widetilde{G}}\|_{L^{\infty}}
\lesssim  \mathbf{a}_k(\sigma).
\end{align*}
Then the $\rm High\times Low$ part is acceptable by directly repeating  [\cite{BIKT}, Lemma 5.7].

\noindent From now on to the end of this lemma, we assume $\sigma\in[\frac{1}{100},\frac{99}{100}]$. In the $\rm High\times High$ part of $P_k[(\phi_t\diamond\phi_s)\mathcal{\widetilde{G}}]$,
by $\|P_kf\|_{L^{\infty}} \lesssim 2^k\|f\|_{F_k}$ and (\ref{CXbZ}), we have
 \begin{align}
&\sum_{|k_1-k_2|\le 8,k_1,k_2\ge k-4}\|P_{k_1}(\phi_t\diamond \phi_s)P_{ k_2}\mathcal{\widetilde{G}}\|_{L^2_{t,x}}\nonumber\\
&\lesssim  \sum_{|k_1-k|\le 8,k_1,k_2\ge k-4}\|P_{k_1}(\phi_t\diamond \phi_s)\|_{L^2_{t,x}}2^{k_2}\|\mathcal{\widetilde{G}}\|_{F_{k_2}}
\lesssim \sum_{ k_1\ge k-4}\mathbf{a}_{k_1}(0)(1+s2^{2k_1})^{-3}2^{-\sigma k_1}b_{k_1}(\sigma)\nonumber\\
&\lesssim  \sum_{ k_1\ge k-4}2^{-\sigma k_1}b_{k_1}(\sigma)(1+s2^{2k_1})^{-3}\sum_{l\le k_1} 2^{l+k_1}b_{k_1}b_l(s2^{2l})^{-\frac{3}{8}}(1+s2^{2k_1})^{-2}\label{huas}\\
&+  \sum_{ k_1\ge k-4}2^{-\sigma k_1}b_{k_1}(\sigma)(1+s2^{2k_1})^{-3}\sum_{l\ge k_1}2^{2l}b_{l}b_l(s2^{2l})^{-\frac{3}{8}}(1+s2^{2l})^{-4}.\label{huasq}
\end{align}
Thus for $j\in\Bbb Z$, $s\in[2^{2j-1},2^{2j+1})$, when $k+j\ge 0$, the above formula is bounded by
\begin{align*}
(1+ 2^{2j+2k})^{-2}2^{2k-\sigma k}b_{k}b_{k}(\sigma)(2^{2k+2j})^{-\frac{3}{8}}.
\end{align*}
When  $k+j\le 0$, by (\ref{huas}), (\ref{huasq}),  the $\rm High\times High$ part is  dominated by
 \begin{align*}
&( {\sum\limits_{{k_1} \ge - j} { + \sum\limits_{k-4 \le  {k_1} \le  - j} {} } } ){(1 + {2^{{\text{2}}{k_1} + 2j}})^{ - 3}}{2^{{\text{2}}{k_1} - \sigma {k_1}}}{b^2_{{k_1}}}{b_{{k_1}}}(\sigma ){2^{ - \frac{3}
{4}({k_1} + j)}} \hfill \\
 & + \sum\limits_{{k_1} \ge k-4} {{1_{{k_1} + j \ge  0}}{{(1 + {2^{{\text{2}}{k_1} + 2j}})}^{ - 3}}{2^{ - \sigma {k_1}}}} b_{k_1}(\sigma)\big[ {\sum\limits_{l \ge  {k_1}} {} {2^{2l}}{b_l}{b_l}{2^{ - \frac{3}
{4}(j + l)}}{{(1 + {2^{2l + 2j}})}^{ - 4}}} \big] \hfill \\
  &+ \sum\limits_{{k_1} \ge  k-4} {{1_{{k_1} + j \le  0}}{{(1 + {2^{{\text{2}}{k_1} + 2j}})}^{ - 3}}} {2^{ - \sigma {k_1}}} b_{k_1}(\sigma) \big[ {(\sum\limits_{l \ge   - j} { + \sum\limits_{{k_1} \le l \le  - j} {} )}{2^{2l}}{b_l}{b_l}{2^{ - \frac{3}{4} (j + l)}}{{(1 + {2^{2l + 2j}})}^{ - 4}}} \big],
  \end{align*}
 which is further bounded by
\begin{align*}
   & \sum\limits_{{k_1} \ge  - j}{b_{k_1}}(\sigma){b^2_{-j}}  {{2^{2\delta |{k_1} + j|}}} {2^{2{k_1} - \sigma {k_1}}}{2^{ - \frac{3}
{4}({k_1} + j)}}{2^{ - 6({k_1} + j)}} +  \sum\limits_{k \le {k_1} \le   - j}{b_{k_1}}(\sigma){b^2_{ - j}}2^{-\sigma k_1} {{2^{2\delta |{k_1} + j|}}} 2^{2{k_1}}{2^{ - \frac{3}
{4}({k_1} + j)}} \hfill \\
  &+ {b^2_{ - j}}{b_{ - j}}(\sigma )\sum\limits_{{k_1} \ge  k-4} {{1_{{k_1} + j \geqslant 0}}{{(1 + {2^{{\text{2}}{k_1} + 2j}})}^{ - 7}}} {2^{2{k_1} - \sigma {k_1}}}{2^{\delta |{k_1} + j|}}{2^{ - \frac{3}
{4}(j + {k_1})}}  \\
  &+ \sum\limits_{{k_1} \ge k-4} {{1_{{k_1} + j \le 0}}{b^2_{ - j}}{b_{k_1}}(\sigma )} \big[ {{2^{ - \sigma {k_1}}}{2^{ - 2j}} + {2^{ - \sigma {k_1}}}\sum\limits_{{k_1} \le  l \le   - j} {{2^{2l}}} {2^{\delta |{l} + j|}}{2^{ - \frac{3}
{4}(j + {l})}}} \big] \\
  &\lesssim {b^2_{ - j}}{b_{ - j}}(\sigma ){2^{\sigma j}}{2^{ - 2j}} + \sum\limits_{{k_1} \ge k-4} {{1_{{k_1} + j \le  0}}{b^2_{-j}}{b_{k_1}}(\sigma )} {2^{ - \sigma {k_1}}}{2^{ - 2j}}  \\
  &\lesssim {b^2_{ - j}}{b_{-j}}(\sigma ){2^{\sigma j}}{2^{ - 2j}} +  2^{-\sigma k}b_{k}(\sigma)b^{2}_{-j}{2^{ - 2j}},
\end{align*}
where in the last line we used $\sigma\ge \frac{1}{100}$.
Summing over $j\ge k_0$ we see the $\rm High\times High$ part satisfies
 \begin{align*}
&\int^{\infty}_0 \sum_{|k_1-k_2|\le 8, k_1,k_2\ge k-4}\|P_{k}[P_{k_1}(\phi_t\diamond \phi_s) P_{k_2}\mathcal{\widetilde{G}}]\|_{L^2_{t,x}}ds'\\
&\lesssim   \sum_{j\le -k} 2^{\sigma j}{b}_{-j}(\sigma)b^{2}_{-j}+\sum_{j\le -k}b^{2}_{-j} 2^{-\sigma k}{b}_{k}(\sigma)+\sum_{j\ge -k} (1+ 2^{2j+2k})^{-2}2^{2k+2j-\sigma k}b_{k}b_{k}(\sigma)(2^{2k+2j})^{-\frac{3}{8}}\\
&\lesssim \varepsilon^2   2^{-\sigma k}{b}_{k}(\sigma),
\end{align*}
where we applied $\sigma\ge\frac{1}{100}$ in the last line again.
Now let us consider the $\rm Low\times High$ part of $P_{k}[(\phi_t\diamond\phi_s)\mathcal{\widetilde{G}}]$. By the same reason as $\rm High\times High$, the  $\rm Low\times High$ part is dominated by
  \begin{align*}
\int^{\infty}_0\sum_{|k_2-k|\le 4}\|P_{\le {k-4}}(\phi_t\diamond \phi_s)P_{ k_2}\mathcal{\widetilde{G}}\|_{L^2_{t,x}}ds'&\lesssim  \int^{\infty}_0\sum_{|k_2-k|\le 4}\|P_{\le {k-4}}(\phi_t\diamond \phi_s)\|_{L^2_{t,x}}2^{k_2}\|P_{ k_2}\mathcal{\widetilde{G}}\|_{L^{\infty}_{t}L^2_{x}}ds'\\
&\lesssim    b_{k}(\sigma)2^{-\sigma k} \int^{\infty}_0 \| (\phi_t\diamond \phi_s)\|_{L^2_{t,x}}ds'\lesssim    b_{k}(\sigma)2^{-\sigma k} \varepsilon^2,
\end{align*}
where we applied (\ref{CXbZ}) and $2^{k}T_{k,j}(1+2^{j+k})^{-1}\lesssim 1$ in the third inequality.
\end{proof}

\section{Evolution along the Schr\"odinger map flow direction}

In this section, we prove the following proposition, which is the key to close the bootstrap for solutions in the Schr\"odinger evolution direction.
\begin{Proposition}\label{4.5}
Assume that  $\sigma\in [0,\frac{99}{100}]$.
Let $Q\in \mathcal{N}$ be a fixed point and $\epsilon_0$ be a sufficiently small constant.  Given any $\mathcal{L}\in \Bbb Z_+$, assume that $T\in(0,2^{2\mathcal{L}}]$. Let $\{c_k\}$ be an $\epsilon_0$-frequency envelope of order $\delta$. And let $\{c_k(\sigma)\}$ be another frequency envelope of order $\delta$. Let $u\in \mathcal{H}_Q(T)$ be the solution to SMF with initial data $u_0$ which satisfies
\begin{align}
\|P_{k}\nabla u_0\|_{L^2_x}&\le c_k\label{w1kxxxjn}\\
\|P_{k}\nabla u_0\|_{L^2_x}&\le c_k(\sigma)2^{-\sigma k}\label{w2kxxxjn}
\end{align}
Denote $\{\phi_i\}$ the corresponding differential fields of the heat flow initiated from $u$. Suppose also that at the heat initial time $s=0$,
\begin{align}
\|P_{k}\phi_i\|_{G_k(T)}&\le \epsilon^{-\frac{1}{2}}_0c_k.
\end{align}
Then when $s=0$, we have for all $i=1,2$, $k\in\Bbb Z$,
\begin{align}
\|P_{k}\phi_i\|_{G_k(T)}&\lesssim c_k\label{4.32}\\
\|P_{k}\phi_i\|_{G_k(T)}&\lesssim c_k(\sigma)2^{-\sigma k}.\label{4.33}
\end{align}
\end{Proposition}

The proof of Proposition \ref{4.5} will be divided into several lemmas.
First of all, Corollary \ref{JkL} shows
\begin{align}\label{bbcc}
\sum^{2}_{i=1}\|P_{k}\phi_i(\upharpoonright_{s=0,t=0})\|_{L^2_x}\lesssim 2^{-\sigma k} c_{k}(\sigma)
\end{align}
for any $k\in\Bbb Z$, $\sigma\in[0,\frac{99}{100}]$.

Second,  we reduce the proof to frequency envelope bounds.
Let
\begin{align}
b(k)&=\sum^{2}_{i=1}\|P_{k}\phi_i\upharpoonright_{s=0}\|_{G_k(T)}.
\end{align}
For $\sigma\in[0,\frac{99}{100}]$, define the frequency envelopes:
\begin{align}
b_k(\sigma)&=\sup_{k'\in\Bbb Z}2^{\sigma k'}2^{-\delta|k-k'|}b(k').
\end{align}
By Proposition \ref{mei} and Sobolev embeddings,  they are finite and $\ell^2$ summable. And
\begin{align}\label{dscx45}
\|P_{k}\phi_i\upharpoonright_{s=0}\|_{G_k(T)}&\lesssim 2^{-\sigma k}b_k(\sigma).
\end{align}

{\bf To prove (\ref{4.32}) and (\ref{4.33}), it suffices to show}
\begin{align}
b_k(\sigma)&\lesssim c_k(\sigma).
\end{align}

By (\ref{4.1}), we have $b_k\le \varepsilon^{-\frac{1}{2}}_0c_k$, and particularly,
\begin{align}
\sum_{k\in\Bbb Z}b^2_{k}\le \epsilon_0.
\end{align}
The assumption (\ref{tian})  of Proposition 4.1 follows from the inclusion $G_k\subset F_k$. The following lemma will show the assumption (\ref{tian2}) holds as a corollary of (\ref{dscx45}) if $u$ solves SMF.

\begin{Lemma}
If $\{b_k(\sigma)\}$ are defined above. Then the field $\phi_t$ at heat initial time $s=0$ satisfies
\begin{align}\label{VXzz}
\|P_{k}\phi_t\upharpoonright_{s=0}\|_{L^{4}_{t,x}}\lesssim  b_{k}(\sigma)2^{-(\sigma-1)k}.
\end{align}
\end{Lemma}
\begin{proof}
When $s=0$, $\phi_t(0)=\sqrt{-1}\sum^{d}_{i=1}\partial_i\phi_i(0)+A_i(0)\phi_i(0)$. The terms $\psi_i(0),A_i(0)$ are estimated before in Section 4. Thus copying the proof of [\cite{BIKT},Lemma 6.1] gives (\ref{VXzz}).
\end{proof}

Thus both the assumption (\ref{tian}) and  the assumption (\ref{tian2}) of Proposition 4.1 are verified. Now one can apply Proposition 4.1, since (4.3) can be dropped. We summarize the results in the following:
\begin{align}
\left\{
  \begin{array}{ll}
   \|P_k(\phi_i(s))\|_{F_k(T)}\lesssim 2^{-\sigma k}b_k(\sigma) (1+2^{2k}s)^{-4}, & \hbox{  } \\
     \|P_k(D_i\phi_i(s))\|_{F_k(T)}\lesssim 2^{k}2^{-\sigma k}b_k(\sigma) (s2^{2k})^{-\frac{3}{8}}(1+2^{2k}s)^{-2}, & \hbox{  }
  \end{array}
\right.
\end{align}
and for $F\in \{\psi_i \diamond\psi_j , A^2_{l}  \}^{2}_{l.i,j=1}\upharpoonright_{s=0}$
\begin{align}
     \|P_kF\|_{L^{2}_{t,x}}\lesssim 2^{-\sigma k}b^2_{>k}(\sigma),\mbox{  }\|F\|_{L^{2}_{t,x}}\lesssim \epsilon_0.
\end{align}
The at $s=0$, $A_t$ satisfies
\begin{align*}
\|A_t(0)\|_{L^{2}_{t,x}}&\lesssim \epsilon_0,\mbox{  }{\rm{if}}\mbox{  }\sigma\in[0,\frac{99}{100}]\\
     \|P_k A_t(0)\|_{L^{2}_{t,x}}&\lesssim 2^{-\sigma k}b_{k}(\sigma),\mbox{  }{\rm{if}}\mbox{  }\sigma\in[\frac{1}{100},\frac{99}{100}]
\end{align*}
Recall that when $s=0$ the evolution equation of  $\phi_i$ along the Schr\"odinger map flow direction (see Lemma \ref{asdf}) is :
\begin{align}\label{Final}
-\sqrt{-1}D_t\phi_i&= \sum^{2}_{j=1}D_jD_j\phi_i +\sum^{2}_{j=1} \mathcal{R}(\phi_i,\phi_j)\phi_j.
\end{align}

\subsection{Control of nonlinearities}

Now let us deal with the nonlinearities in (\ref{Final}).
{\bf In this section we always assume $s=0$. }

Denote
\begin{align}
L'_{j}=A_t\phi_j+\sum^{2}_{i=1}A^2_i\phi_j+2\sum^{2}_{i=1}\partial_i(A_i\phi_j)-\sum^{2}_{i=1}(\partial_iA_i)\phi_j.
\end{align}

\begin{Proposition}[\cite{BIKT}]\label{V1}
For all $j\in\{1,2\}$ and $\sigma\in[0,\frac{99}{100}]$ we have
\begin{align}
\|P_k(L'_{j}) \upharpoonright_{s=0}\|_{N_k(T)}&\lesssim \epsilon_02^{-\sigma k}b_k(\sigma)\label{aCXz12}\\
\sum^{2}_{j_0,j_1,j_3=1}\|P_k\left(  \phi_{j_0}\diamond\phi_{j_1} \diamond\phi_{j_3}\right)  \upharpoonright_{s=0}\|_{N_k(T)}&\lesssim \epsilon_02^{-\sigma k}b_k(\sigma).\label{CXz122}
\end{align}
\end{Proposition}
\begin{proof}
(\ref{aCXz12}) and (\ref{CXz122}) have been proved in [\cite{BIKT}, Proposition 6.2]. We emphasize that to bound $\|A_t\phi_i\|_{N_k}$,  [\cite{BIKT}, Proposition 6.2] used $\|A_t\|_{L^2_{t,x}}\le \varepsilon^2$ when $\sigma \in[0,\frac{1}{12}]$ and $\|P_{k}A_t\|_{L^2_{t,x}}\le 2^{-\sigma k}b_{k}(\sigma)$ when $\sigma \ge \frac{1}{12}$. Thus our bounds  (\ref{cx01}), (\ref{cx02}) suffice to bound $\|A_t\phi_i\|_{N_k}$ as well although (\ref{cx01})-(\ref{cx02})  itself differs from the bounds stated by   [\cite{BIKT}, Lemma 5.7].
\end{proof}

Now we turn to the remained curvature term in (\ref{Final}).
\begin{Proposition}\label{CXz12}
For all $k\in \Bbb Z$ and $\sigma\in[0,\frac{99}{100}]$ we have
\begin{align}\label{jnhbg}
\sum^{2}_{j_0,j_1,j_3=1}\|P_k\left((\phi_{j_0}\diamond\phi_{j_1} \diamond\phi_{j_3})\mathcal{G}\right) \|_{N_k(T)}\lesssim 2^{-\sigma k}\epsilon_0 b_{k}(\sigma).
\end{align}
\end{Proposition}
\begin{proof}
Recall $\mathcal{G}=\Gamma^{\infty}+\mathcal{\widetilde{G}}$. The constant part $\Gamma^{\infty}$ satisfies (\ref{jnhbg}) by directly applying (\ref{CXz122}). It suffices to control $\mathcal{\widetilde{G}}$ part.

As a preparation, we first prove the following estimate
\begin{align}\label{1Cpo}
&\sum^{2}_{i=1}\|P_k (\mathcal{\widetilde{G}}\phi_i)\|_{F_k(T)} \lesssim \left\{
                                                                           \begin{array}{ll}
                                                                             2^{-\sigma k}b_{k}(\sigma), & \hbox{ } \frac{1}{100}<\sigma\le \frac{99}{100}\\
                                                                           2^{-\sigma k}\sum_{j\ge k}b_{j}b_j(\sigma), & \hbox{  }0\le \sigma\le \frac{1}{100}
                                                                           \end{array}
                                                                         \right.
\end{align}
This follows directly by applying Corollary 4.1 and [Lemma 5.1, \cite{BIKT}]: If $\sigma>\frac{1}{100}$, then
  \begin{align*}
\|P_k (\mathcal{\widetilde{G}}\phi_i)\|_{F_k(T)}
& \lesssim 2^{-\sigma k}b_{k}(\sigma)+  2^{-k-\sigma k}b_k(\sigma)\sum_{l\le k}2^{\delta|k-l|}2^{l}b_l +b_k(\sigma)\sum_{j\ge k}  2^{-\sigma j}2^{2\delta|k-j|}\\
& \lesssim  2^{-\sigma k}b_{k}(\sigma).
\end{align*}
If $\sigma\in[0,\frac{1}{100}]$, for the $\rm High\times High$ interaction we directly use
 \begin{align*}
\sum_{|k_1-k_2|\le 8,k_1,k_2\ge k-4}\|P_k (P_{k_1}\mathcal{\widetilde{G}}P_{k_2}\phi_i)\|_{F_k(T)}
&\lesssim  \sum_{ j\ge k-4} 2^{j}(\sum_{|k_1-j|\le 28}\|P_{k_1}\mathcal{\widetilde{G}}\|_{F_{k_1}(T)}) (\sum_{|k_2-j|\le 28}\|P_{k_2}\phi_i\|_{F_{k_2}(T)})\\
&\lesssim  2^{-\sigma k}\sum_{j\ge k}b_{j}b_{j}(\sigma).
\end{align*}
The other two interactions are all the same as $\sigma \ge \frac{1}{100}$. Thus (\ref{1Cpo}) follows.

As before, denoting $\mathbf{F}=\phi_{j_0}\diamond\phi_{j_1}$, by bilinear Littlewood-Paley decomposition, we have
 \begin{align}\label{Cpo}
&\|P_k\left(\mathbf{F}\diamond (\phi_{j_3}\widetilde{\mathcal{G}}) \right)\|_{N_k(T)}\nonumber \\
&=\sum_{|l-k|\le 4}\|{P_k}({P_{ < k - 100}}\mathbf{F}{P_{l}}(\mathcal{\widetilde{G}}\phi_{j_3})) \|_{N_k(T)}+ \sum\limits_{|{k_1} - k| \le  4}\| {P_k}({P_{{k_1}}}\mathbf{F}{P_{   < k - 100}} (\mathcal{\widetilde{G}}\phi_{j_3}))\|_{N_k(T)}  \nonumber \\
&+ \sum\limits_{{k_1},{k_2} \ge k - 100}^{|{k_1} - {k_2}| \le  120}\|  {P_k}({P_{{k_1}}}\mathbf{F}{P_{{k_2}}}(\mathcal{\widetilde{G}}\phi_{j_3}))\|_{N_k(T)}.
\end{align}
For the first RHS term of (\ref{Cpo}), applying (\ref{6a}) and the trivial bound
\begin{align}
\|\mathcal{\widetilde{G}}\|_{L^{\infty}_{t,x}}&\lesssim 1\label{1CXnb}\\
\|\phi_x\|_{L^4_{t,x}}&\lesssim \epsilon_0,\label{2CXnb}
\end{align}
and (\ref{1Cpo}),  for $\sigma\in[\frac{1}{100},\frac{99}{100}]$  we get
 \begin{align*}
\sum_{|k_0-k|\le 4}\|{P_k}({P_{ < k - 100}}\mathbf{F}{P_{k_0}}(\mathcal{\widetilde{G}}\phi_{j_3}))\|_{N_k(T)}&\lesssim \|\phi_{j_0}\phi_{j_1}\|_{L^2_{t,x}}\|P_k(\mathcal{\widetilde{G}}\phi_{j_3}) \|_{F_k(T)}\lesssim   \epsilon_02^{-\sigma k}b_{k}(\sigma)b_{k}.
\end{align*}
For the first RHS term of (\ref{Cpo}), when $\sigma\in[0,\frac{1}{100}]$, one further decomposes $P_{[k-4,k+4]}(\mathcal{\widetilde{G}}\phi_{j_3})$ into $\rm High\times High$, $\rm Low\times High$, $\rm High\times Low$. We schematically write
 \begin{align}
&\sum_{|k_0-k|\le 4}\|{P_k}\left({P_{ < k - 100}}\mathbf{F}{P_{k_0}}(\mathcal{\widetilde{G}}\phi_{j_3})\right)\|_{N_{k}(T)}
\nonumber\\
&\lesssim \sum_{|l-k|\le 8}\|P_k\left((P_{ < k - 100}\mathbf{F}) P_{l}\phi_{j_3}(P_{\le k-8}\mathcal{\widetilde{G}})\right)\|_{N_{k}(T)}\label{1hxjbn}
\\
&+ \sum_{|l-k|\le 8}\|  P_k\left((P_{ < k - 100}\mathbf{F})(P_{\le k-8}\phi_{j_3}P_{l}\mathcal{\widetilde{G}})\right)\|_{N_{k}}\label{2hxjbn}\\
&+\sum_{|k_1-k_2|\le 16,k_1,k_2\ge k-8}\| P_k ((P_{ < k - 100}\mathbf{F}) P_{k_1}\phi_{j_3}(P_{k_2}\mathcal{\widetilde{G}}) )\|_{N_{k}(T)}\label{3hxjbn}
\end{align}
Since for all $\sigma\in[0,\frac{99}{100}]$, the  $\rm Low\times High$ (denoted by $P^{lh}_k$ for short) and $\rm High\times Low$ (denoted by $P^{hl}_k$ for short) interactions lead to $\|(P^{lh}_k+P^{hl}_k)(\mathcal{\widetilde{G}}\phi_{j_3})\|_{F_{k}}\lesssim 2^{-\sigma k}b_{k}(\sigma)$,
we conclude that
\begin{align*}
 (\ref{1hxjbn})+(\ref{2hxjbn}) &\lesssim \|\phi_{j_0}\phi_{j_1}\|_{L^2_{t,x}}\left(\|P^{lh}_k(\mathcal{G}\phi_{j_3}) \|_{F_k(T)}+ \|P^{hl}_k(\mathcal{G}\phi_{j_3}) \|_{F_k(T)}\right)\\
&\lesssim \epsilon_0 2^{-\sigma k}b_{k}(\sigma).
\end{align*}
For the (\ref{3hxjbn}) term, applying (\ref{6c}) yields
 \begin{align*}
 (\ref{3hxjbn})& \lesssim  \sum_{k_2\ge k-8}\sum_{|k_1-k_2|\le 16} \left\|P_k\left[\left((P_{ < k - 100}\mathbf{F})   P_{k_2}\mathcal{\widetilde{G}}\right)P_{k_1}\phi_{j_3} \right]\right\|_{N_k}\\
&\lesssim \sum_{k_2\ge k-8,|k_1-k_2|\le 16} \left \| (P_{ < k - 100}\mathbf{F})  P_{k_2}\mathcal{\widetilde{G}}\right\|_{L^2_{t,x}}2^{\frac{k-k_1}{6}}\|P_{k_1}\phi_{j_3} \|_{G_{k_1}}\\
&\lesssim \sum_{k_1\ge k-12} \left \| \mathbf{F}\right\|_{L^2_{t,x}}2^{\frac{k-k_1}{6}}2^{-\sigma k_1}b_{k_1}(\sigma)\\
&\lesssim \epsilon_0 2^{-\sigma k }b_{k }(\sigma).
\end{align*}
Thus the first RHS term of (\ref{Cpo}) has been done.

For the second RHS term of (\ref{Cpo}), we further divide $\bf F$ into
 \begin{align}
&\sum\limits_{|{k_1} - k| \le  4}  \| {P_k}({P_{{k_1}}}\mathbf{F}{P_{< k - 100}}(\mathcal{G} \phi_{j_3})) \|_{N_k(T)}\nonumber\\
&\lesssim  \sum_{|l-k|\le 8}\| P_{k}(P_{l}\phi_{j_0})(P_{\le k-8}\phi_{j_1})P_{\le k-100}(\mathcal{G} \phi_{j_3}) \|_{N_{k}(T)}\label{D76z}\\
&+\sum_{|l-k|\le 8}\|P_{k}(P_{l}\phi_{j_1})(P_{\le k-8}\phi_{j_0})P_{\le k-100}(\mathcal{G} \phi_{j_3}) \|_{N_{k}(T)}\label{C76z}\\
&+\sum_{|l_1-l_2|\le 16, l_1,l_2\ge k-8}\|P_{k}[(P_{l_1} \phi_{j_0})(P_{l_2}\phi_{j_1})P_{\le k-100}(\mathcal{G} \phi_{j_3})]\|_{N_{k}(T)}.\label{C76z7}
\end{align}
For the first two terms on the RHS, using again  (\ref{6a}) and the bounds (\ref{1CXnb}), (\ref{2CXnb}), we obtain
 \begin{align*}
 (\ref{C76z}) +   (\ref{D76z}) &\lesssim \|P_k(\phi_x)\|_{F_k(T)}\|\phi_x\|^2_{L^4_{t,x}}\lesssim \epsilon_02^{-\sigma k}b_k(\sigma).
\end{align*}
And similarly, for $\sigma>\frac{1}{100}$,
 \begin{align*}
 (\ref{C76z7}) &\lesssim  \|\mathcal{\widetilde{G}} \phi_{j_3}\|_{L^{4}_{t,x}}\|\phi_{j_1}\|_{L^{4}_{t,x}}\sum_{l\ge k-8}\|P_{l}\phi_{j_0}\|_{F_l(T)}\lesssim \epsilon_02^{-\sigma k}b_k(\sigma).
\end{align*}
For $0\le \sigma\le \frac{1}{100}$,  using again (\ref{6c}) and the bounds (\ref{1CXnb}), (\ref{2CXnb}), we have
 \begin{align*}
 (\ref{C76z7}) &\lesssim \sum_{l\ge k} 2^{\frac{k-l}{6}}\left\|\left(P_{\le k-100}(\mathcal{\widetilde{G}} \phi_{j_3})\right)P_{l}\phi_{j_1}\right\|_{L^{2}_{t,x}} \|P_{l}\phi_{j_0}\|_{G_l(T)}\\
&\lesssim  \| \phi_{x}\|^2_{L^{4}_{t,x}}\sum_{l\ge k}2^{\delta|l-k|}2^{\frac{k-l}{6}}2^{-\sigma l}b_{l}b_{l}(\sigma)\lesssim  2^{-\sigma k} b_k(\sigma)\epsilon_0.
\end{align*}
Thus the first two RHS terms of (\ref{Cpo}) are done.

For the third term of (\ref{Cpo}), applying Littlewood-Paley decomposition to $\bf F$ shows
 \begin{align}
&\sum\limits_{{k_1},{k_2} \ge k - 100}^{|{k_1} - {k_2}| \le  120}\| {P_k}({P_{{k_1}}}\mathbf{F}{P_{{k_2}}}(\mathcal{\widetilde{G}} \phi_{j_3}))\|_{N_{k}}\nonumber\\
&=\sum\limits_{{k_1},{k_2} \ge k - 100}^{|{k_1} - {k_2}| \le  120} \sum_{|l-k_1|\le 4}\| {P_k}\left[  P_{l}{\phi_{j_0}}  P_{\le k_1-8}\phi_{j_1} {P_{{k_2}}}(\mathcal{\widetilde{G}} \phi_{j_3}) \right]\|_{N_{k}}\label{B8n}\\
&+\sum\limits_{{k_1},{k_2} \ge k - 100}^{|{k_1} - {k_2}| \le  120} \sum_{|l-k_1|\le 4} \| {P_k}\left[ P_{l}{\phi_{j_1}}  P_{\le k_1-8}\phi_{j_0} {P_{{k_2}}}(\mathcal{\widetilde{G}} \phi_{j_3})\right]\|_{N_{k}}\label{B9n}\\
&+\sum\limits_{{k_1},{k_2} \ge k - 100}^{|{k_1} - {k_2}| \le  120} \sum_{l_1,l_2\ge k_1-8, |l_1-l_2|\le 16} \| {P_k}\left[ P_{l_1}{\phi_{j_1}}P_{l_2}\phi_{j_0} {P_{{k_2}}}(\mathcal{\widetilde{G}} \phi_{j_3})\right]\|_{N_{k}}.\label{B10n}
\end{align}
By (\ref{6c}) and  (\ref{1Cpo}), (\ref{2CXnb}), $\|P_kf\|_{L^{4}_{t,x}}\le \|P_kf\|_{F_{k}}$, the first two terms are bounded as
 \begin{align*}
(\ref{B8n}) + (\ref{B9n}) &\lesssim \sum\limits_{ k_1 \ge k - 100} \sum\limits_{ |k_1-k_2| \le 120} \sum_{|l-k_1|\le 4}2^{\frac{k-l}{6}}\|P_{{l}} \phi_x\|_{G_{l}(T)}\|\phi_x\|_{L^4_{t,x}}
\|P_{k_2}(\mathcal{\widetilde{G}} \phi_{j_3})\|_{F_{k_2}(T)}\\
&\lesssim \sum\limits_{ k_1 \ge k - 100} \varepsilon_0 b_{k_1}\left[2^{\frac{k-k_1}{6}}2^{-\sigma k_1}b_{k_1}(\sigma)\right]\\
&\lesssim \epsilon_0 2^{-\sigma k} b_{k}(\sigma).
\end{align*}
And using (\ref{6c}) and (\ref{6a}), we see
 \begin{align*}
 (\ref{B10n})&\lesssim \sum\limits_{{k_1},{k_2} \ge k - 100}^{|{k_1} - {k_2}| \le  120} \sum_{l_1,l_2\ge k_1-8, |l_1-l_2|\le 16} 2^{\frac{k-l_1}{6}}\|P_{l_1 }{\phi_{j_1}}\|_{G_{l_1}(T)} \|(P_{l_2}\phi_{j_0} )P_{k_2} (\mathcal{\widetilde{G}} \phi_{j_3})\|_{L^2_{t,x}}\\
&\lesssim  \sum\limits_{{k_1},{k_2} \ge k - 100}^{|{k_1} - {k_2}| \le  120} \sum_{l_1,l_2\ge k_1-8, |l_1-l_2|\le 16} 2^{\frac{k-l_1}{6}}\|P_{l_1}{\phi_{j_1}}\|_{G_{l_1}(T)} \|P_{l_2}\phi_{j_0} \|_{L^4}\|P_{k_2}(\mathcal{{\widetilde{G}}} \phi_{j_3})\|_{L^4}\\
&\lesssim \epsilon_0\sum\limits_{ k_1 \ge k - 100}   \sum_{l_1\ge k_1-4, |l_1-l_2|\le 16 }2^{\frac{k -l_1}{6}}   b_{l_1}2^{-\sigma l_2}b_{l_2}(\sigma)\\
&\lesssim 2^{-\sigma k}b_kb_k(\sigma),
\end{align*}
where we used the embedding $L^4_{k}(T)\hookrightarrow F_{k}(T)\hookrightarrow G_{k}(T)$ and the fact $ \|P_{k_2}(\mathcal{{\widetilde{G}}} \phi_{j_3})\|_{L^4}\lesssim \|\phi_x\|_{L^4}$ in the third inequality.
Thus the third RHS term of  (\ref{Cpo}) has been done. Hence, we have finished the proof .

\end{proof}

\begin{Corollary} (Proof of Proposition \ref{4.5})
 Under the assumptions of Proposition \ref{4.5}, for all $i\in\{1,2\}$ and $\sigma\in[0,\frac{99}{100}]$ we have
\begin{align}
\|P_k \phi_i  \|_{G_k(T)}\lesssim  2^{-\sigma k}c_k(\sigma).
\end{align}
\end{Corollary}
\begin{proof}
(\ref{bbcc}) shows for any $k\in\Bbb Z$, $\sigma\in[0,\frac{99}{100}]$,
\begin{align}
2^{\sigma k} \|P_{k}\phi_i(0,0,\cdot)\|_{L^2_x}\lesssim c_{k}(\sigma).
\end{align}
Then by Proposition \ref{CXz12}, Proposition \ref{V1} and the linear estimates of Proposition \ref{bvc457}, one has
\begin{align}
b_{k}(\sigma)\lesssim c_{k}(\sigma)+\epsilon_0 b_{k}(\sigma),
\end{align}
for all $\sigma\in[0,\frac{99}{100}]$.
Thus $b_{k}(\sigma)\lesssim c_{k}(\sigma)$, and our result follows by the definition of $\{b_{k}(\sigma)\}$ in Section 5.
\end{proof}

\subsection{ Unform bounds for $\sigma \in[0,\frac{99}{100}]$  }

We end the arguments for $\sigma\in[0,\frac{99}{100}]$ with the following proposition.
\begin{Proposition}\label{NbM}
Assume that $\sigma\in [0,\frac{99}{100}]$.
Let $Q\in \mathcal{N}$ be a fixed point and $\epsilon_0$ be a sufficiently small constant.  Given any $\mathcal{L}\in \Bbb Z_+$, assume that $T\in(0,2^{2\mathcal{L}}]$. Let $\{c_k\}$ be an $\epsilon_0$-frequency envelope of order $\delta$. And let $\{c_k(\sigma)\}$ be another frequency envelope of order $\delta$. Let $u\in \mathcal{H}_Q(T)$ be the solution to $SMF$ with initial data $u_0$ which satisfies
\begin{align}
\|P_{k}\nabla u_0\|_{L^2_x}&\le c_k\label{1kxxxjn}\\
\|P_{k}\nabla u_0\|_{L^2_x}&\le c_k(\sigma)2^{-\sigma k}.\label{2kxxxjn}
\end{align}
Denote $\{\phi_i\}$ the corresponding differential fields of the heat flow initiated from $u$.
Then we have for all $i=1,2$, $k\in\Bbb Z$, $\sigma\in[0,\frac{99}{100}]$
\begin{align}
\|P_{k}\phi_i\upharpoonright_{s=0}\|_{G_k(T)}&\lesssim c_k\label{HUM1}\\
\|P_{k}\phi_i\upharpoonright_{s=0}\|_{G_k(T)}&\lesssim c_k(\sigma)2^{-\sigma k}\label{HUM2}\\
\sup_{s\ge 0}(1+s2^{2k})^{4}\|P_{k}\phi_i(s)\|_{F_k(T)}&\lesssim c_k(\sigma)2^{-\sigma k}.\label{HUM}
\end{align}
\end{Proposition}
\begin{proof}
Define the function $\Theta:[-T.T]\to \Bbb R^+$ by
\begin{align*}
\Theta(T'):=\sup_{k\in\Bbb Z}c^{-1}_{k}\left(\|P_{k}\phi_i\upharpoonright_{s=0}\|_{G_k(T')}
+\|P_k\nabla u\|_{L^{\infty}L^2_x(T')}\right).
\end{align*}
By Lemma \ref{FFF}, the function $\Theta(T')$ is continuous in $T'\in[0,T]$.
Then Proposition \ref{4.5} implies
\begin{align*}
\Theta(T')\le \epsilon^{-\frac{1}{2}}_0\Longrightarrow \sup_{k\in\Bbb Z}c^{-1}_{k}\left(\|P_{k}\phi_i\upharpoonright_{s=0}\|_{G_k(T')}
\right)\lesssim 1.
\end{align*}
And by Proposition \ref{7.2},
\begin{align*}
\sup_{k\in\Bbb Z}c^{-1}_{k}\left(\|P_{k}\phi_i\upharpoonright_{s=0}\|_{G_k(T')}
\right)\lesssim 1 \Longrightarrow \sup_{k\in\Bbb Z}c^{-1}_{k}\left(\|P_{k}\nabla u\|_{L^{\infty}_tL^2_x(T')}
\right)\lesssim 1.
\end{align*}
Hence, we conclude
\begin{align*}
\Theta(T')\le \epsilon^{-\frac{1}{2}}_0\Longrightarrow \Theta(T')\lesssim 1.
\end{align*}
And it is easy to see $\Theta(T')$ is continuous and increasing.  Moreover, we have the limit
\begin{align*}
\lim_{T'\to 0}\Theta(T')\lesssim 1,
\end{align*}
by the definition of $\Theta(T')$, $G_{k}(T')$ and Corollary 3.1.
Therefore, from the continuity of $\Theta$ we conclude that (\ref{1kxxxjn}), (\ref{2kxxxjn}) suffice  to yield
\begin{align*}
\Theta(T)\lesssim 1,
\end{align*}
thus giving (\ref{HUM1}). Then Proposition \ref{4.5} yields (\ref{HUM2}), and (\ref{HUM}) follows by the inclusion $G_{k}\subset F_k$ and Proposition 4.1.
\end{proof}

\section{Iteration scheme }

 {\bf From now on, the notations $a^{(j)}_{k}(\sigma)$ and $a^{(j)}_{k,s}(\sigma)$ differ from the ones defined in Section 3.} They are defined as follows.
\begin{Definition}\label{1jknmL}
Assume that $u_0\in \mathcal{H}_{Q}$. Given $j\in\Bbb N$, let
\begin{align*}
c_{k,(j)}(\sigma)=\sup_{k'\in\Bbb Z}2^{-\frac{1}{2^{j}}\delta|k-k'|}\|P_{k'}\nabla u_0\|_{L^2_x}, \mbox{ }\forall \sigma\in {\Bbb I}_{j}, \mbox{ }k\in\Bbb Z,
\end{align*}
where ${\Bbb I}_0=[0,\frac{99}{100}]$ and ${\Bbb I}_j=[0,\frac{j}{4}+1]$ for $j\in \Bbb Z_+$.
\begin{itemize}
\item For $ \sigma\in [0,\frac{99}{100}]$, define
\begin{align*}
c^{(0)}_k(\sigma)=c_{k,(0)}(\sigma),\mbox{ }\forall  \sigma\in [0,\frac{99}{100}].
\end{align*}
  \item For $\sigma\in [0,\frac{5}{4}]$, define
\begin{align*}
c^{(1)}_k(\sigma)=\left\{
                  \begin{array}{ll}
                    c_{k,(1)}(\sigma), & \hbox{ }  \sigma\in [0,\frac{99}{100}]  \\
                     c_{k,(1)}(\sigma)+ c_{k,(1)}(\frac{3}{8}) c_{k,(1)}(\sigma-\frac{3}{8}), & \hbox{ }  \sigma\in (\frac{99}{100},\frac{5}{4}].
                  \end{array}
                \right.
\end{align*}
  \item
Given an integer $j\ge 2$,  for $\sigma\in [0,\frac{j}{4}+1]$, define $\{c^{(j)}_k(\sigma)\}$ by induction:
\begin{align*}
c^{(j)}_k(\sigma)=\left\{
                  \begin{array}{ll}
                   c_{k,(j)}(\sigma), & \hbox{ }  \sigma\in [0,\frac{99}{100}]  \\
                   c_{k,(j)}(\sigma)+ c_{k,(j)}(\frac{3}{8}) c_{k,(j)}(\sigma-\frac{3}{8}), & \hbox{ }  \sigma\in [0,\frac{5}{4}]\\
... &\\
 c_{k,(j)}(\sigma)+ c_{k,(j)}(\frac{3}{8}) {c}^{(j)}_{k }(\sigma-\frac{3}{8}), & \hbox{ }  \sigma\in [\frac{m+3}{4},\frac{m}{4}+1] \\
... &\\
 c_{k,(j)}(\sigma)+ c_{k,(j)}(\frac{3}{8}) {c}^{(j)}_{k}(\sigma-\frac{3}{8}), & \hbox{ }  \sigma\in [\frac{j+3}{4},\frac{j}{4}+1].
                  \end{array}
                \right.
\end{align*}
\end{itemize}

\end{Definition}

\begin{Definition}\label{2jknmL}
\begin{itemize}
  \item Assume that  $\{a_{k}(\sigma)\}$ are frequency envelopes of order $\delta$ with $\sigma\in [0,\frac{99}{100}]$. Define
\begin{align*}
a^{(0)}_k(\sigma)=c^{(0)}_k(\sigma),\mbox{ }\forall  \sigma\in [0,\frac{99}{100}].
\end{align*}
  \item Assume that  $\{a_{k}(\sigma)\}$ are frequency envelopes of order $\delta$ with $\sigma\in [0,\frac{99}{100}]$. Define
\begin{align*}
a^{(1)}_k(\sigma)=\left\{
                  \begin{array}{ll}
                    c^{(1)}_k(\sigma), & \hbox{ }  \sigma\in [0,\frac{99}{100}]  \\
                     a_{k}(\sigma)+ c^{(1)}_{k}(\frac{3}{8}) c^{(1)}_{k}(\sigma-\frac{3}{8}), & \hbox{ }  \sigma\in (\frac{99}{100},\frac{5}{4}].
                  \end{array}
                \right.
\end{align*}
  \item
Given an integer $j\ge 2$,  assume that $\{a_{k}(\sigma)\}$ are frequency envelopes of order $\delta$ with $\sigma\in [0,\frac{j}{4}+1]$. Define
\begin{align*}
a^{(j)}_k(\sigma)=\left\{
                  \begin{array}{ll}
                    c^{(j)}_k(\sigma), & \hbox{ }  \sigma\in [0,\frac{j+3}{4}]  \\
                     a_{k}(\sigma)+ c^{(j)}_k(\frac{3}{8}) c^{(j)}_k(\sigma-\frac{3}{8}), & \hbox{ }  \sigma\in (\frac{j+3}{4},\frac{j}{4}+1].
                  \end{array}
                \right.
\end{align*}
\end{itemize}
\end{Definition}

Given an integer $j\in \Bbb N$,  assume that $\{a_{k}(\sigma)\}$ are frequency envelopes of order $\delta$ with $\sigma\in [0,\frac{j}{4}+1]$, we also define
\begin{align*}
{a^{(j)}_{k,s}}\left( \sigma  \right) = \left\{ \begin{gathered}
  {2^{k + {k_0}}}{a_{ - {k_0}}}(0)a^{(j)}_k(\sigma )\mbox{  }{\rm if}\mbox{ }k + {k_0} \geqslant 0 \hfill \\
  \sum\limits_{l = k}^{ - {k_0}}  {a_l}(0){a^{(j)} _l}(\sigma )\mbox{  }{\rm if}\mbox{ }k + {k_0} \leqslant 0, \hfill \\
\end{gathered}  \right.
\end{align*}
for $s\in[2^{2k_0-1},2^{2k_0+1})$, $k,k_0\in\Bbb Z$.

\begin{Remark}
Given $j\ge  2$, we infer from Def. \ref{1jknmL} that  $\{c^{(j)}(\sigma)\}$ is of order $\frac{1}{2^{m}}\delta$ if $\sigma\in [\frac{m+3}{4},\frac{m}{4}+1]$, $2\le m\le j$.
Particularly, $\{c^{(j)}(\sigma)\}$ is of order $ \delta$ for all $\sigma\in[0,\frac{j}{4}+1]$.  One can also  see from Def. \ref{2jknmL} that  $\{b^{(j)}(\sigma)\}$ are of order $\delta$ for all $\sigma\in[0,\frac{j}{4}+1]$.
\end{Remark}

Now we iterate the argument of previous sections to obtain uniform bounds for all  $\sigma\in[0,\frac{5}{4}]$. We aim to prove the following proposition:
\begin{Proposition}
Assume that  $\sigma\in[0,\frac{5}{4}]$.
Let $Q\in \mathcal{N}$ be a fixed point and $\epsilon_0 $ be a sufficiently small constant.  Given any $\mathcal{L}\in \Bbb Z_+$, assume that $T\in(0,2^{2\mathcal{L}}]$. Let $u\in \mathcal{H}_Q(T)$ be the solution to SMF  with initial data $u_0$. Let $\{c^{(1)}_{k}(\sigma)\}$ be frequency envelopes  defined by Definition \ref{1jknmL}, and assume that $\{c^{(1)}_k(0)\}$ is  an $\epsilon_0$-frequency envelope.
Denote $\{\phi_i\}$ the corresponding differential fields of the heat flow initiated from $u$.
Then for all $i=1,2$, $k\in\Bbb Z$, $\sigma\in [0,\frac{5}{4}]$, we have
\begin{align*}
2^{\sigma k}\|P_{k}\phi_i\upharpoonright_{s=0}\|_{G_{k}(T)}&\lesssim c^{(1)}_k(\sigma).
\end{align*}
\end{Proposition}

As before, this proposition will be divided into two propositions for proof, one is for heat flow evolution and the other is for the Schr\"odinger map flow evolution. In the statement of following propositions or lemmas, the notation $\checkmark$ means the line it lies in can be dropped.
\begin{Proposition}\label{Parabolic2}
Let $\sigma\in[0,\frac{5}{4}]$. Let $\{b_{k}(\sigma)\}$ be frequency envelopes of order $\delta$ such that  $b_{k}(\sigma)\lesssim c^{(1)}_{k}(\sigma)$ for $\sigma\in[0,\frac{99}{100}]$.   Assume that $\{c^{(1)}_k(0)\}$ is  an $\epsilon_0$-frequency envelope.
\begin{itemize}
               \item Assume that for $i=1,2$,
\begin{align}
\|P_{k }\phi_i\upharpoonright_{s=0}\|_{F_k(T)}&\le b_k(\sigma')2^{-\sigma' k} \mbox{  }\sigma'\in[0,\frac{5}{4}], \label{sad2}\\
\checkmark\|P_{k }\phi_i(s)\|_{F_k(T)}&\le \varepsilon^{-1}b^{(1)}_k(0)(1+s2^{2k})^{-4}.\label{sad}
\end{align}
Then for $\sigma \in [0,\frac{5}{4}]$, $i=1,2$,
\begin{align}
\|P_{k}\phi_i(s)\|_{F_k(T)}&\lesssim  2^{-\sigma  k}  (1+s2^{2k})^{-4} b^{(1)}_k(\sigma )\label{4.181}\\
\|P_k A_i\upharpoonright_{s=0}\|_{L^{4}_{t,x}}&\lesssim  b^{(1)}_k(\sigma )2^{-\sigma k}.\label{4.191}
\end{align}
               \item Assume further that
\begin{align}\label{kuku}
\|P_{k }\phi_t\upharpoonright_{s=0}\|_{L^{4}_{t,x}}\lesssim b_k(\sigma')2^{-(\sigma' -1)k} \mbox{  }\sigma'\in[0,\frac{5}{4}].
\end{align}
Then for $\sigma\in [0,\frac{5}{4}]$, one has
\begin{align}
\|A_t\upharpoonright_{s=0}\|_{L^{2}_{t,x}}&\lesssim \varepsilon^2\label {IU4}\\
\|P_k\phi_t(s)\|_{L^{4}_{t,x}}&\lesssim   b^{(1)}_k(\sigma)2^{-(\sigma-1)k}(1+2^{2k}s)^{-2}\label{IU3}\\
\|P_kA_t\upharpoonright_{s=0}\|_{L^{2}_{t,x}}&\lesssim \varepsilon b^{(1)}_{k}(\sigma)2^{- \sigma k}.\label{IU5}
\end{align}
\end{itemize}
\end{Proposition}
\begin{proof}
Recalling  definitions of $c^{(1)}_k(\sigma)$, $b^{(1)}_k(\sigma)$ in Def. \ref{1jknmL} and Def. \ref{2jknmL},
  by Proposition \ref{NbM}, we see (\ref{4.181}), (\ref{4.191}), (\ref{IU3}), (\ref{IU5}) are already done for $\sigma\in[0,\frac{99}{100}]$. Moreover, (\ref{IU4}) and the assumption (\ref{sad}) hold  naturally. The left is to prove (\ref{4.181}), (\ref{4.191}), (\ref{IU3}), (\ref{IU5}) for $\sigma\in[\frac{99}{100},\frac{5}{4}]$.

The key and starting point for SMF iteration scheme is to improve $\|P_{k}\widetilde{\mathcal{G}}^{(1)}\|_{L^{4}_xL^{\infty}_t}$ step by step.
\begin{Lemma}\label{KEY}
Let $u\in\mathcal{H}_Q(T)$ solve SMF with data $u_0$. Given any $\sigma\in[0,\frac{99}{100}]$, let $\{c^{(1)}_{k}(\sigma)\}$  be frequency envelopes defined in Def. \ref{1jknmL}. Assume also that $\{c^{(1)}_{k}(0)\}$ is an $\epsilon_0$-frequency envelope. Then for $\epsilon_0$ sufficiently small there holds
\begin{align*}
2^{\frac{1}{2}k}\|P_{k} \widetilde{\mathcal{G}}^{(1)}\|_{L^4_xL^{\infty}_t}\le c^{(1)}_k(\sigma)2^{-\sigma k}[(1+2^{2k+2k_0})^{-20}1_{k+k_0\ge 0}+1_{k+k_0\le 0}2^{\delta|k+k_0|}],
\end{align*}
for any $\sigma\in[0,\frac{99}{100}]$, $k,k_0\in\Bbb Z$, $s\in[2^{2k_0-1},2^{2k_0+1})$.
\end{Lemma}
\begin{proof}
We obtain  by combining Proposition 5.1 and Proposition 4.1 that
\begin{align*}
 \left\| P_{k}\phi_t\right\|_{L^4}\lesssim (1+s2^{2k})^{-2} 2^{-\sigma k+k}c^{(1)}_{k}(\sigma),\mbox{ }\sigma\in[0,\frac{99}{100}]\\
 \left\| P_{k}\phi_i\right\|_{L^4}\lesssim (1+s2^{2k})^{-4} 2^{-\sigma k}c^{(1)}_{k}(\sigma),\mbox{ }\sigma\in[0,\frac{99}{100}].
\end{align*}
Then Proposition 3.6 yields
 \begin{align*}
 \left\| P_{k}\widetilde{\mathcal{G}} \right\|_{L^4_{t.x}\cap L^{\infty}_tL^2_x}&\lesssim (1+s2^{2k})^{-30} 2^{-\sigma k-k}c^{(1)}_{k}(\sigma),\mbox{ } \sigma\in[0,\frac{99}{100}]\\
 \left\| P_{k}\widetilde{\mathcal{G}}^{(m)} \right\|_{L^4_{t.x}\cap L^{\infty}_tL^2_x}&\lesssim (1+s2^{2k})^{-30} 2^{-\sigma k-k}c^{(1)}_{k}(\sigma),\mbox{ } m=1,2,\mbox{ } \sigma\in[0,\frac{99}{100}].
\end{align*}
Then using the schematic formula
 \begin{align*}
 A_t=\int^{\infty}_{s} \phi_t\diamond (D_i\phi_i) \mathcal{G} ds'
\end{align*}
and bilinear Littlewood-Paley decomposition (see Lemma 4.7), we get
 \begin{align*}
\|P_{k} A_t\|_{L^4}\le c^{(1)}_{k}(\sigma)2^{-\sigma k+k}[(1+2^{2k+2j})^{-1}1_{k+j\ge 0}+1_{k+j\le 0}c^{(1)}_{k}2^{\delta|k+j|}]
\end{align*}
for any $\sigma\in[0,\frac{99}{100}]$,  $k,j\in\Bbb Z$, $s\in[2^{2j-1},2^{2j+1})$.  And then  using $\partial_t\widetilde{\mathcal{G}}^{(1)}=A_t\mathcal{G}^{(1)}+\mathcal{G}^{(2)}\phi_t$ and  interpolation (see Lemma 4.7), one deduces that
 \begin{align*}
2^{\frac{1}{2}k}\|P_{k} \widetilde{\mathcal{G}}^{(1)}\|_{L^4_xL^{\infty}_t}\le c^{(1)}_k(\sigma)2^{-\sigma k}[(1+2^{2k+2k_0})^{-20}1_{k+k_0\ge 0}+1_{k+k_0\le 0}2^{\delta|k+k_0|}]
\end{align*}
for any $\sigma\in[0,\frac{99}{100}]$, $k,k_0\in\Bbb Z$, $s\in[2^{2k_0-1},2^{2k_0+1})$.
\end{proof}

As before, we start with the bound for connection forms.
\begin{Lemma}\label{connection2}
Let $\sigma\in[\frac{99}{100},\frac{5}{4}]$. Denote
\begin{align}
h(k)=\sup_{s\ge 0}(1+s2^{2k})^4\sum^{2}_{i=1}\|P_k\phi_i(s)\|_{F_k(T)}.
\end{align}
Define the corresponding envelope by
\begin{align}
h_k(\sigma)&=\sup_{k'\in\Bbb Z}2^{\sigma k'}2^{-\delta|k'-k|}h(k').
\end{align}
Then under the assumptions of Proposition \ref{Parabolic2}, for all $k\in \Bbb Z$, $s\ge 0$, $i=1,2$, we have
\begin{align}
\|P_k(A_i(s))\|_{F_k(T)\bigcap S^{\frac{1}{2}}_k(T)}\lesssim 2^{-\sigma k}(1+s2^{2k})^{-4}h^{(1)}_{k,s}(\sigma),
\end{align}
where the sequence $\{h^{(1)}_{k,s}\}$ when $ 2^{2k_0-1}\le s< 2^{2k_0+1}, k_0\in\Bbb Z$, is defined by
\begin{align}\label{GH}
{h^{(1)}_{k,s}}\left( \sigma  \right) = \left\{ \begin{gathered}
  {2^{k + {k_0}}}{h_{ - {k_0}}}{h_k}^{(1)}(\sigma )\mbox{  }{\rm if}\mbox{ }k + {k_0} \geqslant 0 \hfill \\
  \sum\limits_{l = k}^{ - {k_0}}  {h_l}{h^{(1)} _l}(\sigma )\mbox{  }{\rm if}\mbox{ }k + {k_0} \leqslant 0 \hfill \\
\end{gathered}  \right.
\end{align}
with
\begin{align}\label{Notation2}
{h_k}^{(1)}(\sigma' )=\left\{
                       \begin{array}{ll}
                         c^{(1)}_{k}(\sigma'), & \hbox{ } \sigma'\in[0,\frac{99}{100}]\\
                         h_{k}(\sigma')+c^{(1)}_{k}(\frac{3}{8})c^{(1)}_{k}(\sigma'-\frac{3}{8}), & \hbox{ }\sigma'\in(\frac{99}{100},\frac{5}{4}].\\
                       \end{array}
                     \right.
\end{align}
\end{Lemma}
\begin{proof}
The proof is almost the same as Lemma \ref{connection}. The difference is that more concerns are  needed for the ${\rm High\times Low}$ interaction of $P_{k}[\widetilde{\mathcal{G}}^{(1)}\psi_s]$ in Step 4 of Lemma \ref{connection}.
First of all we point out (\ref{HUM}) of Proposition \ref{NbM} shows for all $\sigma'\in[0,\frac{99}{100}]$
 \begin{align}
h_k(\sigma') \lesssim c^{(1)}_{k}(\sigma').\label{hbL}
\end{align}
Let $B^{(1)}_1$ be the smallest constant such that for all $\sigma\in[\frac{99}{100},\frac{5}{4}]$, $s\ge 0$, $k\in\Bbb Z$, there holds
 \begin{align}
\|P_k(A_i(s))\|_{F_k(T)\bigcap S^{\frac{1}{2}}_k(T)}\lesssim B^{(1)}_12^{-\sigma k}(1+s2^{2k})^{-4}h^{(1)}_{k,s}(\sigma).
\end{align}
Recall the following decomposition of $\mathcal{G}$:
 \begin{align*}
\mathcal{G}=\Gamma^{\infty}-\Gamma^{\infty,(1)}_p\int^{\infty}_s\psi^p_sds'- \int^{\infty}_s\psi^p_s\widetilde{\mathcal{G}}^{(1)}ds'.
\end{align*}
By $\psi_s=\sum^{2}_{i=1}\partial_i\psi_i+A_i\psi_i$, we separate the $\psi_i$ part away.  And thus schematically one has
 \begin{align*}
\mathcal{G}&=\Gamma^{\infty}- \Gamma^{\infty,(1)}_{l} \int^{\infty}_s(\partial_i \psi_i)^{l}ds'- \int^{\infty}_s (\partial_i \psi_i)^l \widetilde{\mathcal{G}}^{(1)}_lds'\\
&-\Gamma^{\infty,(1)}_{l} \int^{\infty}_s(A_i\psi_i)^{l}ds'-\int^{\infty}_s(A_i\psi_i)^{l} \widetilde{\mathcal{G}}^{(1)}_{l}ds'
\end{align*}
In order to prove our lemma,
as before we first prove $B^{(1)}_1\lesssim 1$ under the {\bf Bootstrap Assumption B}: For the fixed given $\sigma\in (\frac{99}{100},\frac{5}{4}]$, there hold
 \begin{align*}
\int^{\infty}_s\|P_k(A_i\psi_i) \|_{F_k(T)}ds'\lesssim \varepsilon^{-\frac{1}{2}} 2^{-\sigma k}T_{k,j}(1+s^{\frac{1}{2}}2^{k})^{-7}h^{(1)}_{k}(\sigma)c^*_0\\
\int^{\infty}_s\|P_k[(A_i\psi_i)\widetilde{\mathcal{G}}^{(1)}] \|_{F_k(T)}ds'\lesssim \varepsilon^{-\frac{1}{2}} 2^{-\sigma k}T_{k,j}(1+s^{\frac{1}{2}}2^{k})^{-7}h^{(1)}_{k}(\sigma)c^*_0.
\end{align*}
where $c^*_0:=\|\{h_k\}\|_{\ell^2}$, $s\in[2^{2j-1},2^{2j+1})$, and $T_{k,j}$ is defined in (\ref{FFgB}).
This part is the same as Step 2 of Lemma \ref{connection} except controlling
\begin{align}\label{pcojn}
\left\|P_{k}\left(\int^{\infty}_s (\partial_i \psi_i)(\widetilde{\mathcal{G}}^{(1)})ds'\right)\right\|_{F_{k}(T)},
\end{align}
which was labeled as $\mathcal{U}_{01}$ in Lemma \ref{connection}.
To estimate (\ref{pcojn}), recall the bounds  in  Lemma \ref{KEY} and  Proposition 3.6 for $\widetilde{\mathcal{G}}^{(1)}$:
  \begin{align}
&2^{k} \|P_{k}(\widetilde{\mathcal{G}}^{(1)})  \|_{L^{\infty}_tL^2_x\cap L^{4}}+2^{\frac{1}{2}k} \|P_{k}(\widetilde{\mathcal{G}}^{(1)}) \|_{L^4_xL^{\infty}_t}\nonumber\\
&\lesssim 2^{-\widetilde{\sigma }k}c_{k}(\widetilde{\sigma} )[1_{k+j\ge 0}(1+2^{2k}s)^{-20}+1_{k+j\le 0}2^{\delta|k+j|}]\label{1pcojn}
\end{align}
for any $k,j\in\Bbb Z$, $s\in[2^{2j-1}, 2^{2j+1})$, $\widetilde{\sigma} \in[0,\frac{99}{100}]$.
By bilinear Littlewood-Paley decomposition and Lemma \ref{mJ},  we have
 \begin{align*}
&\|P_k((\partial_i\psi_i)\widetilde{\mathcal{G}}^{(1)}) \|_{F_k(T)}\\
&\lesssim \sum_{|k_1-k|\le 4}\|P_{k_1}(\partial_i\psi_i) \|_{F_{k_1}(T)}\|P_{\le k-4}\widetilde{\mathcal{G}}^{(1)}\|_{L^{\infty}}\\
&+ \sum_{|k_1-k_2|\le 8,k_1,k_2\ge k-4} \|P_{k_1}(\partial_i\psi_i) \|_{F_{k_1}{(T)}}(\|P_{k_2}(\widetilde{\mathcal{G}}^{(1)})\|_{L^{\infty}}+2^{\frac{k_1}{2}}\|P_{k_2}(\widetilde{\mathcal{G}}^{(1)})\|_{L^{4}_xL^{\infty}_t})
\\
&+\sum_{|k_2-k|\le 4,k_1\le k-4}  2^{\frac{k_1}{2}} \|P_{k_1}(\partial_i\psi_i) \|_{F_{k_1}(T)} \|P_{k_2}(\widetilde{\mathcal{G}}^{(1)})\|_{L^4_xL^{\infty}_t}+ 2^{{k_1}} \|P_{k_1}(\partial_i\psi_i) \|_{F_{k_1}(T)} \|P_{k_2}(\widetilde{\mathcal{G}}^{(1)})\|_{L^4}\\
&\lesssim  2^{-\sigma k}h^{(1)}_{k}(\sigma)+2^{-\sigma  k }h_{k}h^{(1)}_{k}({\sigma} )\left[1_{k+j\ge 0}2^{k}(1+2^{2k}s)^{-4}+1_{k+j\le 0}2^{\delta|k+j|}2^{-j}\right]\\
&+R_{j,k}2^{-(\sigma -\frac{3}{8}) k}c^{(1)}_{k}(\sigma -\frac{3}{8}) [2^{-\frac{1}{2}k}\sum_{k_1\le k-4}2^{\frac{3}{2}k_1-\frac{3}{8}  k_1}c^{(1)}_{k_1}(\frac{3}{8})
+2^{-k}\sum_{k_1\le k-4}2^{2k_1-\frac{3}{8}  k_1}c^{(1)}_{k_1}(\frac{3}{8})  ]
\end{align*}
where we denote $R_{j,k}:= 1_{k+j\ge 0}(1+2^{2k}s)^{-20}+1_{k+j\le 0}2^{\delta|k+j|}h_{k} $ and have used  (\ref{hbL}).
Thus by slow variation of envelopes we get
 \begin{align*}
&\|P_k((\partial_i\psi_i)\widetilde{\mathcal{G}}^{(1)}) \|_{F_k(T)}\lesssim  2^{-\sigma k}h^{(1)}_{k}(\sigma)\left( 1_{k+j\ge 0}2^{k}(1 +2^{2k_2}s)^{-4}
+1_{k+j\le 0}2^{-j}2^{\delta|k+j|}\right).
\end{align*}
for $s\in[2^{2j-1},2^{2j+1})$, $j,k\in\Bbb Z$.
This bound is the same as $\mathcal{U}_{01}$ in Lemma \ref{connection} and acceptable.

In the third step,  we prove the claim: If {\bf Bootstrap Assumption B} holds, then
 \begin{align}
\int^{\infty}_s\|P_k(A_i\psi_i) \|_{F_k(T)}ds'\lesssim 2^{-\sigma k}T_{k,j}(1+s^{\frac{1}{2}}2^{k})^{-7}h^{(1)}_{k}(\sigma)c^*_0\label{nbx1}\\
\int^{\infty}_s\|P_k(A_i\psi_i)\widetilde{\mathcal{G}}^{(1)} \|_{F_k(T)}ds'\lesssim  2^{-\sigma k}T_{k,j}(1+s^{\frac{1}{2}}2^{k})^{-7}h^{(1)}_{k}(\sigma)c^*_0.\label{nbx2}
\end{align}
The proof of (\ref{nbx1}) is the same as Step 4 of  Lemma \ref{connection}. For (\ref{nbx2}), the
$\rm Low\times High$ interaction of $P_{k}[(A_i\psi_i)\widetilde{\mathcal{G}}^{(1)}]$ is different due to the larger $\sigma$. The other two interactions are all the same. We present the modifications as follows.
Since under Bootstrap Assumption B  one has $B^{(1)}\lesssim 1$, $P_k(A_i\psi_i)$ enjoys the same $F_{k}\bigcap S^{\frac{1}{2}}_k$ bound as Lemma \ref{connection} with $h_{k}(\sigma)$ replaced by  $h^{(1)}_k(\sigma)$:
 \begin{align*}
\| P_{k }(A_i\psi_i)\|_{F_k(T)\bigcap S^{\frac{1}{2}}_k(T)}&\lesssim  c^*_02^{-\sigma k} 1_{k+j\le 0}h^{(1)}_{k} (\sigma)2^{\frac{1}{2}(k-j)}2^{\delta|k+j|}\\
&+  c^*_0 2^{-\sigma k}1_{k+j\ge 0} h^{(1)}_{k}(\sigma) 2^{k}(1+2^{j+k})^{-8},
\end{align*}
for all $\sigma\in[0,\frac{5}{4}]$, $s\in[2^{2j-1},2^{2j+1})$, $j,k\in\Bbb Z$.

Then by (\ref{1pcojn}), (\ref{hbL}) and   (\ref{2Jgvm}), the $\rm Low\times High$ part of $(A_i\psi_i)\widetilde{\mathcal{G}}^{(1)}$ is dominated by
 \begin{align*}
&\sum_{|k-k_2|\le 4,k_1\le k+4}\|P_k (P_{k_1 }(A_i\psi_i)P_{k_2}\widetilde{\mathcal{G}}^{(1)})\|_{F_k(T)}\\
&\lesssim  c^{*}_0 2^{-({\sigma}-\frac{3}{8}) k}c^{(1)}_{k}(\sigma-\frac{3}{8})1_{k+j\le 0}\sum_{k_1\le k-4} c^{(1)}_{k_1}(\frac{3}{8}) 2^{\frac{1}{2}(k_1-j)}2^{\delta|k_1+j|} 2^{-\frac{3}{8} k_1}\\
&+  c^{*}_0 2^{-(\sigma-\frac{3}{8}) k}(1+2^{2j+2k})^{-20}c^{(1)}_{k}( \sigma-\frac{3}{8})1_{k+j\ge 0}\left[\sum_{-j\le k_1\le k}c^{(1)}_{k_1}(\frac{3}{8}) 2^{k_1-\frac{3}{8} k_1 }(1+2^{2j+2k_1})^{-4}\right]\\
 &+ c^{*}_02^{-(\sigma-\frac{3}{8}) k}(1+2^{2j+2k})^{-20}c^{(1)}_{k}(\sigma-\frac{3}{8})1_{k+j\ge 0}\left[\sum_{k_1\le -j} c^{(1)}_{k_1}(\frac{3}{8}) 2^{\frac{k_1-j}{2}}2^{\delta|k_1+j|}2^{-\frac{3}{8} k_1}\right]\\
&\lesssim  c^{*}_0 2^{- \sigma k}c^{(1)}_{k}({\sigma}-\frac{3}{8})c^{(1)}_{k}(\frac{3}{8})\left(1_{k+j\ge 0} 2^{-j}(1+2^{j+k})^{-7}+1_{k+j\le 0} 2^{\frac{k-j}{2}}2^{\delta|k+j|}\right).
\end{align*}
Summing in $j\ge k_0$  as well yields
\begin{align*}
&\sum_{j\ge k_0}2^{2j}\sum_{|k-k_2|\le 4,k_1\le k+4}\|P_k (P_{k_1 }\psi_sP_{k_2}\widetilde{\mathcal{G}}^{(1)})\|_{F_k(T)}\\
&\lesssim c^*_0 2^{-\sigma k}h^{(1)}_{k}(\sigma)\left( 1_{k+k_0\ge 0}2^{k_0}(1+2^{k+k_0})^{-7}
+2^{-k}1_{k+k_0\le 0}\right)
\end{align*}
for $s\in[2^{2k_0-1}, 2^{2k_0+1})$, $k_0,k\in\Bbb Z$. This bound is again the same as  $\mathcal{U}_{II}$ in Lemma \ref{connection} and acceptable.

Finally, we need to prove (\ref{nbx1}), (\ref{nbx2}) of Bootstrap Assumption B hold when $T\to 0$.  Let's verify it.
Using (\ref{shujiu2}), (\ref{chui}) and putting $\frac{3}{8}$ order derivatives on $(A_i\psi_i)$ while estimating the $\rm Low\times High$ interaction of $(A_i\psi_i)\widetilde{\mathcal{G}}^{(1)}$, we also have
\begin{align*}
\int^{\infty}_s\|P_k[A_i\psi_i\widetilde{\mathcal{G}}^{(1)}]\|_{L^{\infty}_tL^2_x}ds'\lesssim \|\{h_k\}\|_{\ell^2}T_{k,j} h^{(1)}_{k}(\sigma)2^{-\sigma k}
\end{align*}
if $\sigma\in[\frac{99}{100},\frac{5}{4}]$.

Therefore, combining the above four steps gives Lemma \ref{connection2}.
\end{proof}

The proof of Lemma \ref{connection2} gives an $F_{k}$ bound for $\mathcal{\widetilde{G}}$.

\begin{Lemma}\label{xiaoyixiao}
For all $\sigma\in[\frac{99}{100},\frac{5}{4}]$, $k\in\Bbb Z$,
\begin{align}
{\left\| {{P_k}({\mathcal{\widetilde{G}}})} \right\|_{{F_k}(T) }} \lesssim \left\{ \begin{gathered}
  {2^{ - \sigma k}}{(1 + s{2^{2k}})^{ - 4}}{2^j}{h^{(1)}_k}(\sigma ),\mbox{ }{\rm if}\mbox{  }j + k \ge 0 \hfill \\
  {2^{ - \sigma k}}{2^{ - k}}{h^{(1)}_k}(\sigma ),\mbox{ }{\rm if}\mbox{  }j + k \le  0 \hfill \\
\end{gathered}  \right.
\end{align}
when $ 2^{2j-1}\le s\le 2^{2j+1}, j\in\Bbb Z$. Moreover, for $s=0$
\begin{align}
{\left\| {{P_k}{\mathcal{\widetilde{G}}}}\upharpoonright_{s=0} \right\|_{{F_k}(T) }} \lesssim 2^{-k-\sigma k}h^{(1)}_k(\sigma).
\end{align}
\end{Lemma}

{\bf Proof of  Proposition \ref{Parabolic2}.}
With this improved bound of $\mathcal{\widetilde{G}}$, running the program of Section 4 again gives
\begin{align*}
\sup_{s\ge 0}2^{\sigma k}(1+s2^{2k})^4\sum^{2}_{i=1}\|P_k\phi_i(s)\|_{F_k(T)}\lesssim b_{k}(\sigma)+
\varepsilon h^{(1)}(\sigma).
\end{align*}
Since the right side is a frequency envelope of order $\delta$, there holds
\begin{align*}
h_{k}(\sigma)\lesssim b_{k}(\sigma)+\varepsilon h^{(1)}_k(\sigma).
\end{align*}
By the definition of $h^{(1)}_k(\sigma)$, we conclude for $\sigma\in[\frac{99}{100},\frac{5}{4}]$
\begin{align*}
h_{k}(\sigma)\lesssim b_{k}(\sigma)+c^{(1)}_{k}(\frac{3}{8})c^{(1)}_{k}(\sigma-\frac{3}{8}),
\end{align*}
thus proving (\ref{4.181}). The left  (\ref{4.191}), (\ref{IU3}), (\ref{IU5}) are the same.
\end{proof}

In the following proposition, we finish iteration of $\sigma$ in the Schr\"odinger direction.
\begin{Proposition}\label{4.51}
Given $\mathcal{L}\in\Bbb Z_+$, suppose that $T\in(0,2^{2\mathcal{L}}]$ and $Q\in \mathcal{N}$.  Assume that $\sigma\in[0,\frac{5}{4}]$.  Let $u\in \mathcal{H}_Q(T)$ be a solution to SMF with initial data $u_0$, and set $\{c^{(1)}_k(\sigma)\}_{k\in\Bbb Z}$ to be frequency envelopes defined in Def. \ref{1jknmL}. And  assume that $\{c^{(1)}_k(0)\}$ is an $\epsilon_0$-frequency envelope  with $0<\epsilon_0\ll 1$.
Then
 for any $\sigma\in [0,\frac{5}{4}]$, $k\in\Bbb Z$, we have
\begin{align}\label{14.321}
 \|P_{k}\phi_i \upharpoonright_{s=0}\|_{G_k(T)}\lesssim  c^{(1)}_k(\sigma).
\end{align}
\end{Proposition}
\begin{proof}
(\ref{14.321}) has been proved for $\sigma\in[0,\frac{99}{100}]$ in Section 5.
Thus, it suffices to consider $\sigma\in(\frac{99}{100},\frac{5}{4}]$ only.
Let
\begin{align*}
b(k)&=\sum^{2}_{i=1}\|P_{k}\phi_i\upharpoonright_{s=0}\|_{G_k(T)}.
\end{align*}
For $\sigma\in[0,\frac{5}{4}]$, define the frequency envelopes:
\begin{align*}
b_k(\sigma)&=\sup_{k'\in\Bbb Z}2^{\sigma k'}2^{-\delta|k-k'|}b(k').
\end{align*}
By Proposition \ref{mei}, they are finite and $\ell^2$ summable. And
\begin{align*}
\|P_{k}\phi_i\upharpoonright_{s=0}\|_{G_k(T)}&\lesssim 2^{-\sigma k}b_k(\sigma).
\end{align*}
The assumption (\ref{kuku}) holds by repeating the same argument of Lemma 4.1.
Thus using Proposition \ref{Parabolic2}, we see (\ref{4.181})-(\ref{IU5}) hold. With Lemma \ref{xiaoyixiao}, repeating the argument in Section 5, one obtains when $s=0$,
\begin{align*}
\|P_{k}\phi_i\upharpoonright_{s=0}\|_{G_k(T)}&\lesssim c^{(1)}_k(\sigma)+\epsilon_0 \left(b_{k}(\sigma)+c^{(1)}_{k}(\frac{3}{8})c^{(1)}_{k}(\sigma-\frac{3}{8})\right), \mbox{ }\forall \sigma\in(\frac{99}{100},\frac{5}{4}].
\end{align*}
Since the RHS is frequency envelope of order $\delta$, we conclude
\begin{align*}
b_k(\sigma)&\lesssim c^{(1)}_k(\sigma).
\end{align*}
This gives (\ref{14.321}) and finishes our proof.
\end{proof}

\section{Proof of   Theorem 1.1 and Theorem 1.2}

\subsection{Global Regularity}
In order to prove $u$ is global, it suffices to verify (see Appendix B)
\begin{align}\label{jhbm}
\|\nabla u\|_{L^{\infty}_{t,x}}\lesssim 1.
\end{align}
To prove (\ref{jhbm}), it suffices to give uniform bound for  $\|u(t)\|_{{\dot H}^{1}\bigcap {\dot H}^{2+}}$. Since energy preserves, it reduces to bound  $\|u(t)\|_{ {\dot H}^{2+}}$, which is related to frequency envelopes with $\sigma=1+$.
Thus we need to  transform the intrinsic bound (\ref{14.321}) to bounds for $u$.

The following lemma  follows directly by Corollary  3.1.
\begin{Lemma}\label{huji}
Let $u\in \in\mathcal{H}_Q(T)$ solve  SMF with data $u_0$ of small energy.
For $\sigma\in[0,\frac{5}{4}]$, suppose that $\{c^{(1)}_{k}(\sigma)\}$ are frequency envelopes defined in Def. 6.1.
And assume that the differential fields $\{\phi_i\}$ associated with $u$ under the caloric gauge satisfy
\begin{align}\label{ghxxvn}
\|P_{k }\psi_i\upharpoonright_{s=0}\|_{L^{\infty}_tL^2_x}\le 2^{-\sigma k}c^{(1)}_k(\sigma), \mbox{  }\forall\mbox{  }k\in\Bbb Z.
\end{align}
Then we have
\begin{align}
2^{k}\|P_{k} u\|_{L^{\infty}_tL^2_x}\le 2^{-\sigma k}c^{(1)}_k(\sigma), \mbox{  }\forall\mbox{  }k\in\Bbb Z.
\end{align}
\end{Lemma}

Proposition \ref{4.51} shows the assumption (\ref{ghxxvn}) of  Lemma \ref{huji} holds. And thus
by applying Lemma \ref{huji}, we conclude
\begin{align}
\|u\|_{{\dot H}^{\rho}\bigcap{\dot H}^1}\lesssim C(\|u_0\|_{{\dot H}^{\rho}\bigcap{\dot H}^1}),
\end{align}
for all $\rho\in[0,\frac{9}{4}]$. Particularly $\|\nabla u\|_{L^{\infty}_{t,x}}\lesssim 1$ by Sobolev embedding.
Therefore, $u$ is global by Appendix B and the global regularity follows by local theory of \cite{Mc}.

The remained part for Theorem 1.1 is (\ref{yuhgvb90}) and (1.5). They will be proved in Section 7.4 and Section 7.5 respectively.

\subsection{Uniform Sobolev norm bounds of solutions to SMF}\label{dzfghjkl}

To get uniform Sobolev norm bounds for SMF up to $\sigma=1+\frac{K}{4}$, $K\in \Bbb Z_+$, in the heat flow iteration scheme it suffices to begin with the parabolic decay estimates
\begin{align*}
 \left\|  \partial^{L+1}_x \mathcal{G}^{(K+1)}  \right\|_{L_t^\infty L_x^2}&\lesssim  \epsilon s^{-\frac{L}{2}}, \forall L\in [0,100+K] \\
 \left\| \partial^{L+1}_x  [d\mathcal{P}]^{(K+1)}    \right\|_{L_t^\infty L_x^2}&\lesssim  \epsilon s^{-\frac{L}{2}},  \forall L\in [0,100+K].
\end{align*}
And in the SMF iteration scheme, for the $j$-th iteration we always begin with proving
\begin{align*}
2^{\frac{1}{2}k}\|P_{k} \widetilde{\mathcal{G}}^{(1)}\|_{L^4_xL^{\infty}_t}\le c^{(j)}_k(\sigma)2^{-\sigma k}[(1+2^{2k+2k_0})^{-20}1_{k+k_0\ge 0}+1_{k+k_0\le 0}2^{\delta|k+k_0|}],\mbox{ }\sigma\in[0,1+\frac{j-1}{4}],
\end{align*}
for any $s\in[2^{2k_0-1}, 2^{2k_0+1})$, $k_0,k\in\Bbb Z$.
Then repeating the argument of first time iteration for $K$ times we obtain
\begin{align*}
 \left\|  P_{k} d\mathcal{P}(e)(\upharpoonright_{s=0})\right\|_{L_t^\infty L_x^2}&\lesssim  2^{-\sigma k}c^{(K)}_{k}(\sigma)\\
\left\| P_{k} \phi_x(\upharpoonright_{s=0})\right\|_{L_t^\infty L_x^2}&\lesssim  2^{-\sigma k+k}c^{(K)}_{k}(\sigma).
\end{align*}
By bilinear estimates we then arrive at
\begin{align}
 \left\| P_{k} \partial_x u\right\|_{L_t^\infty L_x^2}\lesssim  2^{-\sigma k+k}c^{(K)}_{k}(\sigma),
\end{align}
by which the uniform Sobolev bounds follow.
Each time iteration requires $\epsilon_*$ to be smaller in our arguments.
We emphasize that the key for the succeeding SMF iterations is  to improve $\|P_{k}\mathcal{G}^{(1)}\|_{L^4_xL^{\infty}_t}$ step by step. (see e.g. Lemma 6.1)

Therefore, we have the following result:
\begin{Proposition}\label{Poibvg}
 For any $j\ge 1$, there exists a constant $\epsilon_j>0$ such that if $u_0\in \mathcal{H}_Q$ with $\|u_0\|_{\dot{H}^1}\le \epsilon_j$, then $\|u(t)\|_{{\dot H}^{j}_x}\le C(\|u_0\|_{\dot{H}^1\cap \dot{H}^j})$ for all $t\in\Bbb R$.
\end{Proposition}

Since the mass of SMF solutions doesnot conserve, the $\|u-Q\|_{L^2_x}$ norm should be handled separately. This will be proved as a corollary of  the well-posedness, see the next section.

\subsection{Well-posedness}

 In fact, the well-posedness stated in Theorem 1.2 follows closely by  \cite{Ta7} and \cite{BIKT}'s original arguments. We sketch it for reader's convenience.

[ Tataru \cite{Ta7}, Prop. 3.13] proves that: given two maps $u^0_0,u^1_0\in \mathcal{H}_Q$ with $\|u^h_0\|_{{\dot H}^1}\ll 1$, $h=0,1$, there exists a smooth one parameter family of initial data $\{u^h_0\}_{h\in[0,1]}\in C^{\infty}([0,1];\mathcal{H}_Q)$ which satisfies
\begin{align}
 \|u^h_0\|_{{\dot H}^1}&\ll 1, \mbox{ }h\in[0,1]\\
\int^{1}_0\left\| P_{k} \partial_x u^{h}\right\|_{L^2_x}dh&\approx \|u^{0}_0-u^{1}_0\|_{L^2_x}.\label{HJdgh}
\end{align}
Given $h\in[0,1]$, Theorem 1.1 yields a solution $u^{h}(t,x)\in C(\Bbb R;\mathcal{H}_Q)$ with initial data  $u^{h}_0$.
Then under the caloric gauge $\{e_{\alpha}, Je_{\alpha}\}$ for $u^{h}(t,x)$, define the differential field $\phi_h$  by
\begin{align}
\phi^{\alpha}_{h}=\langle \partial_h u^{h}, e_{\alpha}\rangle +\sqrt{-1} \partial_h u^{h}, Je_{\alpha}\rangle, \mbox{ }\alpha=1,...,n,
\end{align}
and define $\{\phi_i\}^2_{i=0}$ as before.
Since $-\sqrt{-1}\phi_t=\sum_{i=1,2}D_i\phi_i$ at $s=0$ ( because for all $h\in[0,1]$ $u^{h}(t,x)$ solves SMF), applying $D_h=\partial_h+A_h$ to the both sides gives
\begin{align*}
-\sqrt{-1}D_t\phi_{h}= \sum^{2}_{i=1}D_iD_i\phi_h+\sum^2_{i=1}\mathcal{R}(u^h(t,x))(\phi_i,\phi_h)\phi_i,\mbox{ }{\rm when }\mbox{ }s=0.
\end{align*}
which as before can be further schematically written as
\begin{align}\label{dCvghbm}
-\sqrt{-1}D_t\phi_{h}= \sum^{2}_{i=1}D_iD_i\phi_h+\sum (\phi_i\diamond\phi_h)\phi_i\mathcal{G},\mbox{ }{\rm when }\mbox{ }s=0.
\end{align}
Given $\sigma\in[0,1+\frac{j}{4})$ with $j\in\Bbb Z_+$, let $\{c_{k,(j)}(\sigma)\}$  be
\begin{align*}
c_{k,(j),h}(\sigma)=\sup_{k'\in\Bbb Z}2^{-\frac{1}{2^j}\delta|k'-k|}2^{\sigma k'+k'}\|P_{k'}u^{h}_0\|_{L^2_x}.
\end{align*}
And define   $\{c^{(j)}_{k,h}(\sigma) \}$ as Def. \ref{1jknmL}.
Then Section \ref{dzfghjkl} gives
\begin{align}\label{GDfh}
\sum^{2}_{i=1}2^{\sigma k} \|P_{k}\phi_i(s=0, h,\cdot,\cdot )\|_{G_{k}(T)}\lesssim c^{(j)}_{k,h}(\sigma),
\end{align}
and thus
\begin{align}\label{Dfh}
2^{\sigma k} \|P_{k}{\widetilde{\mathcal{G}}}(s=0, h,\cdot,\cdot )\|_{F_{k}(T)}\lesssim c^{(j)}_{k,h}(\sigma).
\end{align}
Using (\ref{GDfh}), (\ref{Dfh}) we obtain by (\ref{dCvghbm}) that
\begin{align*}
\sum_{k\in\Bbb Z} \|P_{k}\phi_h(s=0)\|^2_{G_{k}(T)}\lesssim  \|\phi_h(s=0,t=0)\|^2_{L^2_x}.
\end{align*}
Transforming this bound to $\partial_h u^{h}$ yields
\begin{align*}
  \| \partial_hu^{h}\|_{L^{\infty}_tL^2_x}\lesssim  \|\partial_hu^{h}_0\|_{L^2_x}.
\end{align*}
Then (\ref{HJdgh}) leads to
\begin{align}\label{FGVBhj}
 \| u^{1} -u^{0} \|_{L^{\infty}_tL^2_x}\lesssim \|u^1_0-u^0_0\|_{L^2_x}.
\end{align}
With (\ref{FGVBhj}) in hand, the continuity of $S_Q$ from  $\mathfrak{B}^{\sigma}_{\epsilon}$ to $C(\Bbb R; {H}^{\sigma+1}_Q)$ follows by the same arguments of [\cite{BIKT},1467-1468], if ${\epsilon}>0$ is sufficiently small depending only on $j$ thus $\sigma$.

Moreover, letting $u^1_0=Q$, $u^0_0=u_0$ in (\ref{FGVBhj}) one obtains
\begin{align}\label{FGVBhj}
 \| u -Q \|_{L^{\infty}_tL^2_x}\lesssim \|u_0-Q\|_{L^2_x}.
\end{align}
which combined with Proposition \ref{Poibvg} gives (\ref{XcDfg}).

\subsection{Asymptotic behavior}

Let us prove (\ref{yuhgvb90}).
First, we notice
\begin{align}\label{W01}
|u(t,x)-Q|=\int^{\infty}_0|\partial_sv(s,t,x)|ds'\lesssim \int^{\infty}_0|\phi_s|ds'.
\end{align}
{\bf Step 1.1.} Recall  the definition of $\{c^{(j)}_{k}(\sigma)\}$ in  Def. \ref{1jknmL}.
Applying  (3.74) with $\beta_k(\sigma)=c^{(0)}_{k}(\sigma)$, and its analogies in succeeding iterations, we get by Bernstein inequality that
\begin{align}
\|\phi_s\|_{L^4_tL^{\infty}_x}&\lesssim s^{-\frac{1}{4}}\sum_{k\in\Bbb Z}c^{(1)}_{k}(1) \label{W03} \\
\|\phi_s\|_{L^4_tL^{\infty}_x}&\lesssim s^{-\frac{3}{4}}\sum_{k\in\Bbb Z}c^{(0)}_{k}(0). \label{W02}
\end{align}
We find  by Young's inequality and triangle inequality that
\begin{align*}
 2^{\frac{1}{2^{j+4}}\delta | k|}c^{(j)}_{k}\lesssim \sup_{k'\in\Bbb Z}2^{\frac{1}{2^{j+4}}\delta|k'|}\|P_{k'}\nabla u_0\|_{L^2_x},
\end{align*}
thus there holds
\begin{align}\label{fuygb98u}
\sum_{k\in\Bbb Z} c^{(j)}_{k}\lesssim \sup_{k'\in\Bbb Z}2^{\frac{1}{2^{j+4}}\delta|k'|}\|P_{k'}\nabla u_0\|_{L^2_x}\lesssim  1,
\end{align}
since $u_0\in  \mathcal{H}_{Q}$. Then (\ref{W02}), (\ref{W03}) show
\begin{align}\label{ghjop}
 \|\phi_s\|_{L^4_tL^{\infty}_x}\lesssim   \min( s^{-\frac{1}{4}},  s^{-\frac{3}{4}}).
\end{align}
We see (\ref{ghjop}) is not enough to put $\|\phi_s\|_{L^{\infty}_x}$ in $L^1_s$, but useful for  Step 2 below.

{\bf Step 1.2.} Applying  (3.74) with $\beta_k(\sigma)=c^{(0)}_{k}(\sigma)$, $\sigma=0$, and by interpolation, we see for any $p\in (4,\infty)$, $ \tilde{p}\in (2,4)$ satisfying $\frac{1}{p}+\frac{1}{\tilde{p}}=\frac{1}{2}$, there holds
\begin{align*}
 \|\phi_s\|_{L^{p}_tL^{\tilde{p}}_x}\lesssim   2^{k}1_{k+j\ge 0}(1+2^{2j+2k})^{-4}c^{(0)}_{k}(0)+ 2^{k}1_{k+j\le 0}2^{\delta|k+j|}c^{(0)}_{k}(0),
\end{align*}
for $s\in[2^{2j-1}, 2^{2j+1})$, $k,j\in\Bbb Z$.
Then  we get by Bernstein inequality that
\begin{align}
\int^{\infty}_0 \|\phi_s\|_{L^{p}_tL^{\infty}_x}ds'&\lesssim   \sum_{k\in\Bbb Z}\sum_{j\le -k}2^{2j+k}2^{\frac{2k}{\tilde{p}}} c^{(0)}_{k}(0)2^{\delta|k+j|} + \sum_{k\in\Bbb Z}\sum_{j\ge -k}2^{2j+k}2^{\frac{2k}{\tilde{p}}}(1+2^{k+j})^{-8} c^{(0)}_{k}(0)\nonumber\\
&\lesssim   \sum_{k\in\Bbb Z} 2^{(\frac{2}{\tilde{p}}-1)k} c^{(0)}_{k}(0).\label{ghjkn090o}
\end{align}
Taking $\tilde{p}\in (2,4)$ such that $|\frac{2}{\tilde{p}}-1|\le\frac{1}{8}\delta$, one finds   (\ref{ghjkn090o}) is finite by (\ref{fuygb98u}).
Hence, there exists a $p\in(4,\infty)$ such that
\begin{align}\label{qJH00}
\int^{\infty}_0 \|\phi_s\|_{L^{p}_tL^{\infty}_x}ds'&\lesssim 1.
\end{align}

{\bf Step 1.3.}
We aim to prove
\begin{align}\label{JH00}
 \lim\limits_{t\to\infty}\int^{\infty}_0\|\phi_s(t)\|_{L^{\infty}_x}ds' =0.
\end{align}
If (\ref{JH00}) fails, then for some $\varrho>0$, there exists a time sequence $\{t^1_\nu\}$, such that $\lim\limits_{\nu\to\infty}t^1_\nu=\infty$,
\begin{align}\label{gvc7g}
\int^{\infty}_0\|\phi_s(t^{1}_{\nu})\|_{ L^{\infty}_x}ds'>\varrho, \mbox{ }\forall \nu\in \Bbb Z_+.
\end{align}
We can also assume $t^1_{\nu}\le t^1_{\nu+1}-4$ for any $\nu\in\Bbb Z_+$. Thus  by (\ref{qJH00}) there must exist a sufficiently large constant $N$ and  a time sequence $\{t^2_{\nu}\}$ such that
\begin{align}
& t^1_{\nu}-1\le t^2_{\nu}\le  t^1_{\nu}+1 \label{sg09uhjk}\\
&\int^{\infty}_0\|\phi_s(t^{2}_{\nu})\|_{ L^{\infty}_x}ds'\le \frac{1}{8}\varrho, \mbox{ }\forall \nu\ge N. \label{dg09uhjk}
\end{align}

{\bf Step 2.}
On the other hand, we have
\begin{align*}
&\partial_t \phi_s=D_t\phi_s-A_t\phi_s=D_s\phi_t-A_t\phi_s\\
&=\Delta\phi_t+\sum_{i=1,2}2A_i\partial_i\phi_t+A_iA_i\phi_t+\phi_t\partial_iA_i+\mathcal{R}(\phi_i,\phi_t)\phi_i-A_t\phi_s.
\end{align*}
Using Proposition 6.2 ((6.12), (6.13)) with $b_{k}(\sigma)$ replaced by $c^{(1)}_{k}(\sigma)$ and similar results for succeeding iterations, we see
\begin{align*}
\| \phi_t\|_{L^4_tL^{\infty}_x}&\lesssim \sum_{k\in \Bbb Z}   c^{(2)}_k( \frac{3}{2} ) \lesssim 1\\
\|\partial_x \phi_t\|_{L^4_tL^{\infty}_x}&\lesssim \sum_{k\in \Bbb Z} c^{(6)}_k(\frac{5}{2}) \lesssim 1\\
\|\partial^2_x \phi_t\|_{L^4_tL^{\infty}_x}&\lesssim \sum_{k\in \Bbb Z}  c^{(10)}_k(\frac{7}{2}) \lesssim 1,
\end{align*}
since as before one has
\begin{align*}
\sum_{k\in \Bbb Z} 2^{\frac{1}{2^{j+4}}\delta|k|} c^{(j)}_k(\sigma)  \lesssim 1.
\end{align*}
And by the same reason there hold
\begin{align*}
\| \partial^2_x\phi_t\|_{L^4_tL^{\infty}_x}\lesssim   s^{-\frac{5}{4}}\sum_{k\in\Bbb Z}c^{(1)}_{k}(1)  \lesssim  s^{-\frac{5}{4}}\\
\|\partial_x \phi_t\|_{L^4_tL^{\infty}_x}\lesssim s^{-\frac{3}{4}}\sum_{k\in\Bbb Z}c^{(1)}_{k}(1)  \lesssim  s^{-\frac{3}{4}}.
\end{align*}
Meanwhile, Lemma 3.3 shows
\begin{align*}
\| \phi_i \|_{L^{\infty}}&\lesssim  (1+s)^{-\frac{3}{4}}, \mbox{ }i=1,2\\
\|\partial^{j}_{x}A_i \|_{L^{\infty}}&\lesssim  (1+s)^{-\frac{3}{4}-\frac{1}{2}j}, \mbox{ }j=0,1.
\end{align*}
Thus we arrive at
\begin{align*}
\|\Delta\phi_t\|_{L^4_tL^{\infty}_x}+\sum_{i=1,2}\|2A_i\partial_i\phi_t+A_iA_i\phi_t+\phi_t\partial_iA_i+\mathcal{R}(\phi_i,\phi_t)\phi_i\|_{L^4_tL^{\infty}_x}\lesssim 1.
\end{align*}
For the rest $A_t\phi_s$,  by the proof of Lemma 6.1 and its analogies in succeeding iterations, we see
\begin{align*}
\|A_t\|_{L^4_tL^{\infty}_x}\lesssim s^{-\frac{1}{4}}\sum_{k\in\Bbb Z}c^{(1)}_{k}(1)\lesssim  s^{-\frac{1}{4}}\\
\|A_t\|_{L^4_tL^{\infty}_x}\lesssim s^{-\frac{3}{4}}\sum_{k\in\Bbb Z}c^{(0)}_{k}(0)\lesssim  s^{-\frac{3}{4}}.
\end{align*}
Hence, (\ref{ghjop}) implies
\begin{align*}
\int^{\infty}_0\|A_t\phi_s\|_{L^2_tL^{\infty}_x}ds'\lesssim 1.
\end{align*}
Therefore, we conclude in this step that there exists a decomposition of $\partial_t\phi_s=I_1+I_2$ such that
\begin{align}\label{d97nbjii}
\int^{\infty}_0\|I_1\|_{L^4_tL^{\infty}_x}ds'\lesssim 1, \mbox{ }\int^{\infty}_0\|I_2\|_{L^2_tL^{\infty}_x}ds'\lesssim 1.
\end{align}

{\bf Step 3.}
(\ref{d97nbjii}) and (\ref{sg09uhjk}) show
\begin{align}\label{d97jii}
\int^{\infty}_0\|\phi_s(t^1_{\nu})-\phi_s(t^2_{\nu})\|_{L^{\infty}_x}ds'\lesssim \int^{\infty}_0\left(\|I_1\|_{L^2_tL^{\infty}_x([t^2_{\nu}-1,t^2_{\nu}+1]\times \Bbb R^d)}+\|I_2\|_{L^2_tL^{\infty}_x([t^2_{\nu}-1,t^2_{\nu}+1]\times \Bbb R^d)}\right)ds'.
\end{align}
Then as $\nu\to \infty$, (\ref{d97nbjii}) further implies the RHS of (\ref{d97jii}) goes to zero. Thus (\ref{dg09uhjk}) yields
\begin{align*}
\int^{\infty}_0\|\phi_s(t^1_{\nu})\|_{L^{\infty}_x}ds'\le   \frac{1}{4}\varrho,
\end{align*}
for $\nu$ sufficiently large, which contradicts with (\ref{gvc7g}). So we have verified  (\ref{JH00}).

Similar to (\ref{JH00}) we also have
\begin{align*}
 \lim\limits_{t\to-\infty}\int^{\infty}_0\|\phi_s(t)\|_{L^{\infty}_x}ds' =0.
\end{align*}
Then (\ref{yuhgvb90}) follows by
(\ref{W01}).

\subsection{Proof of (\ref{y0})}

The proof of  (\ref{y0}) can be reduced to the following lemma.
\begin{Lemma}\label{I90bn}
Given $s>0$, there exists a function $f_s:\Bbb R^2\to \Bbb C^{n}$ belonging to  ${\dot H}^1$ such that
\begin{align*}
 \lim\limits_{t\to \infty} \|\phi_s(t)-e^{it\Delta }f_s\|_{{\dot H^1}_x} =0.
\end{align*}
Moreover, $f_s$ satisfies
\begin{align*}
\|f_s\|_{{\dot H^1}_x} \lesssim {\bf 1}_{s\in[0,1]}+{\bf 1}_{s\ge 1}s^{-\frac{3}{2}}.
\end{align*}
\end{Lemma}

Now, let's prove (\ref{y0}) by assuming Lemma \ref{I90bn}.
Recall that
\begin{align*}
\phi_i&=-\int^{\infty}_s(\partial_i\phi_s +A_i\phi_s)ds'.
\end{align*}
Since  $\|A\|_{L^{\infty}_{s,t}L^2_x}\lesssim 1$, (\ref{JH00}) shows
\begin{align*}
 \lim\limits_{t\to \infty}\|\int^{\infty}_s |A_i\phi_s|ds'\|_{L^2_x}=0.
\end{align*}
Then Lemma \ref{I90bn} yields
\begin{align}\label{Fgnnnn0}
 \lim\limits_{t\to \infty}\|\phi(0,t,x) -\nabla e^{it\Delta}f_+ \|_{L^2_x}=0,
\end{align}
where $f_+\in {\dot H}^{1}$ is defined by
\begin{align*}
f_+=-\int^{\infty}_0f_sds'.
\end{align*}

Let $\mathcal{P}$ denote the isometric embedding of $\mathcal{N}$ into $\Bbb R^{N}$. Recall that $\{e_{\alpha},Je_{\alpha}\}^{n}_{\alpha=1}$ denotes the caloric gauge.  Then the caloric gauge condition shows
\begin{align*}
\sum^{2n}_{l=1}|d\mathcal{P}(e_{l})-d\mathcal{P}(e^{\infty}_{l})|\lesssim \int^{\infty}_0|\phi_s|ds'.
\end{align*}
which combined with (\ref{JH00}) implies for $s=0$
\begin{align}\label{cvmmmb77}
 \lim\limits_{t\to \infty}\|d\mathcal{P}(e_{l})-d\mathcal{P}(e^{\infty}_{l})\|_{L^{\infty}_x}=0,\mbox{  }\forall l=1,...,2n.
\end{align}

Thus, we deduce from
\begin{align*}
\partial_j u &=\sum^{n}_{\alpha=1}\Re (\phi^{\alpha}_{j})e_{\alpha}+\Im (\phi^{\alpha}_{j})J e_{\alpha},
\end{align*}
that for $s=0$
\begin{align*}
&\|d\mathcal{P}(\nabla u)-\sum^{n}_{\alpha=1}\Re (\nabla e^{it\Delta }f_+)^{\alpha}d\mathcal{P}(e^{\infty}_{\alpha})-\sum^{n}_{\alpha=1}\Im (\nabla e^{it\Delta }f_+)^{\alpha}d\mathcal{P}(J e^{\infty}_{\alpha})\|_{L^2_x}\\
&\lesssim \| \phi-\nabla e^{it\Delta }f_+\|_{L^2_x}+\||e^{it\Delta}\nabla f_+||d{\mathcal{P}}(e)-d\mathcal{P}(e^{\infty})|\|_{L^2_x}\\
&+\| |\phi-\nabla e^{it\Delta }f_+||d{\mathcal{P}}(e)-d\mathcal{P}(e^{\infty})|\|_{L^2_x}.
\end{align*}
Therefore, (\ref{cvmmmb77}) and (\ref{Fgnnnn0}) give
\begin{align*}
& \lim\limits_{t\to \infty}\|d\mathcal{P}(\nabla u)-\sum^{n}_{\alpha=1}\Re (\nabla e^{it\Delta }f_+)^{\alpha}d\mathcal{P}(e^{\infty}_{\alpha})-\Im (\nabla e^{it\Delta }f_+)^{\alpha}d\mathcal{P}(J e^{\infty}_{\alpha})\|_{L^2_x}=0.
\end{align*}
Then, letting $\vec{v}_{\alpha}=d\mathcal{P}(e^{\infty}_{\alpha})$, $\vec{v}_{\alpha+n}=d\mathcal{P}(J e^{\infty}_{\alpha})$, we get
\begin{align}\label{y0ubb}
\lim\limits_{t\to \infty}\|  u(t) -\sum^{n}_{j=1}\Re ( e^{it\Delta}f_{+})^j\vec{v}_{j} -\sum^{n}_{j=1}\Im (e^{it\Delta}f_{+})^j\vec{v}_{j+n}\|_{{\dot H}^{1}_x}=0.
\end{align}
Thus, (\ref{y0})  follows form (\ref{y0ubb}) by setting
\begin{align*}
 h^j_{+}:=f^{j}_+\vec{v}_{j}, \mbox{  }  g^{j}_+:=f^{j}_+\vec{v}_{j+n}, \mbox{ }j=1,...,n.
\end{align*}

Now, let's prove  Lemma \ref{I90bn}. The convenient way to verify Lemma \ref{I90bn} is to introduce the so-called Schr\"odinger  map tension field $Z:=\phi_s-i\phi_t$.
Then the heat tension field $\phi_s$ satisfies for any $s\ge 0$
\begin{align}
&(i\partial_t +\Delta)\phi_s={\bf N}\\
&{\bf N}:=-(\sum^{2}_{k=1}\partial_k A_k)\phi_s-\sum^{2}_{j=1}2A_j\partial_j \phi_s-A_jA_j\phi_s+i\partial_s Z+\sum^{2}_{j=1}\mathcal{R}(\phi_j,\phi_s)\phi_j.\label{VV3}
\end{align}
And the Schr\"odinger  map tension field $Z$ satisfies the heat equation
\begin{align}\label{vv4}
\left\{
  \begin{array}{ll}
   & ( \partial_s -\Delta)Z=(\sum^{2}_{k=1}\partial_k A_k)Z+\sum^{2}_{j=1}[2A_j\partial_j Z+A_jA_jZ]\hbox{  } \\
&+\sum^{2}_{j=1}[\mathcal{R}(Z,\phi_j)\phi_j+i\mathcal{R}(\phi_j,\phi_s)\phi_j-\mathcal{R}(\phi_j,i\phi_s)\phi_j ]\hbox{  } \\
     &Z(0,t,x)=0. \hbox{ }
  \end{array}
\right.
\end{align}
To prove  Lemma \ref{I90bn}, it suffices to verify
\begin{align*}
\|\{ 2^{k}\|P_{k}{\bf N}\|_{N_{k}}\}\|_{\ell^2}\lesssim (1+s)^{-\frac{3}{2}},
\end{align*}
where ${\bf N}$ is given by $(\ref{VV3})$.
Except for the $\partial_s Z$ term in ${\bf N}$, the other terms have been handled with before. It remains to dominate  $\|P_{k}\partial_s Z\|_{N_{k}}$. In fact, one can prove a stronger result for $Z$:
\begin{align}\label{Cvssbbm}
 \|\{(1+2^{2k})2^{k}\|P_{k} Z\|_{L^{\frac{4}{3}}}\}\|_{\ell^2}\lesssim (1+s)^{-\frac{3}{2}}.
\end{align}
We see (\ref{Cvssbbm}) follows by bootstrap and (\ref{vv4}). Therefore,  Lemma \ref{I90bn} follows.

\subsection{Conclusion}

Hence, we have finished the proof of Theorem 1.1 and Theorem 1.2.

 \section{Appendix A. {Bilinear estimates} }

\begin{Lemma}\label{A1}
Let $S:\Bbb R^N\to \Bbb R$ be a smooth function in $y\in \Bbb R^N$ and $f:(-T,T)\times \Bbb R^2\to \Bbb R^N$ be smooth w.r.t. $(t,x)\in  (-T,T)\times \Bbb R^2$. Let
\begin{align}\label{FC}
\mu_k=\sum_{|k_1-k|\le 20}2^{k_1}\|P_{k_1}f\|_{L^{\infty}_{L^2_x}}.
\end{align}
Assume that $\|f\|_{L^{\infty}_x}\lesssim 1$ and $\sup_{k\in \Bbb Z}\mu_k\le 1$. Then
\begin{align}
&2^{k}\|P_{k }S(f)(\partial_a f\partial_b f)\|_{L^{\infty}_tL^2_x}\nonumber\\
&\lesssim 2^{k} \sum_{k_1\le k}\mu_{k_1}2^{k_1}\mu_k+\sum_{k_2\ge k}2^{2k}\mu^2_{k_2}\nonumber\\
&+ a_{k}\left(\sum_{k_1\le k}2^{k_1}\mu_{k_1}\right)^2+\sum_{k_2\ge k}2^{2k}2^{-k_2}a_{k_2}\mu_{k_2}\sum_{k_1\le k_2}2^{k_1}\mu_{k_1}.\label{c6FC}
\end{align}
where $\{a_k\}$ denotes
\begin{align}\label{FCbv}
a_k:=\|\nabla P_k(S(f))\|_{L^{\infty}_tL^2_x}.
\end{align}
\end{Lemma}
\begin{proof}
The same proof of [\cite{BIKT}, Lemma 8.2] shows
\begin{align*}
&2^{k}\|P_{k }S(f)(\partial_a f\partial_b f)\|_{L^{\infty}_tL^2_x}\\
&\lesssim 2^{2k}\sum_{k_1\le k}\mu_{k_1}2^{k}\mu_k+\sum_{k_2\ge k}2^{-2(k_2-k)}2^{2k_2}\mu^2_{k_2}\\
&+ a_{k}(\sum_{k_1\le k}2^{k_1}\mu_{k_1})^2+\sum_{k_2\ge k}2^{2k}2^{-2k_2}2^{k_2}a_{k_2}\mu_{k_2}\sum_{k_1\le k_2}2^{k_1}\mu_{k_1}.
\end{align*}
The only difference is that we use
\begin{align*}
\|P_k(S(f))\|_{L^2_x}\le 2^{-k}\|\nabla P_k(S( f))\|_{L^2_x}
\end{align*}
when $S(f)$ lies in the high frequency w.r.t. $\partial_a f\partial_b f$, and the trivial bound
\begin{align*}
\|P_k(S(f))\|_{L^{\infty}_x}\lesssim 1
\end{align*}
when  $S(f)$ lies in the relatively low frequency.
\end{proof}

Denote  $H^{\infty,\infty}(T)$ the set of functions $f$ defined in $(t,x)\in [-T,T]\times \Bbb R^2$ satisfying $\partial^{b_1}_t \partial^{b_2}_xf\in L^2([-T,T]\times \Bbb R^2)$  for any $b_1,b_2\in \Bbb N$.

\begin{Lemma}[\cite{BIKT},Lemma 5.1]\label{FCq}
Given $\mathcal{L}\in \Bbb Z_+$, $\omega\in [0,\frac{1}{2}]$, $T\in (0,2^{2\mathcal{L}}]$. Suppose that $f,g\in H^{\infty,\infty}(T)$, let
\begin{align*}
\alpha_k:=\sum_{|k-k'|\le 20}\|f_{k'}\|_{S^{\omega}_{k'}(T)\cap F_{k'}(T)}, \mbox{ }\beta_k:=\sum_{|k-k'|\le 20}\|g_{k'}\|_{S^{0}_{k'}(T)},
\end{align*}
If $|k_1-k_2|\le 8$, then
\begin{align}\label{FC}
\|P_{k}(P_{k_1}fP_{k_2}g)\|_{F_k(T)\bigcap S^{\frac{1}{2}}_k(T)}\lesssim 2^{\frac{kd}{2}}2^{(k_2-k)(\frac{2d}{d+2}-\omega)}\alpha_{k_1}\beta_{k_2}.
\end{align}
If $|k-k_1|\le 4$, then
\begin{align}\label{PDE}
\|P_k(gP_{k_1}f )\|_{F_k(T)\bigcap S^{\frac{1}{2}}_k(T)}\lesssim \|g\|_{L^{\infty}}\alpha_{k_1}.
\end{align}
\end{Lemma}

\begin{Lemma}[\cite{BIKT},Lemma 5.4]\label{90}
Given $\mathcal{L}\in \Bbb Z_+$, $\omega\in [0,\frac{1}{2}]$, $T\in (0,2^{2\mathcal{L}}]$. Then for $f,g\in H^{\infty,\infty}(T)$
\begin{align}\label{Lemma5.4}
\|P_k(fg)\|_{L^{4}_{t,x}}&\lesssim \sum_{l\le k}2^{l}(\mathbf{a}_{l}\mathbf{b}_{k}+2^{\frac{1}{2}(k-l)}\mathbf{a}_{k}\mathbf{b}_{l})+2^{{k}}\sum_{l\ge k}2^{-\omega(l-k)}\mathbf{a}_{l}\mathbf{b}_{l}.
\end{align}
where we denote
\begin{align}
\mathbf{a}_k:=\sum_{|l-k|\le 20}\|P_{k }f\|_{S^{\omega}_{l}(T)}, \mbox{  }\mathbf{b}_k:=\sum_{|l-k|\le 20}\|P_{k }g\|_{L^4_{t,x}(T)}.
\end{align}

\end{Lemma}

\begin{Lemma}[\cite{BIKT},Lemma 5.4]\label{90}
Given $\mathcal{L}\in \Bbb Z_+$, $\omega\in [0,\frac{1}{2}]$, $T\in (0,2^{2\mathcal{L}}]$. Suppose that $f,g\in H^{\infty,\infty}(T)$, $P_{k}f\in S^{\omega}_k(T)$, $P_{k}g\in L^{4}_{t,x}$ for all $k\in\Bbb Z$. Let
\begin{align}
\mu_k:=\sum_{|l-k|\le 20}\|P_{k }f\|_{S^{\omega}_{l}}(T), \mbox{  }\nu_k:=\sum_{|l-k|\le 20}\|P_{k }g\|_{L^4_{t,x}(T)}.
\end{align}
If $|k_2-k|\le 4,$ $k_1\le k-4$, then
\begin{align}
\|P_{k }(f_{k_1}g_{k_2})\|_{L^{4}_{t,x}}\lesssim 2^{k_1}\mu_{k_2}\nu_{k}.
\end{align}
If $|k_1-k|\le 4,$ $k_2\le k-4$, then
\begin{align}
\|P_{k }(f_{k_1}g_{k_2})\|_{L^{4}_{t,x}}\lesssim 2^{k_2}2^{\frac{1}{2}(k-k_2)}\mu_{k}\nu_{k_2}.
\end{align}
If $|k_1-k_2|\le 8,$ $k_1,k_2\ge k-4$, then
\begin{align}
\|P_{k }(f_{k_1}g_{k_2})\|_{L^{4}_{t,x}}\lesssim 2^{k(1+\omega)}2^{-\omega k_2}\mu_{k_2}\nu_{k_2}.
\end{align}
\end{Lemma}

\begin{Lemma}[\cite{BIKT}, Lemma 6.3]
\begin{itemize}
  \item If $|l-k|\le 80$ and $f\in F_{l}(T)$, then
\begin{align}\label{6a}
\|P_{k }(gf)\|_{N_k(T)}&\lesssim \|g\|_{L^2_tL^2_x}\|f\|_{F_{l}(T)}.
 \end{align}
 \item If $l\le k-80$ and $f\in F_{l}(T)$, then
\begin{align}\label{6b}
\|P_{k }(gf)\|_{N_k(T)}&\lesssim 2^{\frac{l-k}{2}}\|g\|_{L^2_tL^2_x}\|f\|_{F_{l}(T)}.
 \end{align}
  \item If $k\le l-80$ and $f\in G_{l}(T)$, then
\begin{align}\label{6c}
\|P_{k }(gf)\|_{N_k(T)}&\lesssim 2^{\frac{k-l}{6}}\|g\|_{L^2_tL^2_x}\|f\|_{G_{l}(T)}.
 \end{align}
\end{itemize}

\end{Lemma}

\begin{Lemma}[\cite{BIKT}, Lemma 6.5]
\begin{itemize}
  \item If $k\le l$ and $f\in F_{k}(T)$,$g\in F_{l}(T)$ then
\begin{align}
\|fg\|_{L^{2}_{t,x}}&\lesssim \|f\|_{F_{k}(T)}\|g\|_{F_{l}(T)}.\label{6.19}
 \end{align}
 \item If $k\le l$ and $f\in F_{k}(T)$,$g\in G_{l}(T)$ then
\begin{align}
\|fg\|_{L^{2}_{t,x}}&\lesssim 2^{\frac{k-l}{2}}\|f\|_{F_{k}(T)}\|g\|_{G_{l}(T)}.\label{6.20}
 \end{align}
\end{itemize}
\end{Lemma}

  \section {Appendix B. Proof of Remained Claims}

It seems that
the following blow-up criterion  was not explicitly written down in literature of SMF. This result is well-known in energy critical heat flows. For completeness, we give a proof.
\begin{Proposition}
Suppose that $u_0\in H^{L}_Q$ with $L\ge 4$ is the initial data to SMF. If in the time interval $[-T,T]$,  the SMF solution $u$ satisfies
\begin{align}\label{hbk}
\|u(t)\|_{L^{\infty}_{t,x}(T)}\le B<\infty,
\end{align}
then $u$ has the bound
\begin{align}
\|u(t)\|_{L^{\infty}_tH^{L}_x }\le C(B,T,\|u_0\|_{H^{L}_x})<\infty.
\end{align}
As a corollary, if  (\ref{hbk}) holds then $u$ can be extended beyond $[-T,T]$ to $C([-T-\rho,T+\rho];H^{L}_Q)$ for some $\rho>0$.
\end{Proposition}
\begin{proof}
 Recall the tension field  $\tau(u)=\sum^{2}_{j=1}\nabla_j\partial_j u$.
By integration by parts,
 \begin{align}
\int_{\Bbb R^2}\langle \tau(u), \tau(u)\rangle dx&=\int_{\Bbb R^2}\sum^{2}_{j,k=1}\langle \nabla_j\partial_j u,\nabla_k\partial_k u\rangle dx\nonumber\\
&=\int_{\Bbb R^2}\langle \nabla_k\nabla_j\partial_ku, \nabla_k\nabla_j\partial_ku\rangle
+\int_{\Bbb R^2}O(|du|^4)dx.\label{pojn}
\end{align}
Since $u$ solves $SMF$, by integration by parts, we get
 \begin{align*}
&\frac{d}{dt}\int_{\Bbb R^2}\langle \tau(u), \tau(u)\rangle dx=2\sum^2_{j=1}\int_{\Bbb R^2} \langle \nabla_j\partial_j \partial_t u,\tau(u)\rangle dx+\int_{\Bbb R^2}O(|du|^2|\partial_tu||\tau(u)|)dx\\
&=2\sum^2_{j=1}\int_{\Bbb R^2}\langle  \nabla_jJ\tau(u), \nabla_j\tau(u)\rangle
+\int_{\Bbb R^2}O(|du|^2|\partial_tu||\tau(u)|)dx.
\end{align*}
Since $J$ commutes with $\nabla_j$, $\langle JX,X\rangle =0 $,  we then arrive at
 \begin{align*}
\frac{d}{dt}\|\tau(u)\|^2_{L^2_x}\lesssim \|du\|^2_{L^{\infty}_{t,x}}\|\tau(u)\|^2_{L^2_x}.
\end{align*}
Gronwall inequality and (\ref{hbk}) show
 \begin{align*}
\|\tau(u)\|_{L^2_x}\lesssim e^{Bt}\|\tau(u_0)\|_{L^2_x}.
\end{align*}
Using the energy bound
 \begin{align*}
\|\nabla u\|_{L^{\infty}_tL^2_x}\lesssim  \|\nabla u_0\|_{L^2_x}
\end{align*}
and (\ref{pojn}) give
 \begin{align}\label{Hxbn}
\| u(t)\|_{\mathcal{W}^{2,2}}\lesssim  B \|\nabla u_0 \|_{L^2_x}+ e^{Bt}\|\tau(u_0)\|_{L^2_x}.
\end{align}
By integration by parts,
 \begin{align*}
&\int_{\Bbb R^2}\langle \nabla_i\tau(u), \nabla_i\tau(u)\rangle dx=\int_{\Bbb R^2}\sum^{2}_{j,k=1}\langle \nabla_i\nabla_j\partial_j u, \nabla_i\nabla_k\partial_k u\rangle dx=\int_{\Bbb R^2}\langle \nabla_i\nabla_j\partial_ku, \nabla_i\nabla_j\partial_ku\rangle\\
&+\int_{\Bbb R^2}O(|du|^3|\nabla^2du| +|\nabla u|^2|\nabla du|^2+|\nabla du||du|^2)dx.
\end{align*}
Thus we have
 \begin{align}
\| \nabla^2du(t)\|^2_{L^2_x}&\lesssim \|\nabla\tau(u)\|^2_{L^2_x}+\|du\|^6_{L^6_x}+\|du\|^2_{L^{\infty}_x}\|\nabla du\|^2_{L^2_x}+\|du\|^2_{L^{4}_x}\|\nabla du\|_{L^2_x}\nonumber\\
&\lesssim \|\nabla\tau(u)\|^2_{L^2_x}+C(B,t,\|u_0\|_{\mathcal{W}^{2,2}})\label{Hdxbn}
\end{align}
And applying integration by parts furthermore gives
 \begin{align*}
&\frac{1}{2}\frac{d}{dt}\|\nabla \tau (u)\|^2_{L^2_x}=\sum_{i,j}\langle \nabla_i\nabla_j\partial_ju,\nabla_t\nabla_i\nabla_j\partial_ju\rangle\\
&=\sum_{i,j}\langle \nabla_i\tau(u), \nabla_i\nabla_j\nabla_j\partial_tu\rangle +\int_{\Bbb R^2}|\nabla\tau(u)||\nabla\partial_tu||du|^2dx\\
&+\int_{\Bbb R^2}|\nabla^2du| \nabla du||\partial_t u||du|dx+  \int_{\Bbb R^2}|du|^3 |\partial_tu||\nabla^2du| dx\\
&\lesssim -\langle\sum_{i}\nabla_i \nabla_i\tau(u), J\sum_{j} \nabla_j\nabla_j\tau(u)\rangle+B^2\|\nabla\tau(u)\|^2_{L^2_x}+B\|\nabla\tau(u)\|_{L^2_x}\|\nabla du\|^2_{L^4_x}\\
&+B^3\|\nabla\tau(u)\|_{L^2_x}\|\tau(u)\|_{L^2_x}\\
&\lesssim B^2\|\nabla\tau(u)\|^2_{L^2_x}+B\|\nabla\tau(u)\|_{L^2_x}\|\nabla^2 du\|_{L^2_x}\|\nabla du\|_{L^2_x}+B^3\|\nabla\tau(u)\|_{L^2_x}\|\tau(u)\|_{L^2_x}.
\end{align*}
Hence, denoting $F(t)=\|\nabla\tau(u)\|_{L^2_x}$, (\ref{Hxbn}) and (\ref{Hdxbn}) show
 \begin{align*}
&\frac{1}{2}\frac{d}{dt}F^2(t)\lesssim C_1(B,T) F(t)[F(t)+ C_2(B,T)],
\end{align*}
where $C_1(B,T)$ and $C_2(B,T)$ are smooth functions of $B,T$.
So the Sobolev norm of $u$ has a uniform bound in $[-T,T]$ up to order three. This with the classical local existence theory (see \cite{DW} or \cite{Mc})  implies $u$ can be extended to $[-\rho-T,T+\rho]$ for some $\rho>0$. And the bounds for the higher order Sobolev norms follow by Theorem 3.3 of \cite{Mc} or induction. Then by Sobolev embedding $u$ is smooth in $[-\rho-T,T+\rho]$ if $u_0\in\mathcal{H}_{Q}$.
\end{proof}

\end{document}